\documentclass[11pt, twoside, leqno, a4]{amsart}  
\usepackage{lipsum}
\usepackage{amsfonts}
\usepackage{graphicx}
\usepackage{epstopdf}
\usepackage{algorithmic}
\usepackage{calligra}
\usepackage[dvipsnames]{xcolor}
\usepackage{amsfonts,amsmath,amsthm,amssymb}
\usepackage{enumitem}
\usepackage{mathtools}
\usepackage{hyperref}
\usepackage{autonum}
\usepackage{hhline}
\usepackage{array}
\usepackage{diagbox}
\usepackage{tcolorbox}
\usepackage{mdframed}
\usepackage{multicol}
\usepackage{graphicx}
\usepackage{subcaption}
\usepackage{moreverb}
\usepackage{bbm}
\usepackage[margin=1.4in]{geometry}
\usepackage{todonotes}
\usepackage{scalerel,amssymb}
\allowdisplaybreaks
\usepackage{mathrsfs}  
\usepackage{lineno}
\usepackage{todonotes}
\usepackage{tikz}
\usepackage{pgfplots}
\pgfplotsset{compat=1.7}
\usetikzlibrary{arrows.meta}
\usepackage[numbers,sort&compress]{natbib}
\usepackage{appendix}
\DeclareMathAlphabet{\mathdutchcal}{U}{dutchcal}{m}{n}
\SetMathAlphabet{\mathdutchcal}{bold}{U}{dutchcal}{b}{n}
\DeclareMathAlphabet{\mathdutchbcal}{U}{dutchcal}{b}{n}

\newcommand{\dt}{\, \textup{d} t}
\newcommand{\ds}{\, \textup{d} r }
\newcommand{\dx}{\, \textup{d} x}

\newcommand{\R}{\mathbb{R}} 
\newcounter{rownumber}

\newcommand{\vecc}[1]{\boldsymbol{#1}}

\usepackage{scalerel,stackengine}
\stackMath
\newcommand\reallywidehat[1]{%
\savestack{\tmpbox}{\stretchto{%
  \scaleto{%
    \scalerel*[\widthof{\ensuremath{#1}}]{\kern-.6pt\bigwedge\kern-.6pt}%
    {\rule[-\textheight/2]{1ex}{\textheight}}
  }{\textheight}%
}{0.5ex}}%
\stackon[1pt]{#1}{\tmpbox}%
}
\newtheorem{theorem}{Theorem}
\newtheorem{lemma}{Lemma}
\newtheorem{proposition}{Proposition}
\newtheorem*{assumption*}{Assumptions on the memory kernel}

\newtheorem{remark}{Remark}
\numberwithin{lemma}{section}
\numberwithin{proposition}{section}
\numberwithin{theorem}{section}
\numberwithin{equation}{section}
\makeatletter
\newcommand{\leqnomode}{\tagsleft@true}
\newcommand{\reqnomode}{\tagsleft@false}
\makeatother

\title[Asymptotic behavior of     nonlinear sound waves]{Asymptotic behavior of nonlinear sound waves \\[0.5mm] in inviscid media with \\[0.5mm] thermal and molecular   relaxation} 


\author[V. Nikoli\'c \& B. Said-Houari ]{\bfseries Vanja Nikoli\'c$^1$ and Belkacem Said-Houari$^2$}

\address{ 
Department of Mathematics \\ 
Radboud University   \\ 
Heyendaalseweg 135,
6525 AJ Nijmegen, The Netherlands}
\email{vanja.nikolic@ru.nl} 

\address{  
	Department of Mathematics\\ College of Sciences\\ University of
Sharjah, P. O. Box: 27272 \\ Sharjah, United Arab Emirates}
\email{bhouari@sharjah.ac.ae}  
\thanks{$^*$Corresponding author: Belkacem Said-Houari, \href{mailto:bhouari@sharjah.ac.ae}{bhouari@sharjah.ac.ae}}  
\begin{document}  
\vspace*{-4mm}       
\maketitle             
\vspace*{-4mm}  
\begin{center}
	{\footnotesize
		$^1$Department of Mathematics, Radboud University, The Netherlands \\
		$^2$Department of Mathematics, University of Sharjah, United Arab Emirates    
	}
\end{center}	
\vspace*{3mm}
\begin{abstract}      
Ultrasonic propagation through media with thermal and molecular relaxation can be modeled by third-order in time nonlinear wave-like equations with memory. This paper investigates the asymptotic behavior of a Cauchy problem for such a model, the nonlocal Jordan--Moore--Gibson--Thompson equation, in the so-called critical case, which      corresponds to propagation in inviscid fluids. The memory has an exponentially fading character and type I, meaning that involves only the acoustic velocity   potential. A   major challenge in the global analysis is that the linearized equation's decay estimates are of regularity-loss type. As a result, the classical energy methods fail to work for the nonlinear problem. To overcome this difficulty, we construct appropriate time-weighted norms, where weights can have negative exponents. These problem-tailored norms create artificial damping terms that help control the nonlinearity and the loss of derivatives, and ultimately allow us to discover the model's asymptotic behavior.                          
\end{abstract}              
\vspace*{2mm}       
                
\section{Introduction}
Ultrasonic waves traveling through fluids with impurities are known to be influenced by viscoelastic relaxation effects. An example of such a medium is water with micro-bubbles, commonly used as a contrast agent in ultrasonic imaging~\cite{dijkmans2004microbubbles} as well as in improving the speed and efficacy of focused ultrasound treatments~\cite{stride2010cavitation}.\\
\indent The goal of the present work is to study asymptotic behavior of such ultrasonic waves as modeled by the Jordan--Moore--Gibson--Thompson (JMGT) equation with memory:
\begin{equation}  \label{Main_Equation}
\begin{aligned}
&\tau \psi_{ttt}+\psi_{tt}-c^2\Delta u-b \Delta \psi_t+%
\displaystyle\int_{0}^{t}g(r)\Delta \psi(t-r)\ds
=\left( k \psi_{t}^{2}+|\nabla
\psi|^{2}\right)_t,
\end{aligned}
\end{equation}
in the so-called critical case when the medium parameters satisfy the relation
\[b= \tau c^2.\]
This wave-like equation models nonlinear sound propagation through inviscid media with thermal and molecular relaxation. The relaxation mechanisms are responsible for the third-order propagation and the viscoelastic term in the equation, whereas the negligible sound diffusivity leads to the critical condition $b= \tau c^2$. This kind of memory acting only on the solution of the equation (and not on its time derivatives) is often referred to in the literature as the memory of type I; cf.~\cite{lasiecka2017global, dell2016moore}. The reader is referred to Section~\ref{Section:Modeling} below for a more detailed insight into nonlinear acoustic modeling. \\
\indent To our best knowledge, the present work is the first treating the problem \eqref{Main_Equation} in the critical case. A major difficulty in treating the critical case is that the linearized equation's decay estimates are of \emph{regularity-loss} type. It is well-known that such a loss of regularity going from the initial data to the solution presents a big obstacle in proving nonlinear stability. The classical energy method fails. To solve this problem, we construct appropriate time-weighted norms:
\begin{equation}
\|\vecc{\Psi}\|_{\mathbbm{E}, t}^2=\sum_{i=0}^{[\frac{s-1}{2}]}\sup_{0\leq \sigma\leq t} (1+\sigma)^{i-1/2}||| \nabla^i \vecc{\Psi} (\sigma)|||_{H^{s-2i}}^2,  
\end{equation}
where $\vecc{\Psi}$ will denote the solution vector after rewriting our problem as a first-order system and $s\geq [\tfrac{3n}{2}]+5$; we refer to Section~\ref{Sec:Preliminaries} below for details. These tailored norms have a weight with a negative exponent, which helps introduce artificial damping to the system. This damping, in turn, allows us to control the nonlinearity and handle the loss of derivatives. \\
\indent The main result of this work is contained in Theorem~\ref{Main_Theorem} below and concerns global existence and asymptotic decay of solutions in $\R^n$, where $n \geq 3$, for smooth and small initial data. The decay estimates hold for a solution with a lower Sobolev regularity than that assumed for the initial data, which is to be expected in the presence of a loss of regularity; see, for instance, \cite{IK08,   Duan_Ruan_Zhu_2012, Racke_Said_2012_1}. The damping introduced by the memory term plays a key role in stabilizing the solution in the critical case. Without memory, the linearized problem is unstable. Indeed, it has been proven in \cite{PellSaid_2019} by relying on the Routh--Hurwitz theorem that the real parts of the eigenvalues associated with the linearized system are negative if and only if $b >\tau c^2$. \\
\indent We organize the rest of the paper as follows. In Section~\ref{Section:Modeling}, we discuss the modeling and related work on analyzing third-order acoustic equations.  In Section~\ref{Sec:Preliminaries}, we recall the problem's local-well posedness in the so-called history framework and then present our main result on the global well-posedness and asymptotic behavior of solutions for small and smooth data. Section~\ref{Sec:EstLin} deals with the decay estimates for the linear version of the equation, which we will rely on in the decay analysis of the nonlinear problem. We present the proof of the main result in Section~\ref{Sec:ProofMain}, up to two crucial energy bounds. Their proof is contained in Sections~\ref{Sec:EnergyAnalysis} and \ref{Sec:ArtificialDamping}, based on carefully designed time-weighted energies. We conclude the paper with a discussion and an outlook on open problems. Auxiliary technical results and proofs are collected in Appendices \ref{AppendixA}--\ref{AppendixC}.
\section{Acoustic waves in media with thermal and molecular relaxation}\label{Section:Modeling}
     \label{Sec:ProblemSetting}
     The  Jordan--Moore--Gibson--Thompson (JMGT) equation 
     \begin{equation}\label{MGT}
        \tau \psi_{ttt}+\alpha \psi_{tt}-c^{2}\Delta \psi-b\Delta \psi_{t}= \dfrac{\partial}{\partial t}\left(k(\psi_{t})^{2}+|\nabla \psi|^{2}\right)
        \end{equation}
arises in acoustics as a model of   nonlinear sound propagation through thermally relaxing fluids and gases; see~\cite{jordan2008nonlinear} for its derivation, which builds upon~\cite{moore1960propagation, Stokes,thompson}. Here $\psi=\psi(x,t)$ denotes the acoustic velocity potential.  The constant $c>0$ is the speed of sound in a given fluid and \[b=\delta +\tau c^2,\] where $\delta>0$ represents the so-called sound diffusivity and $\tau>0$ is the thermal relaxation time. Furthermore, $k=\beta_{\textup{a}}/c^2$, where $\beta_{\textup{a}}$ is the coefficient of nonlinearity. The coefficient $\alpha>0$ accounts for the losses due to friction.\\
\indent This third-order equation and its linear version, often called the Moore--Gibson--Thompson (MGT) equation, have received a lot of attention recently; we provide here only a selection of references~\cite{Kaltenbacher_2011, kaltenbacher2012well, bucci2019regularity, marchand2012abstract} with analysis in smooth bounded domains and~\cite{chen2020cauchy, Racke_Said_2019, PellSaid_2019_1} with analysis in $\R^n$. We also point out the recent studies on the controllability of the MGT equations in \cite{bucci2019feedback, Lizama_Zamorano_2019} and the vanishing thermal relaxation dynamics in~\cite{bongarti2018singular, KaltenbacherNikolic}. \\
\indent It is known that  
the behavior of the linear model hinges on the so-called critical parameter 
\begin{equation}\label{chi}
\chi=\alpha-\frac{c^{2}\tau}{b}.
\end{equation}
In the subcritical case when $\chi>0,$ the solution is exponentially stable on smooth bounded domains, whereas the energy is preserved in the critical case $\chi=0$; see~\cite{Kaltenbacher_2011} for the revealing analysis.  The MGT equation in $\R^n$ with a power-source nonlinearity $|u|^p$ has also been considered recently in~\cite{Chen_Palmieri_1}, with blow-up proven in the critical case $\chi=0$. \\
\indent  In media that exhibits molecular relaxation, such as water with micro-bubbles, chemically reacting fluids, or a mixture of gases, the viscoelastic effects influence the wave propagation; see, for example, the books~\cite{naugolnykh2000nonlinear, prieur2011nonlinear} for a deeper insight into this process. In such cases, the resulting wave equations have memory terms that correspond to particular relaxation mechanisms. This motivates our present study of the following nonlocal JMGT equation:
\begin{equation}  \label{Main_Equation}
\begin{aligned}
&\tau \psi_{ttt}+\alpha \psi_{tt}-c^2\Delta \psi-b \Delta \psi_t+%
\displaystyle\int_{0}^{t}g(r)\Delta \psi(t-t)\ds
=\left( k \psi_{t}^{2}+|\nabla
\psi|^{2}\right)_t,
\end{aligned}
\end{equation}
where the function $g$ is the memory kernel related to a particular relaxation mechanism. It typically has an exponentially fading character, meaning that the more recent inputs have a bigger influence on the acoustic velocity potential field compared to the older ones. We refer to~\cite{dell2016moore, lasiecka2017global,Lasiecka_Wang_2,Lasiecka_Wang_1, alves2018moore,Liuetal._2019} for a selection of recent theoretical results on this model and its linear version.\\
\indent We single out the contribution of~\cite{dell2016moore}, which shows that, in the for us relevant critical case $\chi=0$, the linearized problem associated to \eqref{Main_Equation} (with a general operator $\mathcal{A}$ instead of $-\Delta$) in bounded smooth domains is exponentially stable if and only if $\mathcal{A}$ is a bounded operator. We note that the decay rate given in \cite[Theorem 7.1]{dell2016moore} enjoys the property of the regularity loss. {In addition, this decay rate is polynomial of the form $1/t$, whereas the decay is exponential in the subcritical case; see~\cite[Theorem 1.4]{Lasiecka_Wang_1}. It has been confirmed in~\cite{Bounadja_Said_2019} that in $\R^n$, where $n \geq 1$, the linear problem in the critical case also has the decay property of regularity-loss type. \\ 
\indent Our goal in this paper is to investigate the asymptotic behavior of solutions to the nonlinear equation \eqref{Main_Equation} in the critical case $\chi=0$. We point out the results of~\cite{lasiecka2017global}, where the nonlocal JMGT equation without the quadratic gradient nonlinearity in the critical case is stabilized by considering a memory term that acts on both $\Delta \psi$ and $\Delta \psi_t$:
\begin{align} \label{mixed_memory}
\int_0^\infty g(r)\Delta (\psi+\tfrac{c^2}{b}\psi_t)(t-r)\, \ds.
\end{align} 
A major difficulty in the present analysis where the memory acts only on $\Delta \psi $ lies in the linearization's regularity loss. This loss in regularity going from the initial data to the solution prevents the use of standard energy methods in analyzing the corresponding nonlinear problem. To prove nonlinear stability with the memory of type I, we thus intend to construct problem-tailored time-weighted energies. The general ideas along these lines can be found, for example, in~\cite{IK08, Racke_Said_2012_1, Duan_Ruan_Zhu_2012}. Using weighted energy arguments, we will show that the solution is global and decays as in the subcritical case, provided that the initial data is very smooth and sufficiently small. Additionally, compared to the analysis in~\cite{lasiecka2017global}, the equation considered here has a quadratic gradient nonlinearity, which requires us to devise higher-order energy bounds and employ suitable commutator estimates. 
\section{Theoretical preliminaries and the main result}\label{Sec:Preliminaries}
For future use, we discuss in this section several useful background results and set the notation. In media with molecular relaxation caused by the presence of ``impurities" in the fluid, the memory kernel typically has the form
\begin{equation}
g(r)=m c^2\exp{(-r/\tau)},
\end{equation}
where $m$ is the relaxation   parameter; see~\cite[\S 1]{naugolnykh2000nonlinear} and \cite[\S 1]{lasiecka2017global}. Throughout the paper, we thus make the following assumptions on the relaxation kernel; cf.~\cite[\S 1]{dell2016moore}. 
\begin{assumption*} The memory kernel is assumed to satisfy the following conditions:
	\begin{enumerate}
		\item[(G1)] \label{itm:first}  $g\in W^{1,1}(\R^+)$ and $g'$ is almost continuous on $\R^+=(0, +\infty)$. \vspace{0.1 cm}
		\item[(G2)] $g(r) \geq 0$ for all $s>0$ and 
		\reqnomode
		\begin{align} \label{def_cg}
		\ c^2_g:=c^2-\displaystyle\int_{0}^{\infty}g(r)\ds>0. 	\vspace{0.1 cm}
		\end{align}
		\item[(G3)] There exists $\zeta>0$, such that the function $g$ satisfies the 
		differential inequality given by
		\begin{equation}
		g^\prime(r)\leq -\zeta g(r)
		\end{equation}
		for every $r\in (0,\infty)$. 	\vspace{0.1 cm}
		\item[(G4)] It holds that $g^{\prime\prime}\geq0$ almost everywhere.\vspace{0.2 cm}
	\end{enumerate}
\end{assumption*}
We wish to point out recent efforts in the works on linear wave equations with memory to relax the above assumptions on the memory kernel. In particular, the analysis of abstract linear viscoelastic equations in~\cite{conti2020general} removes the restrictive assumption on the differential inequality that the memory kernel should satisfy. \\
\indent We choose to adopt the so-called history framework of Dafermos~\cite{dafermos1970asymptotic}, following previous research on the wave equations with  memory in~\cite{dell2016moore, grasselli2002uniform}. This is achieved by introducing the auxiliary history variable $\eta=\eta(x, t, r)$ for $t \geq 0$, defined as
\begin{equation}  \label{def of eta}
\eta(x, t, r)= \begin{cases}
\psi(x,t)-\psi(x,t-r), \quad  &0<r\leq t, \\
\psi(x,t), \quad &r>t.
\end{cases}
\end{equation}   
Equation~\eqref{Main_Equation} is considered with the following initial data: 
\begin{equation}  \label{initial data}
\psi(x,0)=\psi_{0}(x),\qquad \psi_{t}(x,0)=\psi_{1}(x), \qquad \psi_{tt}(x,0)=\psi_{2}(x),
\end{equation}
whose regularity will be specified in theorems below. We can then rewrite our problem as
\begin{equation}  \label{eta syst}
\begin{cases}
\tau \psi_{ttt}+\alpha \psi_{tt}-b \Delta \psi_{t}-c^2_g\Delta \psi-%
\displaystyle\int_{0}^{\infty}g(r)\Delta \eta(r)\ds 
= 2k\psi_{t}\psi_{tt}+2 \nabla \psi \cdot \nabla \psi_t,  \\[2mm] 
\eta_{t}(x,s)+\eta_{r}(x,r)=\psi_{t}(x,t).%
\end{cases}%
\end{equation}
Provided that 
\begin{equation}
\eta(t=0)=\psi_0(x), \quad \eta(r=0)=0,
\end{equation} 
we can recover \eqref{def of eta} from the second equation in \eqref{eta syst}. We refer to~\cite[\S 3]{grasselli2002uniform} for a detailed discussion on this additional equation. Going forward, we set \[\alpha=1\] without the loss of generality. The critical condition  then reads as   
\[b =\tau c^2.\] We always assume that $\tau c^2 > \tau c^2_g$, which is equivalent to $\int_{0}^{\infty}g(r)\ds>0$.
\subsection{Notation} Throughout the paper, the constant $C$ denotes a generic positive constant that does not depend on time, and can have different values on different occasions. We often write $x \lesssim y$ instead of $x \leq C y$. \\
\subsection{The main result}  
To state our main result on the global existence and asymptotic decay, we first rewrite equation \eqref{eta syst} as a first-order in time system:
\begin{subequations} \label{Main_whole}
\begin{equation}\label{Main_System}  
\begin{aligned}
\begin{cases}
\psi_{t}=v, \\ 
v_{t}=w, \\ 
\tau w_{t}=- w+c^2_g\Delta \psi+b\Delta v + \displaystyle%
\int_{0}^{\infty}g(r)\Delta\eta(r)\ds+2k (vw+\nabla \psi \cdot \nabla v),
\\ 
\eta_{t}=v-\eta_{r},%
\end{cases}%
\end{aligned}
\end{equation}
with the initial data 
\begin{equation}  \label{Main_System_IC}
(\psi, v, w, \eta) |_{t=0} =(\psi_0, \psi_1, \psi_2, \psi_0).
\end{equation}
\end{subequations}
We then introduce the vector solution $\vecc{\Psi}=(\psi, v, w, \eta)^T$, where $(\psi, v, w, \eta)^T$ solves \eqref{Main_System}, with $\vecc{\Psi}(0)=\vecc{\Psi}_0=(\psi_0, \psi_1, \psi_2, \psi_0)^T $. Throughout this work, we assume that $n \geq 3$. For $s\geq 1$, we define the norm
\begin{equation} \label{1}
\begin{aligned}
||| \vecc{\Psi} |||_{H^s}^2
=&\, \begin{multlined}[t] \Vert\Delta(\psi +\tau v)\Vert^{2}_{H^s}+\Vert\nabla(\psi +\tau v)\Vert^{2}_{H^{s-1}}+\Vert  v+\tau w\Vert^{2}_{H^{s-1}}\\+\Vert\nabla(v+\tau w)\Vert^{2}_{H^s}+\Vert  \Delta v\Vert^{2}_{H^s} 
+\Vert \nabla v\Vert^{2}_{H^{s-1}}+\Vert w\Vert^2_{H^{s-1}} \\+\int_0^\infty (-g')\Vert \nabla\eta (r)\Vert^{2}_{H^{s-1}}\, \ds+\int_0^\infty (-g')\Vert \Delta\eta (r)\Vert^{2}_{H^s}\, \ds,\end{multlined}  
\end{aligned}  
\end{equation}
and the semi-norm
\begin{equation} \label{2}
\begin{aligned}
|\vecc{\Psi}|_{\mathbf{H}^{s}}^2
=&\, \begin{multlined}[t]
\Vert \Delta (\psi+\tau v)\Vert _{H^{s-1}}^{2} +\Vert\Delta v\Vert_{H^{s-1}}^{2}+\Vert\nabla(v+\tau w)\Vert^{2}_{H^{s-1}}+\Vert \nabla v\Vert_{H^{s-1}}^2\\+\| w\|^2_{H^{s-1}} 
+\int_0^\infty (-g')\Vert \nabla \eta(r)\Vert _{H^{s-1}}^{2}\, \ds+\int_0^\infty (-g')\Vert \Delta \eta(r)\Vert _{H^{s}}^{2}\, \ds. \end{multlined}
\end{aligned}  
\end{equation}
To prove global well-posedness in the critical case, we intend to derive energy estimates that are uniform in time. We will achieve this by using carefully designed weighted energies. In particular, we define the weighted energy norm     
\begin{equation}\label{Energy_Main}
\|\vecc{\Psi}\|_{\mathbbm{E}, t}^2=\sum_{i=0}^{[\frac{s-1}{2}]}\sup_{0\leq \sigma\leq t} (1+\sigma)^{i-1/2}||| \nabla^i \vecc{\Psi} (\sigma)|||_{H^{s-2i}}^2,    
\end{equation}
and the  associated dissipative norm 
\begin{equation} \label{Energy_Main_Dissipative}
\begin{aligned}  
\|\vecc{\Psi}\|_{\mathbbm{D},t}^2=&\, \begin{multlined}[t]\sum_{i=0}^{[\frac{s-1}{2}]}\int_0^t (1+\sigma)^{i-1/2}| \nabla^i \vecc{\Psi}(\sigma)|_{\mathbf{H}^{s-2i}}^2\, \textup{d}\sigma\\      
+\sum_{i=0}^{[\frac{s-1}{2}]}\int_0^t (1+\sigma)^{i-3/2}|||\nabla^i \vecc{\Psi}(\sigma)|||_{H^{s-2i}}^2(\sigma) \, \textup{d}\sigma. \end{multlined}
\end{aligned}
\end{equation}  
The choice of the upper index in the above sums will be clarified by the upcoming analysis. To formulate the decay results, it is helpful to introduce the vector   
\begin{equation} \label{def_U}
\vecc{U} =(v+\tau w,\nabla(\psi+\tau v),\nabla v)^T. 
\end{equation}
Its value at initial time is then
\begin{equation}
\vecc{U}_{0}=(\psi_{1}+\tau \psi_{2},\nabla(\psi_0+\tau \psi_{1}),\nabla \psi_{1})^T.
\end{equation}
We can now state the main result of this work; its proof will be given in Section~\ref{Section_Proof_Main_Result}.
\begin{theorem}\label{Main_Theorem}
Let $b =\tau c^2$ and assume that $n\geq 3$.  Furthermore, suppose that $s\geq [\tfrac{3n}{2}]+5$ and that assumptions $(\textup{G}1)$--$(\textup{G}4)$ on the memory kernel hold. Then there exists a constant $\delta_0>0$ such that if 
\begin{equation}
||| \vecc{\Psi}_0|||_{H^s}+\Vert \vecc{U}_0\Vert_{L^1} \leq \delta_0,
\end{equation}		 
 then problem \eqref{Main_whole} has a unique global solution $\vecc{\Psi}$, which satisfies the weighted energy estimate 
 \begin{subequations}
		\begin{equation}\label{Global_Existence_Estimate}
	\|\vecc{\Psi}\|_{\mathbbm{E},t}^2+\|\vecc{\Psi}\|_{\mathbbm{D},t}^2 \lesssim ||| \vecc{\Psi}_0|||_{H^s}^2+\Vert \vecc{U}_0\Vert^2_{L^1}. 
		\end{equation}
Furthermore, the following optimal decay estimate for the lower-order derivatives holds:
		\begin{equation}\label{Decay_Estimates}  
		\Vert \nabla^j \vecc{U}(t)\Vert_{L^2}\lesssim (||| \vecc{\Psi}_0|||_{H^s}+\Vert \vecc{U}_0\Vert_{L^1}) (1+t)^{-n/4-j/2}, \quad 0\leq j\leq s_0,
		\end{equation}
		where $s_0=[\frac{2s-n}{4}]$, as well as
		\begin{equation}\label{Estimate_Decay_v}
\Vert \nabla^j v(\sigma)\Vert_{L^2}\lesssim (||| \vecc{\Psi}_0|||_{H^s}+\Vert \vecc{U}  _0\Vert_{L^1}) (1+t)^{-n/4-j/2}, \quad 0\leq j\leq s_0-1,
\end{equation}
and 
\begin{equation}\label{Estimate_Decay_w}
\Vert \nabla^j w(\sigma)\Vert_{L^2}\lesssim (||| \vecc{\Psi}_0|||_{H^s}+\Vert \vecc{U}_0\Vert_{L^1}) (1+t)^{-n/4-1/2-j/2}, \quad 0\leq j\leq s_0-1.
\end{equation}
\end{subequations}  
	\end{theorem}
\subsection{Discussion of the main result} Before moving onto the proof, we discuss the claims made above.
\begin{itemize}[leftmargin=0.45cm]	  
\item The global well-posedness stated above requires initial data to be significantly smother than in the subcritical case; see~\cite{nikolic2020jordan} for the analysis when $b> \tau c^2$. In the proof of Theorem~\ref{Main_Theorem}, we will rely on the decay rate of the linearized problem to build appropriate time-weighted norms. The above higher-regularity assumption is thus justified by the loss of regularity in the linearization, although possibly not optimal.  	
\item Although vector $\vecc{U}$ contains $\nabla v$, decay estimate for $\nabla^j v$ given in \eqref{Estimate_Decay_v} is better than the one that follows from \eqref{Decay_Estimates}. Estimate \eqref{Estimate_Decay_v} also contains the decay rate for $\Vert v\Vert_{L^2}$, which cannot be deduced from \eqref{Decay_Estimates}. Moreover, \eqref{Estimate_Decay_w} reveals a better decay rate for $\Vert \nabla^ j w\Vert_{L^2}$.
\item   It is known that solutions of the linear wave equation in $\R^n$, where both linear damping and memory damping with the exponentially decaying kernel are present, have a decay rate $(1+t)^{-n/4}$ in the $L^2$ norm; see~\cite{DRK_2010}. It is also known that the same decay can be recovered if only one of the above-mentioned damping mechanisms is present in the equation; see \cite{Matsu_1977, CGP07}. Thus combining damping mechanisms does not necessarily lead to an improved decay rate.\\  
\indent For the linear MGT equations without memory (i.e., when $g=0$), the same decay rates as in the theorem above are optimal;  see~\cite{PellSaid_2019_1}. For small enough data, the nonlinear problem's solution should obey the same decay rate. As in the linear damped wave equation, the memory term is not expected to affect this. We thus expect the estimates of Theorem~\ref{Main_Theorem} to be sharp. However, the decay estimates hold for a solution with a lower Sobolev regularity than that of the initial data. This is common when there is a loss of regularity in the linearized problem; see, for instance, \cite{IK08, Duan_Ruan_Zhu_2012, Racke_Said_2012_1}.  
\item As mentioned before, in~\cite{lasiecka2017global}, the nonlocal JMGT equation in the critical case on smooth bounded domains is stabilized by considering a memory term that combines both $\Delta \psi$ and $\Delta \psi_t$:
\begin{align} \label{mixed_memory}
\int_0^\infty g(r)\Delta (\psi+\tfrac{c^2}{b}\psi_t)(t-r)\, \ds.
\end{align} 
We expect that our analysis can be extended to a memory acting on $\psi_t$ only or a mixed-type of memory like \eqref{mixed_memory} as well. 
\item In the subcritical case, global existence can be obtained without relying on the time decay of the linearized problem; see~\cite{nikolic2020jordan}. However, here due to the regularity loss, we have to rely on the time decay of the solution to the linear problem to construct appropriate time-weighted norms. For this reason, linear estimates of Propositions \ref{Proposition_Linear},  \ref{Decay_w_New}, and \ref{Lemma_decay_v_infty} below are crucial in the proof of the main result.      
\end{itemize}      
\subsection{The semigroup framework and short-time existence} 
We next briefly recall the semigroup framework that allows us to prove short-time well-posedness of the problem; we refer to~\cite{nikolic2020jordan} for more details. \\
\indent We adapt the functional framework of~\cite{dell2016moore} to our setting and introduce the weighted $L^{2}$ spaces,
\begin{equation}
L^2_{\tilde{g}}=L^{2}_{\tilde{g}}(\mathbb{R}^{+}, L^2(\mathbb{R}^{n}))
\end{equation}
with three types of weights: $\tilde{g} \in \{g, -g' , g''\}$; see also~\cite{nikolic2020jordan, nikolic2020mathematical}. The space is endowed with the inner product 
\begin{equation}
\left(\eta,\tilde{\eta} \right)_{L^2, \tilde{g}}=%
\displaystyle\int_{0}^{\infty}\tilde{g}(r)\left(
\eta(r),\tilde{\eta}(r)\right)_{L^{2}(\mathbb{R}%
	^{n})}\ds  
\end{equation}
for $\eta, \tilde{\eta} \in L^2_{\tilde{g}}$, and the corresponding norm is 
\begin{equation}
\Vert\eta\Vert^{2}_{L^2, \tilde{g}}=\int_{0}^{\infty}\tilde{g}(r)\Vert%
\eta(r)\Vert_{L^{2}}^{2}\ds.
\end{equation} 
For an integer $m \geq 1$, we also introduce the Hilbert spaces
\begin{equation} \label{Hs}
\begin{aligned}
\mathcal{H}^{s-1}=&\, \begin{multlined}[t] \{ \psi: \ D^\alpha \psi \in L^2(\R^n), \  1 \leq |\alpha| \leq s\} \times H^{s}(\R^n) \times {H}^{s-1}(\R^n)
\times \mathcal{M}^s,
\end{multlined}
\end{aligned}
\end{equation}
where 
\begin{equation}
\mathcal{M}^s= \{\eta: \ D^\alpha \eta \in L^2_{-g'}, \  1 \leq |\alpha| \leq s\}.  
\end{equation}
We recall that we have assumed $n \geq 3$ and that the homogeneous Sobolev space $\dot{H}^s(\R^n)$ is not complete when $s\geq \frac{n}{2}$; see \cite[\S 1]{bahouri2011fourier}. The space $\mathcal{H}^{s-1}$ is equipped with the norm
\begin{equation} \label{norm_Hm-1}
\begin{aligned}
\|\vecc{\Psi}\|^2_{\mathcal{H}^{s-1}}=\|\nabla \psi\|^2_{H^{s-1}}+\|v\|^2_{H^{s}}+\|w\|^2_{H^{s-1}}+\|\nabla \eta\|^2_{H^{s-1},-g'},
\end{aligned}
\end{equation}
where
\begin{equation}
\|\nabla \eta\|^2_{H^{s-1},-g'} = \sum_{i=1}^{s-1} \|\nabla^{(i)} \eta\|^2_{L^2,-g'}.
\end{equation}
We remark that 
\begin{equation}
\begin{aligned}
\|\vecc{\Psi}\|^2_{\mathcal{H}^{s-1}}\lesssim&\, \begin{multlined}[t]
\Vert\nabla(\psi +\tau v)\Vert^{2}_{H^{s-1}}+\Vert  v+\tau w\Vert^{2}_{H^{s-1}} +\Vert \nabla v\Vert^{2}_{H^{s-1}}\\+\Vert w\Vert^2_{H^{s-1}}+\|\nabla \eta\|^2_{H^{s-1},-g'} \end{multlined}\\
\lesssim&\, ||| \vecc{\Psi} |||_{H^s}^2, 
\end{aligned}     
\end{equation}  
recalling that the norm $|||\cdot |||_{H^s}$ is defined in \eqref{1}. To rewrite problem \eqref{Main_System} as an abstract first-order evolution equation, we introduce the operator 
\begin{equation}
\begin{aligned}
\mathcal{A}\begin{bmatrix}
\psi \\[1mm]
v \\[1mm]
w \\[3mm]
\eta 
\end{bmatrix} = \begin{bmatrix}
v \\[1mm]
w \\[1mm]
-\dfrac{1}{\tau} w+ \frac{c^2_g}{\tau} \Delta \psi+\frac{b}{\tau} \Delta v+\frac{1}{\tau}\displaystyle \int_0^\infty g(r)\Delta \eta(r) \ds \\[3mm]
v+\mathbb{T} \eta
\end{bmatrix}  
\end{aligned}
\end{equation}
with the domain
\begin{equation} \label{D(A)_m}
\begin{aligned}
D(\mathcal{A})
=\,  \left\{\vecc{\Psi} \in \mathcal{H}^{s-1} \left\vert\rule{0cm}{1cm}\right. \begin{matrix}
w \in H^{s}(\R^n), \\[2mm]
\dfrac{c^2_g}{\tau} \Delta \psi+\dfrac{b}{\tau} \Delta v+\dfrac{1}{\tau}\displaystyle \int_0^\infty g(r)\Delta \eta(r)  \in H^{s-1}(\R^n), \\[4mm]
\eta \in D(\mathbb{T}) \end{matrix} \right\};
\end{aligned}
\end{equation}  
cf.~\cite{dell2016moore, nikolic2020jordan}. Here the linear operator $\mathbb{T}$ is given by
\begin{equation} \label{def_T_eta}
\mathbb{T}\eta=-\eta_r,
\end{equation}
and has the domain
\begin{equation}
D(\mathbb{T})=\{\eta\in%
\mathcal{M}^s\, \big| \ \eta_r\in \mathcal{M}^s,\ \eta(r=0)=0\},
\end{equation}
where the index $r$ stands for the distributional derivative with respect to the variable $r$; cf.~\cite{dell2016moore}. Then we can formally see $\vecc{\Psi}$ as the solution to
\begin{equation} \label{abstract_evol_eq}
\begin{aligned}
\begin{cases}
\dfrac{\textup{d}}{\dt}\vecc{\Psi}(t)= \mathcal{A} \vecc{\Psi}(t)+\mathbb{F}(\vecc{\Psi}, \nabla \vecc{\Psi}),\quad t>0, \vspace{0.2cm}\\[1mm] 
\vecc{\Psi}(0)=\vecc{\Psi}_0,
\end{cases} 
\end{aligned}
\end{equation}
with the nonlinear term given by
\begin{equation} \label{def_F}
\begin{aligned}
\mathbb{F}(\vecc{\Psi}, \nabla \vecc{\Psi})= \frac{2}{\tau} \, \big[
0,\
0 ,\
kv w+ \nabla \psi \cdot \nabla v ,\
0 
\big]^T.
\end{aligned}
\end{equation}
It is known that the operator $\mathcal{A}$ generates a strongly continuous semigroup.
\begin{proposition}[see Theorem 4.1 in~\cite{nikolic2020jordan}]
	Let $b \geq \tau c^2 > \tau c^2_g$ and $s \geq 1$. Assume that $n \geq 3$. Then the linear operator $\mathcal{A}$
	is the infinitesimal generator of a linear $\textup{C}_0$-semigroup $ S(t) = e^{t\mathcal{A}}: \, \mathcal{H}^{s-1} \rightarrow \mathcal{H}^{s-1}.$
\end{proposition}
We also recall that our problem is well-posed in the critical case for a sufficiently short final time.
\begin{theorem}[see Theorem 5.1. in~\cite{nikolic2020jordan}] \label{Thm:LocalExistence}
	Let $b \geq \tau c^2> \tau c^2_g$ and $n \geq 3$. Assume that $\vecc{\Psi}_0 \in \mathcal{H}^{s-1}$ for an integer $s>n/2+1$. 
	Then there exists a final time $T=T(\|\vecc{\Psi}_0\|_{\mathcal{H}^{s-1}})$,
	such that problem \eqref{Main_whole} admits a unique mild solution 
	\begin{equation}
	\vecc{\Psi}=(\psi, v ,w, \eta)^T \in C([0,T]; \mathcal{H}^{s-1}),
	\end{equation} 
	given by       
	\begin{equation} \label{IP_Psi0}
	\vecc{\Psi}= e^{t\mathcal{A}}\vecc{\Psi}_0+\int_0^te^{(t-r)\mathcal{A}}\mathbb{F}(\vecc{\Psi}, \nabla \vecc{\Psi})(r)\, \textup{d}r,  
	\end{equation}  
	where the functional $\mathbb{F}$ is defined in \eqref{def_F}. 
\end{theorem}

\section{Auxiliary estimates of a regularity-loss type \\ for the corresponding linearization}   \label{Sec:EstLin} 
We intend to use the mild solution  \eqref{IP_Psi0} of the nonlinear problem to establish its asymptotic behavior. Therefore, much of the decay analysis will transfer to the results for the corresponding linear version of the JMGT equation. For this reason, we next recall known decay estimates and derive new ones for the following linearization:
\begin{equation}  \label{Main_System_linear}
\begin{cases}
\psi_{t}=v, \\ 
v_{t}=w, \\ 
\tau w_{t}=- w+c^2_g\Delta \psi+b\Delta v + \displaystyle%
\int_{0}^{\infty}g(r)\Delta\eta(r)\ds,
\\ 
\eta_{t}=v-\eta_{r},%
\end{cases}%
\end{equation}
with the same initial data as in \eqref{Main_System_IC}. To formulate the results, we set
\begin{equation} \label{def_U}
\vecc{U} =(v+\tau w,\nabla(\psi+\tau v),\nabla v)^T,
\end{equation}
where in this section $(\psi, v, w, \eta)^T$ solves the linear problem \eqref{Main_System_linear}. We choose to work with the vector $\vecc{U}$ here instead of the solution vector because the $L^2$ norm of $\vecc{U}$ and its corresponding $H^s$ norm define energy norms. These norms are known to satisfy certain decay estimates with respect to time, which we next recall. \\
\indent To state the decay estimates for the linearized problem, we also introduce the following energy norm: 
\begin{equation} \label{norm_1}
\begin{aligned}
\left\Vert \nabla^j\vecc{\Psi}\right\Vert _{\mathrm{H}}^{2} =& \begin{multlined}[t]\, \Vert
\nabla^{j+1} v\Vert _{L^{2}}^{2}+\Vert \nabla^{j+1} (\psi+\tau v)\Vert
_{L^{2}}^{2}
 +\Vert \nabla^j(v+\tau w)\Vert _{L^{2}}^{2}+\Vert \nabla^{j+1} \eta\Vert^{2}_{L^2, -g'}. \end{multlined}
\end{aligned}
\end{equation}
We note that this norm relates to $\|\cdot\|_{\mathcal{H}^{m-1}}$, defined in \eqref{norm_Hm-1}, through the inequality
 \[\left\Vert \nabla^{m-1}\vecc{\Psi}\right\Vert _{\mathrm{H}}\leq \|\vecc{\Psi}\|_{\mathcal{H}^{m-1}}.\]

\begin{proposition}  [See Theorem 3.2 in~\cite{Bounadja_Said_2019}]\label{Proposition_Linear} 
	Let $b=\tau c^2$. Suppose that \[\vecc{U}_{0} \in (H^{s}(\mathbb{R}^n)\cap L^{1}(\mathbb{R}^n))^3\] for an integer $s\geq 0$. Then for all $j= 0, 1,2,\dots s-\ell$, it holds 
	\begin{align} \label{decay_loss}
	\Vert\nabla^{j}\vecc{U}(t)\Vert _{L^{2}}\lesssim\Vert \nabla^j\vecc{\Psi}(t)   \Vert_{\mathrm{H}}\lesssim  (1+t)^{-{n}/{4-{j}/{2}}}\Vert \vecc{U}_{0}\Vert _{L^{1}} +(1+t)^{-\ell/2} \Vert\nabla^{j+\ell}\vecc{U}_{0}\Vert _{L^{2}},         
	\end{align} 
	where $\ell\leq s$. 
\end{proposition}  
Estimate \eqref{decay_loss} has the decay property of regularity-loss type. In other words, it holds for a solution with a lower regularity than that of the initial data. In fact, we need $\vecc{U}_0$ to be at least in $\left(L^1(\R^n)\cap H^{[n/2]+1}(\R^n)\right)^3$, which is a significant regularity gap compared to the subcritical case, where the same decay rate of the solution for the linearized  problem holds for $\vecc{U}_0$ in $\left(L^1(\R^n)\cap L^2(\R^n)\right)^3$; cf.~\cite{Bounadja_Said_2019}. It is well-known that this gap creates a  major difficulty when dealing with the corresponding nonlinear problems. To get around this obstacle, we should carefully design time-weighted  norms and combine the energy method in the analysis with a time-weight with a negative power. \\
\indent We next give an  estimate of $\Vert \nabla ^j w\Vert_{L^2}$, which will be used to motivate the decay rate in the nonlinear equation.   
\begin{proposition}\label{Decay_w_New}
	Let the assumptions of Proposition~\ref{Proposition_Linear} hold with $s \geq 1$. Let \[w(t=0) = \psi_2 \in H^s(\R^n).\] Then for any integer $\ell \leq s-1$ and any $j \in \{0,\dots,s-\ell-1\}$, it holds
	\begin{equation}\label{Decay_estimate_W} 
\Vert\nabla^{j}w(t)\Vert_{L^{2}}\lesssim  (\Vert \nabla^j \psi_2 \Vert_{L^2}+\Vert \mathbf{U}_{0}\Vert_{L^{1}})(1+t)^{-\frac{n}{4}-\frac{1}{2}-\frac{j}{2}}
	+(1+t)^{-\frac{\ell}{2}}\Vert\nabla^{j+\ell+1}\mathbf{U}_{0}\Vert_{L^{2}},
\end{equation}
provided that the thermal relaxation time $\tau>0$ is sufficiently small. 
\end{proposition} 
The proof of Proposition~\ref{Proposition_Linear} is similar to the one of~\cite[Proposition 7.1]{nikolic2020mathematical} in the subcritical case with small modifications to accommodate the loss of regularity, so we postpone it to Appendix~\ref{Appendix_B}. 
\begin{proposition}\label{Lemma_decay_v_infty}
	Let the assumptions of Proposition~\ref{Decay_w_New} hold. Then for any $j \in \{0, \ldots, s-\ell-1\}$, we have
	\begin{equation}\label{v_L_2_Estimate}
	\Vert\nabla^{j}v(t)\Vert_{L^{2}}\lesssim (\Vert \nabla^j \psi_2\Vert_{L^2}+\Vert \mathbf{U}_{0}\Vert_{L^{1}})(1+t)^{-\frac{n}{4}-\frac{j}{2}}+(1+t)^{-\ell/2} \Vert\nabla^{j+\ell}\vecc{U}_{0}\Vert _{H^{1}}.
	\end{equation}
 Moreover, assuming that $\mathbf{U}_0 \in \left(L^1(\R^n) \cap H^{n+[n/2]+3}(\R^n)\right)^3$ and $ \psi_2 \in H^{[n/2]+1}(\R^n)$, it holds
 	\begin{equation} \label{v_L_infty_Estimate}
 	\begin{aligned}
 	\Vert v(t) \Vert_{L^\infty}\lesssim (\Vert \psi_2\Vert_{H^{[n/2]+1}}+\Vert \mathbf{U}_0\Vert_{L^1}+\Vert \vecc{U}_0\Vert_{H^{n+[n/2]+3}})(1+t)^{-\frac{n}{2}}.
 	\end{aligned}
 	\end{equation}  
\end{proposition}
The proof of Proposition~\ref{Lemma_decay_v_infty} follows along the lines of \cite[Lemma 7.2]{nikolic2020mathematical}, but adapted to the loss of regularity setting, so we postpone it as well to Appendix~\ref{AppendixC}. \\

\paragraph{\bf Auxiliary time-weighted quantities.} Motivated by the decay estimates for the linear problem,  which we expect to carry over to the nonlinear problem for sufficiently small data, we introduce the following time-weighted quantity: 
\begin{equation}\label{M_j_Def}
\begin{aligned}  
M_{j}[u](t) = \sup_{0\leq \sigma\leq t}\left( 1+\sigma \right) ^{%
	\frac{n+j}{2}}\left\Vert \nabla^j u(\sigma)\right\Vert _{L^{\infty }}.
\end{aligned} 
\end{equation}%
This function will play an important role in the energy analysis of the nonlinear equation. In particular, we will utilize $M_0[\boldsymbol{U}]$, $M_1[\boldsymbol{U}]$, and $M_0[v]$. Furthermore, we define
\begin{equation}\label{Def_Mcal}
\begin{aligned}
\mathdutchcal{M}[v, w, \vecc{U}](t)
 =&\,\begin{multlined}[t] \sum_{j=0}^{s_0}\sup_{0\leq \sigma \leq t}( 1+\sigma
) ^{n/4+j/2}\left\Vert \nabla ^{j} \vecc{U}\left( \sigma \right)
\right\Vert _{L^2}\\
+\sum_{j=0}^{s_0-1}\sup_{0\leq \sigma \leq t}\left[( 1+\sigma
) ^{n/4+j/2}\Vert \nabla^j v(\sigma)\Vert_{L^2}\right.\\ \left.+(1+\sigma)^{\frac{n}{4}+\frac{1}{2}+\frac{j}{2}}\Vert \nabla^jw(\sigma)\Vert_{L^2}\right]. \end{multlined}
\end{aligned}
\end{equation}  
Motivated by the linear decay rates, we set \[s_0=\left[\frac{2s-n}{4}\right].\]  
Let us clarify this choice. For the first term in $\mathcal{M}[v, w, \vecc{U}](t)$ to be uniformly bounded in time, we have from \eqref{decay_loss} that
	\begin{align} 
\Vert\nabla^{j}\vecc{U}(t)\Vert _{L^{2}}\lesssim (1+t)^{-{n}/{4-{j}/{2}}}(\Vert \vecc{U}_{0}\Vert _{L^{1}} +\Vert\nabla^{s}\vecc{U}_{0}\Vert _{L^{2}}).     
\end{align} 
Thus, we should have 
\begin{equation}
\frac{n}{2}+2j \leq j+\ell \leq s,
\end{equation}
where $j$, $\ell$, and $s$ are integers. Therefore, we arrive at the condition for $j$ in $\nabla^j \vecc{U}$ as follows:
\begin{equation}
j \leq s_0 =\left[\frac{s-n/2}{2} \right]=\left[\frac{2s-n}{4}\right].
\end{equation}
\noindent For the second term in $\mathdutchcal{M}[v, w, \vecc{U}](t)$ we have from \eqref{v_L_2_Estimate} the expected decay
\begin{equation}
\Vert\nabla^{j}v(t)\Vert_{L^{2}}\lesssim (\Vert \nabla^j \psi_2\Vert_{L^2}+\Vert \mathbf{U}_{0}\Vert_{L^{1}}+\Vert\nabla^{j+\ell}\vecc{U}_{0}\Vert _{H^{1}})(1+t)^{-\frac{n}{4}-\frac{j}{2}},
\end{equation}
provided that
\begin{equation}
\frac{\ell}{2}\geq \frac{n}{4}+\frac{j}{2}\qquad \text{and}\qquad j+\ell+1\leq s.   
\end{equation}  
This leads to the condition
\begin{equation}
\frac{n}{2}+2j +1 \leq j+\ell +1\leq s,
\end{equation}
and thus
\begin{equation}
j \leq \frac{s-n/2-1}{2}.
\end{equation}
We thus set the upper index $j$ in the term $\nabla^j v$ as \[s_0-1=\left[\frac{2s-n}{4}\right]-1 \leq  \frac{s-n/2-1}{2}.\] Justification of the bound on $j$ in the term $\nabla^j w$ within $\mathdutchcal{M}[v, w, \vecc{U}](t)$ can be done analogously. \\
\indent We can relate the three quantities $M_0$, $M_1$, and $\mathdutchcal{M}$ by the following inequality. 
\begin{lemma} \label{M_estimate}
Let $s_0=[\tfrac{2s-n}{4}]\geq n/2+1$ in \eqref{Def_Mcal}. Then the following inequality holds:
	\begin{equation}\label{M_0_M_estimate}
	M_0[\boldsymbol{\vecc{U}}](t)+M_1[\boldsymbol{\vecc{U}}](t)+M_0[v](t)\lesssim \mathdutchcal{M}[v, w, \vecc{U}](t).
	\end{equation} 
\end{lemma}
\begin{proof}
We rely on the Gagliardo--Nirenberg interpolation inequality 
\begin{equation}\label{L_infty_Interp}
\left\Vert v\right\Vert _{L^{\infty }}\lesssim \left\Vert \nabla^{q}%
v\right\Vert _{L^{2}}^{\frac{n}{2q}}\left\Vert v
\right\Vert _{L^{2}}^{1-\frac{n}{2q}} \quad \text{for} \ q\geq \frac{n}{2},
\end{equation}
to arrive at
\begin{equation}
\begin{aligned}
\Vert v(\sigma)\Vert _{L^{\infty }}\lesssim&\, (1+\sigma)^{(-n/4-q/2)\frac{n}{2q}} \left(\mathdutchcal{M}[v, w, \vecc{U}](t)\right)^{\frac{n}{2q}}
\times (1+\sigma)^{-n/4(1-\frac{n}{2q})} \left(\mathdutchcal{M}[v, w, \vecc{U}](t)\right)^{1-\frac{n}{2q}}\\
=&\, C(1+\sigma)^{-n/2}\mathdutchcal{M}[v, w, \vecc{U}](t),
\end{aligned}
\end{equation}
where $C>0$ does not depend on time and $0 \leq \sigma \leq t$. Above, we have also used
\[(1+\sigma)^{(n/4+q/2)}\Vert \nabla^{q}v \Vert _{L^{2}} \lesssim \mathdutchcal{M}[v, w, \vecc{U}](t).\]
Therefore, we have
\[M_{0}[v](t)\lesssim \mathdutchcal{M}[v, w, \vecc{U}](t),\] 
provided that $s_0-1\geq q\geq n/2$. We can retrace the steps above with $\vecc{U}$ in place of $v$ to show that
\[M_{0}[\vecc{U}](t)\lesssim \mathdutchcal{M}[v, w, \vecc{U}](t).\]
Using again the interpolation inequality \eqref{L_infty_Interp}, we have
\begin{equation}
\begin{aligned}
\left\Vert \nabla \vecc{U}\right\Vert _{L^{\infty }}\lesssim \left\Vert \nabla ^{q}%
\nabla \vecc{U} \right\Vert _{L^{2}}^{\frac{n}{2q}}\left\Vert \nabla \vecc{U}
\right\Vert _{L^{2}}^{1-\frac{n}{2q}}.
\end{aligned}
\end{equation}
We thus obtain 
\begin{equation}
\begin{aligned}
\left\Vert \nabla \vecc{U}(\sigma)\right\Vert _{L^{\infty }}\lesssim& \, \begin{multlined}[t] (1+\sigma)^{(-n/4-(q+1)/2)\frac{n}{2q}} \left(\mathdutchcal{M}[v, w, \vecc{U}](t)\right)^{\frac{n}{2q}}\\
\times (1+\sigma)^{-(n/4+1/2)(1-\frac{n}{2q})} \left(\mathdutchcal{M}[v, w, \vecc{U}](t)\right)^{1-\frac{n}{2q}} \end{multlined}\\
=&\, C(1+\sigma)^{-n/2-1/2}\mathdutchcal{M}[v, w, \vecc{U}](t).
\end{aligned}
\end{equation}
Therefore, on account of the assumptions on $s_0$, we have 
\begin{equation}
M_1[\boldsymbol{\vecc{U}}](t)\lesssim \mathdutchcal{M}[v, w, \vecc{U}](t). 
\end{equation}
This step completes the proof. 
\end{proof}
\section{Proof of the main result} \label{Sec:ProofMain}
\label{Section_Proof_Main_Result}  
We present here the proof of Theorem~\ref{Main_Theorem} up to the following two bounds:
\begin{equation}\label{Main_Estimate_E_D_Main}
\begin{aligned}  
\|\vecc{\Psi}\|_{\mathbbm{E},t}^2+\|\vecc{\Psi}\|_{\mathbbm{D},t}^2 \lesssim& \,\begin{multlined}[t] ||| \vecc{\Psi}(0)|||_{H^s}^2+\Big\{\mathdutchcal{M}[v, w, \vecc{U}](t)\Big.\vspace{0.2cm}\\
\Big.+M_0[\vecc{U}](t)+M_1[\vecc{U}](t)+M_0[v](t)\Big\}\|\vecc{\Psi}\|_{\mathbbm{D},t}^2, \end{multlined}
\end{aligned}
\end{equation}
and 
\begin{equation} \label{M_weighted_estimate_Main}
\begin{aligned}
\mathdutchcal{M}[v, w, \vecc{U}](t)\lesssim&\, \begin{multlined}[t] ||| \vecc{\Psi}(0)|||_{H^s}^2+\Vert \vecc{U}_0\Vert^2_{L^1}\vspace{0.2cm}\\
+\mathdutchcal{M}[v, w, \vecc{U}]^2(t) +M_{0}[\vecc{U}]( t)
\|\vecc{\Psi}\|_{\mathbbm{E},t}.  \end{multlined}
\end{aligned} 
\end{equation}
Their proof is more involved and will thus be carried out separately in the upcoming section. 
\renewcommand*{\proofname}{Proof of Theorem~\ref{Main_Theorem}}
\begin{proof}
By Lemma~\ref{M_estimate}, we have
\begin{equation}
M_0[\vecc{U}](t)+M_1[\vecc{U}](t)+M_0[v](t)\lesssim \mathdutchcal{M}[v, w, \vecc{U}](t).  
\end{equation}
Therefore, if we set
\begin{eqnarray*}
	\mathbb{Y}(t)=\|\vecc{\Psi}\|_{\mathbbm{E},t}+\|\vecc{\Psi}\|_{\mathbbm{D},t}+\mathdutchcal{M}[v, w, \vecc{U}](t),
\end{eqnarray*}
 the two estimates \eqref{Main_Estimate_E_D_Main} and \eqref{M_weighted_estimate_Main} yield
\begin{equation}    \label{Main_Bootstrap}     
\mathbb{Y}(t)^2\lesssim ||| \vecc{\Psi}(0)|||_{H^s}^2+\Vert \vecc{U}_0\Vert^2_{L^1}+\mathbb{Y}(t)^3. 
\end{equation}  
Provided that $\delta_0=||| \vecc{\Psi}(0)|||_{H^s}+\Vert \vecc{U}_0\Vert_{L^1}$  is sufficiently small, this inequality further implies that \[\mathbb{Y}(t)\lesssim ||| \vecc{\Psi}(0)|||_{H^s}+\Vert \vecc{U}_0\Vert_{L^1};\] see Lemma~\ref{Lemma_Stauss} for the technical inequality employed to arrive at this conclusion. Therefore, we have
\begin{subequations}\label{Uniform_Bound_E_M}
\begin{equation}\label{Uniform_Bound}
\|\vecc{\Psi}\|_{\mathbbm{E},t}+\|\vecc{\Psi}\|_{\mathbbm{D},t}\lesssim ||| \vecc{\Psi}(0)|||_{H^s}+\Vert \vecc{U}_0\Vert_{L^1} 
\end{equation}
as well as 
\begin{equation}\label{M_Uniform_Estimate}
\mathdutchcal{M}[v, w, \vecc{U}](t)\lesssim ||| \vecc{\Psi}(0)|||_{H^s}+\Vert \vecc{U}_0\Vert_{L^1},
\end{equation}
\end{subequations}
where we recall that the hidden constants within $\lesssim$ do not depend on time. The time-uniform estimate in \eqref{Uniform_Bound} allows us to extend the solution globally in time. Moreover, it proves the energy bound \eqref{Global_Existence_Estimate}, while the claimed decay rates \eqref{Decay_Estimates}, \eqref{Estimate_Decay_v}, and \eqref{Estimate_Decay_w} follow from estimate \eqref{M_Uniform_Estimate}. 
\end{proof}
\begin{remark}
Theorem~\ref{Main_Theorem} treats sound propagation in the presence of quadratic gradient nonlinearity $\partial_t (|\nabla \psi|^2)$. If local nonlinear effects in sound propagation can be neglected, as is the case, for example, when the propagation distance in terms of wavelengths is large enough, this term can be approximated as follows:
\[|\nabla \psi|^2 \approx \frac{1}{c^2} \psi_t^2;\]
see the discussions given in~\cite[\S 2.3]{coulouvrat1992equations} and ~\cite[\S 2.3]{jordan2016survey}. In such cases, the right-hand side nonlinearity in the JMGT equations involves only $\tilde{k}\psi_t \psi_{tt}$ for some $\tilde{k} \in \R$. We expect that the regularity requirements of Theorem~\ref{Main_Theorem} can then be relaxed and the proof further simplified; see the work \cite{nikolic2020mathematical} by the authors in this direction in the subcritical case  $b>\tau c^2$. \\
\end{remark}
The remaining of the paper is devoted to proving estimates \eqref{Main_Estimate_E_D_Main} and \eqref{M_weighted_estimate_Main}. To this end, we employ a delicate energy analysis based on time-weighted norms. 
\renewcommand*{\proofname}{Proof}
\section{Energy analysis of the JMGT equation\\ in the critical case} ~\label{Sec:EnergyAnalysis}
~\\
\noindent Before proceeding to the energy analysis, we need the following preparatory result.
\begin{lemma} \label{Norm_Estimate} Let $s \geq 3$  be a given integer. The inequality
	\begin{eqnarray}  
	|||\nabla \vecc{\Psi}|||_{H^{s-2}} \lesssim |\vecc{\Psi}|_{\mathbf{H}^{s}} 
	\end{eqnarray}
	holds for all $\vecc{\Psi}$, such that $|\vecc{\Psi}|_{\mathbf{H}^{s}}< \infty$.     
\end{lemma}
\begin{proof}
By virtue of the embedding $H^{s-1}(\R^n)\hookrightarrow H^{s-3}(\R^n)$, we have the following bounds:
	\begin{equation}
	\begin{aligned}
	\Vert\nabla^2(\psi +\tau v)\Vert_{H^{s-3}}\lesssim&\,  \Vert \Delta (\psi+\tau v)\Vert _{H^{s-1}},\\
	\Vert \nabla( v+\tau w)\Vert_{H^{s-3}}\lesssim&\, \Vert \nabla( v+\tau w)\Vert_{H^{s-1}},\\
	\Vert \nabla^2 v\Vert^{2}_{H^{s-3}}\lesssim&\, \Vert \Delta v\Vert^{2}_{H^{s-1}},\\  
	\Vert \nabla w\Vert_{H^{s-3}}\lesssim&\,\Vert w\Vert_{H^{s-1}}.
	\end{aligned}
	\end{equation}
	We note that
	\begin{equation}
	\begin{aligned}
	\Vert\nabla\Delta(\psi +\tau v)\Vert_{H^{s-2}}\lesssim \Vert \Delta (\psi+\tau v)\Vert _{H^{s-1}},\qquad \Vert  \nabla\Delta v\Vert_{H^{s-2}}\lesssim \Vert\Delta v\Vert_{H^{s-1}}.
	\end{aligned}
	\end{equation}
	Thus,
	\[\|\nabla^2(\psi+\tau v)\|_{H^{s-3}}+\Vert \nabla \Delta (\psi+\tau v)\Vert _{H^{s-2}} \lesssim \|\Delta (\psi+\tau v)\| _{H^{s-1}}. \]
	Moreover, due to the embedding $H^{s-1}_{-g'}\hookrightarrow H^{s-3}_{-g'}$, we have
	\begin{equation}
	\Vert \nabla^2\eta\Vert^{2}_{H^{s-3}, -g'}+\Vert \nabla\Delta\eta\Vert^{2}_{H^{s-2}, -g'}\lesssim \Vert \Delta \eta\Vert _{H^{s}, -g'}^{2}\end{equation}
	The term $\Vert\nabla^2(v+\tau w)\Vert_{H^{s-2}}$ can be estimated as follows:
	\begin{equation}
	\Vert\nabla^2(v+\tau w)\Vert_{H^{s-2}}\lesssim \Vert\nabla(v+\tau w)\Vert_{H^{s-1}}.
	\end{equation} 
	Employing the above inequalities leads to \eqref{Norm_Estimate}.
\end{proof}

\subsection{Energy bounds and the construction of norms}
Our main goal going forward is to derive the two key estimates \eqref{Main_Estimate_E_D_Main}  and \eqref{M_weighted_estimate_Main}.\\
\indent  For simplicity of notation  involving the nonlinear terms, we introduce 
\begin{equation}
F^{(0)}(v,w,\nabla \psi,\nabla v)=2k (vw+\nabla \psi \cdot \nabla v).
\end{equation}
To facilitate the high-order energy analysis, we apply the operator $\nabla^\kappa $ ($\kappa\geq 1$) to the system \eqref{Main_System} and denote
\begin{equation}
\tilde{f}=\nabla^\kappa f,\quad \text{for}\quad f\in{\{\psi, v,w,\eta \}},   
\end{equation}
which results in the space-differentiated system   
\begin{equation}  \label{Main_System_kappa}
\begin{cases}
\tilde{\psi}_{t}=\tilde{v}, \\ 
\tilde{v}_{t}=\tilde{w}, \\ 
\tau \tilde{w}_{t}=- \tilde{w}+c^2_g\Delta \tilde{\psi}+b\Delta \tilde{v} + \displaystyle%
\int_{0}^{\infty}g(r)\Delta\tilde{\eta}(r)\ds+
F^{(\kappa)}(\psi,v,\nabla \psi,\nabla v) 
,
\\ 
\tilde{\eta}_{t}=\tilde{v}-\tilde{\eta}_{r},%
\end{cases}%
\end{equation}
where $F^{(\kappa)}$ is defined as
\begin{equation}\label{F_k_Form}
F^{(\kappa)}(\psi,v,\nabla \psi,\nabla v)=2k[\nabla^\kappa,v]w+2k v\tilde{w}+2\kappa[\nabla^\kappa,\nabla \psi] \cdot \nabla v+2k\nabla \psi \cdot \nabla\tilde{w}
\end{equation}
for $\kappa \geq 1$. Above, $[\cdot, \cdot]$ denotes the commutator:
\[[A,B]=AB-BA.\] Note that
\begin{equation}\label{Derivative_Comuta}
\nabla^\kappa(AB)=[\nabla^\kappa,A]B+A\nabla^\kappa B, \qquad \kappa \geq 1.
\end{equation}
We also introduce the right-hand side functionals $R_\kappa^{(1)}$ and $R_\kappa^{(2)}$ as 
\begin{equation} \label{Def_R_k}
\begin{aligned}
R_\kappa^{(1)}(\varphi)= (F^{(\kappa)}, \varphi)_{L^2}, \quad R_\kappa^{(2)}(\varphi)= (\nabla F^{(\kappa)}, \nabla \varphi)_{L^2}, \quad \kappa \geq 0.
\end{aligned}
\end{equation}
 Here  
 $\varphi$ stands for different test functions that will be used in the proofs. \\
\paragraph{\bf Energy functionals} We define the energy of order $\kappa \geq 0$ as
\begin{equation} \label{energy_mathbfE}
\begin{aligned}
\mathbf{E}^{(\kappa)}(t)=& \,\Vert\nabla^{\kappa+1}(\psi +\tau v)\Vert^{2}_{L^2} + \Vert \nabla^\kappa( v+\tau w)\Vert^{2}_{L^2}+\Vert \nabla^{\kappa+1} v\Vert^{2}_{L^2}
\\
&+\Vert \nabla^{\kappa+1}\eta\Vert^{2}_{L^2, -g'}
+\Vert \Delta\nabla^{\kappa}\eta\Vert^{2}_{H^1, -g'}+\Vert\Delta\nabla^{\kappa}(\psi +\tau v)\Vert^{2}_{H^1} \\
&+ \Vert\nabla^{\kappa+1}(v+\tau w)\Vert^{2}_{H^1}
+\Vert \Delta \nabla^{\kappa} v(t)\Vert^{2}_{H^1}
+\Vert \nabla^\kappa w(t)\Vert_{L^2}^2.\vspace{0.2cm}
\end{aligned}
\end{equation} 
The corresponding dissipative energy is given by
\begin{equation} \label{energy_mathbfD}
\begin{aligned}
\mathbf{D}^{(\kappa)}(t)=&\, \begin{multlined}[t]\Vert \nabla^{\kappa+1} \eta\Vert _{H^2, -g'}^{2}+\Vert \Delta \nabla^{\kappa}\eta
\Vert _{H^1, -g'}^{2}
+\Vert \Delta\nabla^\kappa (\psi+\tau v)\Vert _{L^{2}}^{2} \vspace{0.2cm}\\
+\Vert\nabla^{\kappa+1}(v+\tau w)\Vert^{2}_{L^2}+\Vert \nabla^{\kappa+1} v\Vert_{L^2}^2+\Vert\Delta \nabla ^{\kappa}v\Vert_{L^2}^{2}+\|\nabla^\kappa w\|^2_{L^2}. \end{multlined}
\end{aligned}
\end{equation}
Note that these energies can be related to $|||\cdot |||_{H^s}$ and $|\cdot|_{\mathbf{H}^s}$ as follows:
\begin{equation} \label{identity_EHs}
\displaystyle \sum_{\kappa=0}^{s-1} \mathbf{E}^{(\kappa)}(t) = ||| \vecc{\Psi}(t)|||_{H^s}^2, \qquad \sum_{\kappa=0}^{s-1} \mathbf{D}^{(\kappa)}(t)= |\vecc{\Psi}(t)|_{\mathbf{H}^{s}}^2.
\end{equation}
\paragraph{\bf Problem-tailored energies.} To state our results and following~\cite{nikolic2020jordan, dell2016moore}, we also introduce energies that are tailored to our particular problem. For $\kappa \geq 0$ and $t \in [0,T]$, we set
\begin{equation} \label{energy}
\begin{aligned}
E_1^{(\kappa)}(t)
=& \,\begin{multlined}[t] \dfrac{1}{2}\left [\vphantom{\int_{0}^{%
		\infty}} c^2_g%
\|\nabla^{\kappa +1}(\psi+\tau v)\|_{L^2}^{2}+\tau(b -\tau
c^2_g )\|\nabla^{\kappa+1} v\|^{2}_{L^2}\right.   \\
\left. +\|\nabla^{\kappa}(v+\tau w)\|^{2}_{L^2}+\tau \Vert \nabla^{\kappa+1}
\eta \Vert _{L^2, -g'}^{2}+ \|\nabla^{\kappa+1} \eta\|^2_{L^2, g}
\right.\\
\left.+ 2\tau \int_{\mathbb{R}^{n}}
\int_{0}^{\infty}g(r) \nabla^{\kappa+1} \eta (r) \cdot \nabla^{\kappa+1} v
\ds \dx\right]\end{multlined}
\end{aligned}
\end{equation}
as well as
\begin{equation} \label{E_2}
\begin{aligned}
E_{2}^{(\kappa)}(t)
=&\, \begin{multlined}[t]\frac{1}{2}\left[\vphantom{\int_{0}^{%
		\infty}} c^2_g\left\Vert \Delta
( \nabla^{\kappa} (\psi+\tau v))\right\Vert _{L^{2}}^{2}+\tau(b -\tau c^2_g )\left\Vert
\Delta \nabla ^\kappa v\right\Vert _{L^{2}}^{2}\right. \\ \left.+\Vert \nabla^{\kappa+1} (v+\tau w)\Vert
_{L^2}^2+\tau \Vert \Delta \nabla^\kappa \eta \Vert^2_{L^2, -g'}    
+\Vert \Delta \nabla^\kappa  \eta\Vert^2_{L^{2}, g} \right. \\ \left. +2\tau \int_{\mathbb{R}^{n}}\int_{0}^{%
	\infty}g(r)(\Delta \nabla^\kappa v)(\Delta \nabla^\kappa \eta(r))\ds\dx\right]. \end{multlined}
\end{aligned}
\end{equation}
It can be shown that $E_{1}^{(\kappa)}(t)$ and $E_{2}^{(\kappa)}(t)$ are equivalent to $\mathbf{E}^{(\kappa)}(t)$ and $\mathbf{E}^{(\kappa+1)}(t)$, respectively; see~\cite{nikolic2020jordan} for the proof. Furthermore, they are known to satisfy the following decay estimates.
\begin{proposition}
	\label{Prop:E1_k} 
	Assume that $b=\tau c^2$. 
	Then the following estimates hold for $\kappa \geq 0$:
	\begin{equation} \label{dE_1_Dt}
	\frac{\textup{d}}{\dt}E_1^{(\kappa)}(t)+\frac{%
		1}{2}\Vert \nabla^{\kappa+1} \eta\Vert _{L^2, -g'}^{2}\leq \,
	|R^{(1)}_\kappa (\nabla^\kappa (v + \tau  w))| 
	\end{equation}
	and 
	\begin{equation}  \label{dE_1_Dt_k}
	\frac{\textup{d}}{\dt}E_2^{(\kappa)}(t)+\frac{1}{2}\Vert \Delta \nabla^\kappa\eta
	\Vert _{L^2, -g'}^{2}\leq \, |R^{(2)}_\kappa(\nabla ^\kappa (v + \tau  w))|.   
	\end{equation}
	In addition, it holds that 
	\begin{equation} \label{dE_2_Dt}
	\begin{aligned}
	&\frac{1}{2}\frac{\textup{d}}{\dt}\|\nabla^\kappa w(t)\|^2_{L^2}+\frac{1}{2}\|\nabla^\kappa w\|^2_{L^2}\\
	\lesssim&\,
	\Vert
	\Delta \nabla^\kappa (\psi+\tau v)\Vert_{L^2}^2+\Vert \Delta\nabla^\kappa v\Vert _{L^{2}}^2+\Vert\Delta\nabla^\kappa\eta\Vert^{2}_{L^2, g}+|R_\kappa^{(1)}(\nabla^\kappa w)| ,
	\end{aligned}
	\end{equation}
	for all $t\geq 0$, where the functionals $R^{(1)}_\kappa$  and $R^{(2)}_k$ are  defined in \eqref{Def_R_k}. 
\end{proposition}
\begin{proof}
	Keeping in mind the assumption $b=\tau c^2$, the statement follows from Propositions 4.1 and 4.2 in~\cite{nikolic2020mathematical}. 
\end{proof}
\subsection{Recovering dissipation via auxiliary functionals} \indent Compared to the analysis in the subcritical case performed in~\cite{nikolic2020jordan}, the dissipation of the $v$ component will be lost from the main energy inequality due to the critical  assumption $b =\tau c^2$.  To recover such a damping term, we introduce the functionals  $F_{i}^{(\kappa)}(t)$ below, where $i \{1, \ldots, 4\}$, and utilize the memory term  in a suitable  way.  We recall that in the absence of the memory damping (i.e., when $g=0$),  the linearized problem is unstable in the critical case; see~\cite{PellSaid_2019} for the revealing analysis.  \\ 
\indent Following~\cite{Racke_Said_2019, nikolic2020mathematical}, we first introduce 
\begin{equation} \label{F_Functionals}
\begin{aligned} 
F_{1}^{(\kappa)}(t)&=\,\int_{\mathbb{R}^{n}}\nabla^{\kappa+1} ( \psi+\tau v)\cdot \nabla^{\kappa+1} (v+\tau w)\dx, \\
F_{2}^{(\kappa)}(t)&=\,-\tau \int_{\mathbb{R}^{n}} \nabla^{\kappa+1} v  \cdot  \nabla^{\kappa+1} (v+\tau w) \dx, 
\end{aligned}
\end{equation}
They are known to satisfy the following estimates. 
\begin{proposition}[see~\cite{nikolic2020jordan}]
	\label{Lemma_F_1_k} Assume that $b=\tau c^2$.  For any $\epsilon _{0},\epsilon _{1}>0,$ it holds
	\begin{equation} \label{F_1_Estimate}
	\begin{aligned}
	&\frac{\textup{d}}{\dt}F_{1}^{(\kappa)}(t)+(c^2_g -\epsilon _{0}-(c^2 -c^2_g )\epsilon
	_{1})\Vert \Delta \nabla^{\kappa}(\psi+\tau v)\Vert _{L^{2}}^{2} \\
	\leq&\, \begin{multlined}[t] \Vert \nabla^{\kappa+1} (v+\tau w)\Vert _{L^{2}}^{2}+C(\epsilon
	_{0})\Vert \Delta \nabla^\kappa v\Vert _{L^{2}}^{2}
	+C(\epsilon _{1})\Vert \Delta\nabla^\kappa \eta\Vert _{L^{2},g}^{2}+|R_{\kappa}^{(2)}(\nabla^{\kappa}(\psi + \tau v))|.\end{multlined}
	\end{aligned}
	\end{equation}
	Moreover, for any $\epsilon_{2},\epsilon_{3}>0,$  we have 
	\begin{equation}\label{F_2_Estimate}
	\begin{aligned}
	&\frac{\textup{d}}{\dt}F_{2}^{(\kappa)}(t)+(1-\epsilon_{3})\Vert\nabla^{\kappa+1}(v+\tau w)\Vert^{2}_{L^2} \\
	\leq&\, \begin{multlined}[t]\epsilon_{2}\Vert\Delta\nabla^\kappa(\psi +\tau v)\Vert^{2}_{L^2}  +C(\epsilon_{3},\epsilon_{2})(\Vert \Delta\nabla^\kappa v\Vert^{2}_{L^2}+\Vert \nabla^{\kappa+1} v\Vert^{2}_{L^2}) \\
	+\frac{1}{2}\Vert \nabla^{\kappa+1}\eta\Vert^{2}_{L^2, g}+|R_\kappa^{(2)}(\tau \nabla^\kappa v)|,\end{multlined}
	\end{aligned}
	\end{equation}
	where the functional $R_\kappa^{(2)}$ is defined in \eqref{Def_R_k}.
\end{proposition}  
We observe that estimate \eqref{dE_1_Dt} does not have dissipative terms containing  $\Vert \nabla^{\kappa+1} v \Vert_{L^2}$ or $\Vert \Delta \nabla^\kappa v\Vert_{L^2}$. To recover the a dissipation term for $\Vert \nabla^{\kappa+1} v\Vert_{L^2}$, we introduce the functional 
\begin{equation}\label{F3_k}
F_{3}^{(\kappa)}(t)=- \tau\int_{0}^{\infty}\int_{\R^n} g(r)\nabla^{\kappa+1}\eta (r)\cdot \nabla^{\kappa+1} v \ds \dx;
\end{equation}
see~\cite{Bounadja_Said_2019} for a similar approach. We have the following estimate for this functional. 
\begin{proposition}\label{Lemma_F_3}
     	For any $\epsilon_{4}$, $\epsilon_{5}$, $\epsilon_{6}>0$, it holds
     	\begin{equation} \label{dF_dt_3_k}
\begin{aligned}  
&\dfrac{\textup{d}}{\dt}F_{3}^{(\kappa)}(t)+(\tau(c^2-c_g^2)-\epsilon_{4}g_{0}-\epsilon_{6}(c^2-c_g^2))\Vert \nabla^{\kappa+1} v\Vert_{L^2}^2\\
\leq &\,\begin{multlined}[t] \epsilon_{5}(c^2-c_g^2)\Vert \nabla^{\kappa+1}(v+\tau w)\Vert_{L^2}^2 + C(\epsilon_{4})\Vert \nabla^{\kappa+1}\eta\Vert^{2}_{L^2, -g^\prime}  
 \\+C(\epsilon_{5},\epsilon_{6})\Vert \nabla^{\kappa+1}\eta\Vert^{2}_{L^2, g},\end{multlined} 
\end{aligned}
\end{equation}
	where $g_0=g(0)$. 
     \end{proposition}
\begin{proof}
We prove the case $\kappa=0$; the general case $\kappa \geq 0$ follows analogously. From the second equation in \eqref{Main_System}, we have by applying the Laplacian 
\begin{equation}\label{Lapla_Eq}
\Delta v_t=\Delta w. 
\end{equation}
Multiplying the above equation by $\tau\int_{0}^{\infty}g(r)\eta(r)\ds$ and integrating over $\R^n$, using integration by parts, we get
\begin{equation}
\tau\int_{\R^n}\int_{0}^{\infty}g(r)\nabla\eta(r)\cdot\nabla v_{t} \ds\dx =	\tau\int_{\R^n}\int_{0}^{\infty}g(r)\nabla\eta(r)\cdot\nabla w \ds\dx. 
\end{equation}
Consequently,
\begin{equation}
\begin{aligned}
&\tau\dfrac{\textup{d}}{\dt}\int_{\R^n}\int_{0}^{\infty}g(r)\nabla\eta(r)\cdot\nabla v \ds\dx-\tau\int_{\R^n}\int_{0}^{\infty}g(r)\nabla\eta_{t}(r)\cdot \nabla v \ds \dx\\
=& \, 
     	\tau\int_{\R^n}\int_{0}^{\infty}g(r)\nabla\eta(r)\cdot\nabla w \ds\dx.
\end{aligned}     	
\end{equation} 
By using the fact that $\eta_{t}=v-\eta_{r}$, integrating by parts with respect to $r$, and keeping in mind that $\int_0^\infty g(r)\ds=c^2-c_g^2$,
 we obtain
     	\begin{equation}
     	\begin{aligned}
     	&\dfrac{\textup{d}}{\dt}F_{3}^{(0)}(t)+\tau(c^2-c_g^2)\Vert\nabla v\Vert_{L^2}^{2}\\
	=&\, \begin{multlined}[t]-\tau\int_{\R^n}\int_{0}^{\infty}g^{\prime}(r)\nabla\eta(r)\cdot \nabla v \ds\dx-\int_{0}^{\infty}g(r)\nabla\eta(r)\cdot \nabla (v+\tau w) \ds\dx \\
     	+\int_{\R^n}\int_{0}^{\infty}g(r)\nabla\eta(r)\cdot\nabla v \ds\dx.\end{multlined}
     	\end{aligned}
     	\end{equation}
 By applying Young's inequality,  we obtain \eqref{dF_dt_3_k}. This completes the proof.    
\end{proof}
In the critical case $ b=\tau c^2$,  the term $\Vert \Delta\nabla^\kappa v \Vert_{L^2}^2$ is not present in estimate \eqref{dE_1_Dt_k} and, therefore, also not in \eqref{dE_1_Dt}.  To restore this term, we define the functional 
\begin{equation}\label{F_4_Function_k}
F_4^{(\kappa)}(t)=-\tau \int_{\R^n} \int_{0}^{\infty} g(r)\Delta\nabla^\kappa\eta(r) \Delta \nabla^\kappa v \ds \dx,
\end{equation}
which satisfies the following bound. 
\begin{lemma}\label{Lemma_F_4}
For any $\epsilon_{4}$, $\epsilon_{5}$, $\epsilon_{6}>0$, the following bound holds:
\begin{equation}\label{F_4_Estimate_k}
\begin{aligned}
&\dfrac{\textup{d}}{\dt}F_4^{(\kappa)}(t)+\Big[\tau(c^2-c_g^2)-\epsilon_{4} g_{0}-\epsilon_{6} (c^2-c_g^2))\Big]\Vert\Delta\nabla^\kappa v\Vert_{L^2}^{2}\\
\leq&\, \epsilon_{5}(c^2-c_g^2)\Vert \nabla^{\kappa+1}(v+\tau w)\Vert_{L_2}^2 + C(\epsilon_{4})\Vert \Delta \nabla^\kappa\eta\Vert^{2}_{L^2, -g^\prime}\\
&+C(\epsilon_5) \Vert \Delta \nabla^{\kappa+1}\eta\Vert^{2}_{L^2, g} + C(\epsilon_6) \Vert \Delta \nabla^\kappa\eta\Vert^{2}_{L^2, g}. 
\end{aligned}
\end{equation}
\end{lemma}
\begin{proof}
We present the proof for $\kappa=0$; the general case $\kappa \geq 0$ follows analogously. Multiplying  \eqref{Lapla_Eq} by $\tau\int_{0}^{\infty}g(r)\Delta\eta(r)\ds,$ and integrating over $\R^n$, using the fact that 
\begin{eqnarray*}
\tau\int_{\R^n}\int_{0}^{\infty}g(r)\Delta\eta(r)\Delta v_{t} \ds\dx =	\tau\int_{\R^n}\int_{0}^{\infty}g(r)\Delta\eta(r)\Delta w \ds\dx, 
\end{eqnarray*}
we get 
\begin{eqnarray*}
&&\tau\dfrac{\textup{d}}{\dt}\int_{\R^n}\int_{0}^{\infty}g(r)\Delta\eta(r)\Delta v \ds\dx-\tau\int_{\R^n}\int_{0}^{\infty}g(r)\Delta\eta_{t}(r)\Delta v \ds \dx\\
&=&  
     	\tau\int_{\R^n}\int_{0}^{\infty}g(r)\Delta\eta(r)\Delta w \ds\dx.
\end{eqnarray*}
Hence, this yields 
\begin{equation}
\begin{aligned}\label{F_4_First_Term}
     	&-\tau\dfrac{\textup{d}}{\dt}\int_{\R^n}\int_{0}^{\infty}g(r)\Delta\eta(r)\Delta v \ds\dx+\tau(c^2-c_g^2)\Vert\Delta v\Vert_{L^2}^{2}\\
	=&\,\begin{multlined}[t]-\tau\int_{\R^n}\int_{0}^{\infty}g^{\prime}(r)\Delta\eta(r) \Delta v \ds\dx-\int_{\R^n}\int_{0}^{\infty}g(r)\Delta\eta(r) \Delta (v+\tau w) \ds\dx \\
     	+\int_{\R^n}\int_{0}^{\infty}g(r)\Delta\eta(r)\Delta v \ds\dx.\end{multlined}
	\end{aligned}
     	\end{equation}
Applying Young's inequality yields the desired bound.  	 
\end{proof}
\subsection{The Lyapunov functional and loss of regularity} 
 We are now ready to define  the Lyapunov functional of order zero as
\begin{equation}
\begin{aligned}\label{Lyapunov_F_crit_0}
\mathscr{F}^{(0)}(t)=&\,N_0 (E_1^{(0)}(t)+E_2^{(0)}+E_2^{(1)}(t)+\varepsilon\tau \Vert w\Vert^2 _{L^{2}})\\
&+F_1^{(0)}(t)+2F_2^{(0)}(t)+N_1(F_3^{(0)}(t)+F_4^{(0)}(t)), 
\end{aligned}    
\end{equation}
for $t \geq 0$, where $N_0,\, N_1$ are large positive constants  and $\varepsilon$ is a small positive constant that will be fixed below. We claim that this Lyapunov functional can be made equivalent to $\mathbf{E}_0^2$, where the energy $\mathbf{E}_0$ is defined in \eqref{energy_mathbfE}. 
\begin{proposition}
\label{Lemma_Equivl_2} Let $b\geq \tau c^2>\tau c^2_g$. There exist positive constants $C_{1}$ and $C_{2}$, 
such that 
\begin{equation}  \label{Eq_L_E}
C_{1}\mathbf{E}_0^2(t)\leq \mathscr{F}^{(0)}(t)\leq C_{2}\mathbf{E}_0^2(t), \quad t \geq 0,
\end{equation}
provided that the constant $N_0$ in the Lyapunov functional \eqref{Lyapunov_F_crit_0} is chosen to be large enough. Furthermore, the following estimate holds:
\begin{equation}
\begin{aligned}
\label{F_Lyap}
&\frac{\textup{d}}{\dt}\mathscr{F}^{(0)}(t)+\Vert \nabla \eta\Vert _{L^2, -g'}^{2}+\Vert \Delta \eta
\Vert _{L^2, -g'}^{2}+\Vert \Delta \nabla\eta
\Vert _{L^2, -g'}^{2}+\|w\|^2_{L^2}\\
&+\Vert \Delta (\psi+\tau v)\Vert _{L^{2}}^{2} +\Vert\nabla(v+\tau w)\Vert^{2}_{L^2}+\Vert \nabla v\Vert_{L^2}^2+\Vert\Delta v\Vert_{L^2}^{2}\\
\lesssim &\,|R_0^{(2)}(\tau v)|
+|R_0^{(2)}(\psi + \tau v)|+|R_0^{(1)}(w)|\\
&+|R_0^{(1)}(v + \tau w)|+|R_0^{(2)}(v + \tau w)| +|R^{(2)}_1(\nabla  (v + \tau  w))|.
\end{aligned} 
\end{equation}
\end{proposition}
\begin{proof}
The proof of the equivalence follows along the lines of \cite[Lemma 4.5]{nikolic2020mathematical}, so we omit it here and focus on proving estimate \eqref{F_Lyap} instead. To this end, we take the time derivative of the Lyapunov functional \eqref{Lyapunov_F_crit_0}. Note that setting $\kappa=1$ in \eqref{dE_1_Dt_k}  yields the following estimate:
\begin{equation}  \label{dE_1_Dt_k_Crit}
\frac{\textup{d}}{\dt}E_2^{(1)}(t)+\frac{1}{2}\Vert \Delta \nabla\eta
\Vert _{L^2, -g'}^{2}\leq \, |R^{(2)}_1(\nabla  (v + \tau  w))|.   
\end{equation}%
By making use of energy estimates \eqref{dE_1_Dt} and \eqref{dE_1_Dt_k}, the above inequality, the derived bounds on $F^{(\kappa)}_{1, \ldots, 4}$, and assumption (G3) on the memory kernel, we infer 
\begin{equation}
\begin{aligned}
& \begin{multlined}[t]\frac{\textup{d}}{\dt}\mathscr{F}^{(0)}(t)+C_\eta\left[\Vert \nabla \eta\Vert _{L^2, -g'}^{2}+\Vert \Delta \eta
\Vert _{L^2, -g'}^{2}+\Vert \Delta \nabla\eta
\Vert _{L^2, -g'}^{2}\right]
+C_{(\psi+\tau v)}\Vert \Delta (\psi+\tau v)\Vert _{L^{2}}^{2}\\
+C_{(v+\tau w)}\Vert\nabla(v+\tau w)\Vert^{2}_{L^2}+C_{(\nabla v)}\Vert \nabla v\Vert_{L^2}^2+C_{(\Delta v)}\Vert \Delta v\Vert_{L^2}^2++C_w\|w\|^2_{L^2} \end{multlined}\\
\leq&\, \begin{multlined}[t] \Lambda_1\left(|R_0^{(2)}(\tau v)|
+|R_0^{(2)}(\psi + \tau v)|+|R_0^{(1)}(w)|+|R_0^{(1)}(v + \tau w)|\right. \\
\left.+|R_0^{(2)}(v + \tau w)| +|R^{(2)}_1(\nabla  (v + \tau  w))| \right),  \end{multlined}
\end{aligned}
\end{equation}
with the constants defined as
 \begin{equation}
 \left\{
\begin{array}{ll}
C_\eta= N_0(\frac{1}{2}-C\varepsilon)-\Lambda_0,\vspace{0.2cm}\\
C_w=\frac{N_0}{2}\varepsilon,\vspace{0.2cm}\\
C_{(\psi+\tau v)}=(c^2_g -\epsilon _{0}-(c^2 -c^2_g )\epsilon
_{1})-CN_0\varepsilon-2\epsilon_2,\vspace{0.2cm}\\
C_{(v+\tau w)}=(1-\epsilon_{3})-2N_1\epsilon_{5}(c^2-c_g^2),\vspace{0.2cm}\\
C_{(\nabla v)}=N_1\Big(\tau(c^2-c_g^2)-\epsilon_{4}g_{0}-\epsilon_{6}(c^2-c_g^2)\Big)-4C(\epsilon_2,\epsilon_3),\vspace{0.2cm}\\
C_{(\Delta v)}=N_1\Big(\tau(c^2-c_g^2)-\epsilon_{4}g_{0}-\epsilon_{6} (c^2-c_g^2)\Big)-2C(\epsilon_2,\epsilon_3)-C(\epsilon_0)-CN_0\varepsilon. 
\end{array}  
\right.
\end{equation}
Above, $\Lambda_0$ and $\Lambda_1$ are positive constants that may depend on the parameters $\epsilon_0,\epsilon_1, N_1,\dots$ The constant $\Lambda_0$ depends on $\zeta$, yet it is independent of $N_0$ and $\varepsilon$. Furthermore, the constant $\Lambda_1$ depends on $N_0$, yet it is independent of  $\zeta$. \\
\indent We can fix our parameters in such a way that $C_\eta,\dots,C_{(\Delta v)}$
are positive. Indeed, we first take $\epsilon_{0}=\epsilon_{1}$ and pick $\epsilon_{1}>0$ small enough, such that 
$$\epsilon_{1}<\dfrac{c_g^2}{1+(c^2-c_g^2)}.$$
Once $\epsilon_{0}$ and $\epsilon_{1}$ are fixed, we select  $\epsilon_{2}>0$ small enough so that
$$\epsilon_{2}<\dfrac{c_g^2-\epsilon_{0}(1+(c^2-c_g^2))}{2}.$$
Now we pick $\epsilon_{4}=\epsilon_{6}$ so that
$$0<\epsilon_{4}=\epsilon_{6}<\dfrac{\tau(c^2-c_g^2)}{g_{0}+(c^2-c_g^2)}.$$
We choose $\epsilon_{3}<\frac{1}{4}$ and we take $N_1$ large enough such that
$$N_1>\dfrac{C(\epsilon_{0})+2C(\epsilon_{2},\epsilon_{3})}{\tau(c^2-c_g^2)-\epsilon_{4}(g_{0}+(c^2-c_g^2))}.$$
Then we can select $\epsilon_{5}$ small enough such that 
$$ \epsilon_{5}<\dfrac{(1-\epsilon_{3})}{2N_1(c^2-c_g^2)}.$$
We next take $N_0$ large enough such that 
$N_0>2\Lambda_0.$
Finally, we fix $\varepsilon>0$ small enough such that 
\begin{eqnarray*}
\varepsilon<\min\left(\frac{N_0-2\Lambda_0}{2C},\frac{N_1\Big(\tau(c^2-c_g^2)-\epsilon_{4}g_{0}-\epsilon_{6} (c^2-c_g^2)\Big)-2C(\epsilon_2,\epsilon_3)-C(\epsilon_0)}{CN_0} \right).
\end{eqnarray*}
In this manner, we have arrived at estimate \eqref{F_Lyap}.
 \end{proof}
\section{Introduction of the artificial damping} \label{Sec:ArtificialDamping}
We next intend to derive a low-order energy estimate of a regularity-loss type for our problem. To this end, we introduce artificial damping to the system by considering time-weighted energies with a negative exponent.
\begin{proposition}\label{Lemma_Energy_Order_0}
Let $s_0 \geq n/2$ be an integer. The following bound holds:
\begin{equation}   
\begin{aligned}
&(1+t)^{-1/2}\mathbf{E}^{(0)}(t)+\Upsilon^{(0, 0)}(t)\\[0.1cm]
\lesssim&\,\begin{multlined}[t] \mathbf{E}^{(0)}(0)
+\Big(\mathdutchcal{M}[v, w, \vecc{U}](t)
+M_0[\boldsymbol{U}]+M_0[v]\Big)\Upsilon^{(0, 0)}(t), \end{multlined}
\end{aligned}
\end{equation}
where the energy $\mathbf{E}^{(0)}$ is defined in \eqref{energy_mathbfE}, the dissipative term $\mathbf{D}^{(0)}$ in \eqref{energy_mathbfD}, and
\begin{equation}
\Upsilon^{(0, 0)}(t)= \int_0^t (1+  \sigma)^{-3/2}\mathbf{E}^{(0)}(\sigma) \, \textup{d}\sigma+\int_0^t (1+\sigma)^{-1/2}\mathbf{D}^{(0)}(\sigma)\, \textup{d}\sigma;
\end{equation}
see \eqref{D_k_dissipation} below for the general definition of $\Upsilon^{j,\kappa}$.
\end{proposition}
\begin{proof}
Multiplying estimate \eqref{F_Lyap} by $(1+\sigma)^{\gamma}$ and integrating from $0$ to $t$ yields 
\begin{equation}    \label{Main_Estimate_0}
\begin{aligned} 
&(1+t)^\gamma\mathbf{E}^{(0)}(t)+\int_0^t (1+\sigma)^{\gamma}\mathbf{D}^{(0)}(\sigma) \, \textup{d}\sigma\\
\lesssim&\, \mathbf{E}^{(0)}(0)+ \gamma \int_0^t (1+  \sigma)^{\gamma-1}\mathbf{E}^{(0)}(\sigma)  \, \textup{d}\sigma
+\int_0^t (1+\sigma)^{\gamma}\mathbf{R}_0(\sigma) \, \textup{d}\sigma, 
\end{aligned}
\end{equation} 
where we have set 
\begin{equation} \label{R0}
\begin{aligned}
\mathbf{R}_0 = \begin{multlined}[t]|R_0^{(2)}(\tau v)|
+|R_0^{(2)}(\psi + \tau v)|+|R_0^{(1)}(w)|+|R_0^{(1)}(v + \tau w)| \\
+|R_0^{(2)}(v + \tau w)| +|R^{(2)}_1(\nabla  (v + \tau  w))|. \end{multlined} 
\end{aligned}
\end{equation}
To derive the above bound, we have used Lemma \ref{Lemma_Equivl_2}, i.e.,
\begin{equation}\label{eq_F_0_E}
\mathscr{F}_0(t)\sim \mathbf{E}^{(0)}(t)
\end{equation}
and that
\begin{equation}
(1+t)^{\gamma}\frac{\textup{d}}{\dt}\mathscr{F}^{(0)}(t)=\frac{\textup{d}}{\dt} (1+t)^{\gamma}\mathscr{F}^{(0)}-\gamma (1+t)^{\gamma-1} \mathscr{F}^{(0)}.
\end{equation}
The last term in estimate \eqref{Main_Estimate_0} contains the ``problematic"  term 
\begin{eqnarray*}
\int_0^t (1+\sigma)^{\gamma}|R^{(2)}_1(\nabla  (v + \tau  w))(\sigma)|\, \textup{d}\sigma. 
\end{eqnarray*}
The above term is responsible for the regularity loss. Keeping in mind the estimate \eqref{R_2_kappa_New}, it can be treated as follows: 
\begin{eqnarray*}
\int_0^t (1+\sigma)^{\gamma}|R^{(2)}_1(\nabla  (v + \tau  w))(\sigma)|\, \textup{d}\sigma \lesssim (M_0[\vecc{U}](t)+M_0[v](t)) \int_0^t (1+\sigma)^{\gamma} \mathbf{D}_1^2(\sigma)\, \textup{d}\sigma.    
\end{eqnarray*}
However, the term $\int_0^t (1+\sigma)^{\gamma} \mathbf{D}_1^2(\sigma)\, \textup{d}\sigma$ cannot be absorbed by $\int_0^t (1+\sigma)^{\gamma}\mathbf{D}_0^2(\sigma)\, \textup{d}\sigma$ on the left-hand side of \eqref{Main_Estimate_0} due to the loss of derivatives going from $\mathbf{D}_1^2$ to $\mathbf{D}_0^2$.  So, the classical energy method fails. To overcome this difficulty and inspired by \cite{IK08}, we use a time-dependent weight with a negative exponent; see also~\cite{HKa06, Racke_Said_2012_1}.  In other words, we take $\gamma<0$ in estimate \eqref{Main_Estimate_0} to obtain an artificial damping term \[-\gamma \int_0^t (1+  \sigma)^{\gamma-1}\mathbf{E}^{(0)}(\sigma) \, \textup{d}\sigma\] on the left-hand side. This damping allows us to control the term $\int_0^t (1+\sigma)^{\gamma} \mathbf{D}_1^2(\sigma)\, \textup{d}\sigma$. \\
\indent  Indeed, by taking  $\gamma=-1/2$, we have
\begin{equation}   
\label{Main_Estimate_0_1}
(1+t)^{-1/2}\mathbf{E}^{(0)}(t)+\Upsilon^{(0,0)}(t)
\lesssim \mathbf{E}^{(0)}(0)
+\int_0^t (1+  \sigma)^{-1/2}\mathbf{R}  _0(\sigma)\, \textup{d}\sigma.
\end{equation}
It remains to estimate the six terms contained within the $\mathbf{R}_0$ term on the right; see \eqref{R0} for its definition. First we have
 \begin{equation}
\begin{aligned}
|R_1^{(2)}(\nabla(v+\tau w))|
&\leq \Vert \nabla F^{(1)}\Vert _{L^{2}}%
\Vert\nabla^2\left( v+\tau w\right)\Vert_{L^2},
\end{aligned}  
\end{equation}
where we also recall how $F^{(1)}$ is defined in \eqref{F_k_Form}. We can bound this term as follows:
\begin{equation} \label{R_1_k_Nabla}
\begin{aligned}
\Vert \nabla F^{(1)}\Vert _{L^{2}} \lesssim&\, \begin{multlined}[t] \Vert w\Vert
_{L^{\infty }}\Vert \nabla ^{2}v\Vert _{L^{2}}+\Vert v\Vert _{L^{\infty
}}\Vert \nabla ^{2}w\Vert _{L^{2}}   \\
+ \Vert \nabla \psi\Vert _{L^{\infty }}\Vert \nabla ^{3}v\Vert 
_{L^{2}}+\Vert \nabla v\Vert _{L^{\infty }}\Vert \nabla ^{3}\psi\Vert
_{L^{2}}. \end{multlined}
\end{aligned}
\end{equation}
Setting $\vecc{U} =(v+\tau w,\nabla(\psi+\tau v),\nabla v)^T$, we have
\begin{equation}
\begin{aligned}
\|v\|_{L^\infty}+\Vert w\Vert_{L^\infty}\lesssim&\, \Vert v+\tau w \Vert_{L^\infty}+\Vert v\Vert_{L^\infty}\\
\lesssim& \,(1+t)^{-n/2}\big(M_0[\vecc{U}](t)+M_0[v](t)\big).
\end{aligned}
\end{equation}
Similarly, it holds
\begin{equation}
\begin{aligned}
\Vert \nabla \psi\Vert _{L^{\infty }}+\Vert \nabla v\Vert _{L^{\infty }}\lesssim&\,  (1+t)^{-n/2} \sup_{0\leq \sigma\leq t}\left( 1+\sigma \right) ^{%
	\frac{n}{2}}\left\Vert \vecc{U}\left( \sigma \right) \right\Vert _{L^{\infty }}\\
	\lesssim&\,(1+t)^{-n/2}M_0[\vecc{U}](t).
	\end{aligned}
\end{equation}
We obtain from above
\begin{equation}\label{R_1_2_Estimate}  
\begin{aligned}  
&\int_0^t (1+\sigma)^{-1/2}|R^{(2)}_1(\nabla  (v + \tau  w))(\sigma)|\, \textup{d}\sigma\\
\lesssim&\,\, \big(M_0[\vecc{U}](t)+M_0[v](t)\big)\\  
&\times \int_0^t (1+\sigma)^{-\frac{n+1}{2}}\Big[\Vert \nabla ^{2}v\Vert _{H^{1}}+\Vert \nabla ^{2}w\Vert _{L^{2}}+\Vert \nabla ^{3}\psi\Vert  
_{L^{2}}\Big]\Vert\nabla^2\left( v+\tau w\right)\Vert_{L^2}\, \textup{d}\sigma\\
\lesssim&\, \big(M_0[\vecc{U}](t)+M_0[v](t)\big)\Upsilon^{(0,0)}(t). 
\end{aligned}
\end{equation}
Next we wish to estimate the $R_0^{(1)}$ terms within \eqref{Main_Estimate_0_1} on the right. We have
\begin{equation}  \label{R_1_Estimate}
\begin{aligned}
&\int_0^t | R^{(1)}_0(v+\tau w)(\sigma)|\, \textup{d}\sigma+\int_0^t| R^{(1)}_0(w)(\sigma)|\, \textup{d}\sigma\\
\lesssim&\,
\sup_{0\leq \sigma\leq t}\Big(\Vert \nabla \psi(\sigma)\Vert_{L^\infty}+\Vert\nabla \psi(\sigma)\Vert_{\dot{H}^{\frac{n-2}{2}}}+\Vert v(\sigma)\Vert_{\dot{H}^{\frac{n-2}{2}}}\Big)%
\mathbf{D}^2_0(t).
\end{aligned}
\end{equation}
We note that
\begin{equation}\label{v_hom_Est}
\begin{aligned}
\Vert v(\sigma)\Vert_{\dot{H}^{\frac{n-2}{2}}}\lesssim&\, \Vert v(\sigma)\Vert_{H^{\frac{n-2}{2}}}
\lesssim\, \begin{multlined}[t](1+t)^{-\frac{n}{4}}\mathdutchcal{M}[v, w, \vecc{U}](t),
\end{multlined}
\end{aligned}
\end{equation}
and 
\begin{equation}\label{psi_hom_Est}
\begin{aligned}
\Vert\nabla \psi(\sigma)\Vert_{\dot{H}^{\frac{n-2}{2}}}\lesssim&\, \Vert\nabla \psi(\sigma)\Vert_{H^{\frac{n-2}{2}}}
\lesssim\,  \begin{multlined}[t](1+t)^{-\frac{n}{4}} \mathdutchcal{M}[v, w, \vecc{U}](t).\end{multlined}
\end{aligned}
\end{equation} 
For the above bounds to hold, it is crucial that $\frac{n-2}{2}\leq s_0-1$. By making use of \eqref{M_j_Def}, \eqref{v_hom_Est} and \eqref{psi_hom_Est}, and  
applying estimate \eqref{R_1_Estimate}, we have  
\begin{equation}\label{R_1_First_Term_}
\begin{aligned}  
&\int_0^t (1+\sigma)^{-1/2}\big(| R_0^{(1)}(v+\tau w)(\sigma)|+| R_0^{(1)}(w)(\sigma)|\big)\, \textup{d}\sigma\\
\lesssim &\,\Big(\mathdutchcal{M}[v, w, \vecc{U}](t)+M_0[v](t)\Big)
\int_0^t (1+\sigma)^{-\frac{n}{4}-\frac{1}{2}}\mathbf{D}^2_0(\sigma)\, \textup{d}\sigma, 
\end{aligned}
\end{equation}
from which we further derive
\begin{equation}\label{R_1_First_Term}
\begin{aligned}  \
&\int_0^t (1+\sigma)^{-1/2}\big(| R_0^{(1)}(v+\tau w)(\sigma)|+| R_0^{(1)}(w)(\sigma)|\big)\, \textup{d}\sigma\\
\lesssim &\,\Big(\mathdutchcal{M}[v, w, \vecc{U}](t)+M_0[v](t)\Big)
\Upsilon^{(0,0)}(t). 
\end{aligned}
\end{equation}
It remains to estimate the $R_0^{(2)}$ terms within \eqref{Main_Estimate_0_1} on the right. We note that
\begin{equation}  \label{R_2_Estimate}
\begin{aligned}
 |R_0^{(2)}(v+\tau w)|\leq& \, \begin{multlined}[t]\sup_{0\leq \sigma\leq t}\Big(\left\Vert
v(\sigma)\right\Vert _{L^{\infty }}+\left\Vert\nabla
v(\sigma)\right\Vert _{L^{\infty }}\Big.\\
\Big.+\left\Vert (v+\tau w)(\sigma)\right\Vert
_{L^{\infty }}+\left\Vert \nabla (\psi+\tau v)(\sigma)\right\Vert _{L^{\infty }}\Big)%
\mathbf{D}_0^2(t).\end{multlined}
\end{aligned}  
\end{equation}
We then have  
\begin{equation}\label{R_2_Terms_1}
\begin{aligned}
\int_0^t (1+\sigma)^{-1/2}| R_0^{(2)}(v+\tau w)(\sigma)|\, \textup{d}\sigma 
\lesssim&\, \big(M_0[\vecc{U}](t)+M_0[v](t)\big) \int_0^t (1+\sigma)^{-\frac{n}{2}-\frac{1}{2}} \mathbf{D}_0^2(\sigma) \, \textup{d}\sigma \\
\lesssim&\, \big(M_0[\vecc{U}](t)+M_0[v](t)\big)\Upsilon^{(0,0)}(t). 
\end{aligned}
\end{equation}
Similarly, we can show that 
\begin{equation}\label{R_2_Second_Term}
\begin{aligned}
&\int_0^t (1+\sigma)^{-1/2}|R_0^{(2)}((\psi+\tau v))(\sigma)|+|R_0^{(2)}(\tau v)(\sigma)|\, \textup{d}\sigma\\
&\lesssim \big(M_0[\vecc{U}](t)+M_0[v](t)\big)  \Upsilon^{(0,0)}(t). 
\end{aligned}
\end{equation}
Now, by collecting the derived bounds, we deduce that 
  \begin{equation}\label{R_0_Main_Estimate}
  \begin{aligned}
\int_0^t (1+\sigma)^{-1/2}\mathbf{R}_0(\sigma)\, \textup{d}\sigma
\lesssim&\, \begin{multlined}[t] \Big(\mathdutchcal{M}[v, w, \vecc{U}](t) +M_0[\vecc{U}](t)+M_0[v](t)\Big)\Upsilon^{(0,0)}(t). \end{multlined}
\end{aligned}
\end{equation}
Plugging this bound into \eqref{Main_Estimate_0_1} completes the proof.
\end{proof}  
\subsection{The Lyapunov functional of higher order}
Our next aim is to extend the previous considerations to higher-order energies. We thus retrace our previous steps, but now adapted to space-differentiated system \eqref{Main_System_kappa}. We define the Lyapunov functional of order $\kappa \geq 1$ analogously as
\begin{equation}\label{Lyapunov_F_crit}
\begin{aligned}
\mathscr{F}^{(\kappa)}(t)=& \, \begin{multlined}[t]N_0 (E_1^{(\kappa)}(t)+E_2^{(\kappa)}+E_2^{(\kappa+1)}(t)+\varepsilon\tau \Vert w\Vert^2 _{L^{2}})\vspace{0.2cm}\\+F_1^{(\kappa)}(t)+2F_2^{(\kappa)}(t)+N_1(F_3^{(0)}(t)+F_4^{(\kappa)}(t)),  \end{multlined}
\end{aligned}   
\end{equation}
for $t \geq 0$. We claim that a higher-order version of estimate \eqref{Main_Estimate_0} holds as well.
 \begin{proposition}\label{Pro_Higher_Order}
	Assume that $b=\tau c^2$. Then, for any integer $\kappa\geq 1$,  it holds  
	\begin{equation}
	\begin{aligned}
	\label{Main_Estimate_k}
	&(1+t)^\gamma\mathbf{E}^{(\kappa)}(t)+\int_0^t (1+\sigma)^{\gamma}\mathbf{D}^{(\kappa)}(\sigma)\, \textup{d}\sigma\\
	 \lesssim& \, \begin{multlined}[t]\mathbf{E}^{(\kappa)}(0)+ \gamma \int_0^t (1+\sigma)^{\gamma-1}\mathbf{E}^{(\kappa)}(\sigma) \, \textup{d}\sigma
	+\int_0^t (1+\sigma)^{\gamma}\mathbf{R}_\kappa(\sigma)\, \textup{d}\sigma.\end{multlined}
	\end{aligned} 
	\end{equation}
\end{proposition}
 \begin{proof}
It can be shown that $\mathscr{F}^{(\kappa)}(t)$ is equivalent to $\mathbf{E}^{(\kappa)}$; see \cite[Lemma 4.5]{nikolic2020mathematical} for a similar analysis when $\kappa=0$.  By taking the time derivative of $\mathscr{F}^{(\kappa)}(t)$, making use of Proposition \ref{Prop:E1_k}, and the derived bounds on $F^{(\kappa)}_{1, \ldots, 4}$, we obtain 
\begin{equation}
\begin{aligned}
\label{F_Lyap_k}
&\begin{multlined}[t]\frac{\textup{d}}{\dt}\mathscr{F}^{(\kappa)}(t)+\Vert \nabla^{\kappa+1} \eta\Vert _{L^2, -g'}^{2}+\Vert \Delta\nabla^\kappa \eta
\Vert _{L^2, -g'}^{2}+\Vert \Delta \nabla^{\kappa+1}\eta
\Vert _{L^2, -g'}^{2}+\|\nabla^\kappa w\|^2_{L^2}\\
+\Vert \Delta \nabla^\kappa(\psi+\tau v)\Vert _{L^{2}}^{2} +\Vert\nabla^{\kappa+1}(v+\tau w)\Vert^{2}_{L^2}+\Vert \nabla^{\kappa+1} v\Vert_{L^2}^2+\Vert\Delta \nabla^\kappa v\Vert_{L^2}^{2}\end{multlined}\\
\lesssim& \, \begin{multlined}[t]|R_\kappa^{(2)}(\tau \nabla^\kappa v)|
+|R_\kappa^{(2)}(\nabla^\kappa(\psi + \tau v))|+|R_\kappa^{(1)}(\nabla ^\kappa w)|\\
	+|R_\kappa^{(1)}(\nabla^\kappa(v + \tau w))|+|R_\kappa^{(2)}(\nabla^\kappa(v + \tau w))| +|R^{(2)}_{\kappa+1}(\nabla^{\kappa+1}  (v + \tau  w))|\end{multlined}
	\end{aligned} 
\end{equation}
provided that the constant $N_0$ in the Lyapunov functional is chosen to be large enough. Multiplying the above estimate by $(1+\sigma)^{\gamma}$, integrating from $0$ to $t$, and relying on the fact that $\mathscr{F}^{(\kappa)}(t)$ is equivalent to $\mathbf{E}^{(\kappa)}$ yield \eqref{Main_Estimate_k}. 
\end{proof}
\subsection{Estimates of the right-hand side terms of higher order}\label{Section_Induction}
The main challenge in deriving a higher-order version of Proposition~\ref{Lemma_Energy_Order_0} is to control the right-hand side terms contained within $\int_0^t (1+\sigma)^{\gamma}\mathbf{R}_\kappa(\sigma)\, \textup{d}\sigma$ in estimate \eqref{Main_Estimate_k}. \\
\indent In particular, we set $\gamma=j-1/2$ for an integer $j \geq 0$.
  \begin{theorem}\label{Thm_R_k}
Let $s_0\geq [n/2]+2$
  in the definition of the function $\mathdutchcal{M}[v, w, \vecc{U}](t)$. For any integers $\kappa\geq 1$ and $j\geq 0$, the following estimate holds:
	\begin{equation}\label{R_j_Estimates_k}
	\begin{aligned}  
	& \int_0^t (1+\sigma)^{j-1/2}\mathbf{R}_\kappa(\sigma)\, \textup{d}\sigma \\
		\lesssim&\, \mathdutchcal{M}[v, w, \vecc{U}](t)\Big(\Upsilon^{(j,\kappa)}(t)+\int_0^t (1+\sigma)^{j-3/2}\mathbf{D}^{\kappa-1}(\sigma)\textup{d}\sigma\Big),
	\end{aligned}
	\end{equation}
where 
\begin{equation}
\begin{aligned}
\mathbf{R}_\kappa=\begin{multlined}[t]|R_\kappa^{(2)}(\tau \nabla^\kappa v)|
+|R_\kappa^{(2)}(\nabla^\kappa(\psi + \tau v))|+|R_\kappa^{(1)}(\nabla ^\kappa w)|\\
+|R_\kappa^{(1)}(\nabla^\kappa(v + \tau w))|+|R_\kappa^{(2)}(\nabla^\kappa(v + \tau w))| +|R^{(2)}_{\kappa+1}(\nabla^{\kappa+1}  (v + \tau  w))|\end{multlined}
\end{aligned}
\end{equation}	
and
\begin{equation}\label{D_k_dissipation}
\Upsilon^{(j,\kappa)}(t)= \int_0^t (1+  \sigma)^{j-3/2}\mathbf{E}^{(\kappa)}(\sigma) \, \textup{d}\sigma+\int_0^t (1+\sigma)^{j-1/2}\mathbf{D}^{(\kappa)}(\sigma)\, \textup{d}\sigma. 
\end{equation}
\end{theorem}
We provide the proof through several steps, in which we estimate the terms contained within $\int_0^t (1+\sigma)^{j-1/2}\mathbf{R}_\kappa(\sigma)\, \textup{d}\sigma$ under the assumptions of Theorem~\ref{Thm_R_k}.
\newtheorem*{prop1}{Step I}
\begin{prop1}\label{Lemma_R_1_k}
The following estimate holds:
\begin{equation}
\begin{aligned}  \label{R_1_Estimate_k}
&\int_0^t (1+\sigma)^{j-1/2}| R_\kappa^{(1)}\nabla^\kappa(v + \tau w)(\sigma)|\, \textup{d}\sigma+\int_0^t (1+\sigma)^{j-1/2}| R_\kappa^{(1)}\nabla^\kappa(\tau w)(\sigma)|\, \textup{d}\sigma \\ 
\lesssim& \,  \mathdutchcal{M}[v, w, \vecc{U}](t)\Upsilon^{(j,\kappa)}(t).
\end{aligned}   
\end{equation}  
\end{prop1}
\begin{proof}
We first note that the term $R_\kappa^{(1)}$ can be written out as
\begin{equation}
\begin{aligned}
& R_\kappa^{(1)}(\nabla^{\kappa}(v+\tau w))\\
=&\,  (F^{(\kappa)}, \nabla^{\kappa}(v+\tau w))_{L^2} \\
=&\,  (2k[\nabla^\kappa,v]w+2k v\tilde{w}+2\kappa[\nabla^\kappa,\nabla \psi] \cdot \nabla v+2k\nabla \psi \cdot \nabla\tilde{w}, \ \tilde{v}+\tau \tilde{w})_{L^2},
\end{aligned}
\end{equation}
where we have used the short-hand notation $\tilde{v}=\nabla^\kappa v$ and $\tilde{w}=\nabla^\kappa w$, and the definition \eqref{F_k_Form} of the functional $F^{(\kappa)}$. Thus we can estimate this term as follows:
\begin{equation}\label{I_1_Terms_New}
\begin{aligned}
& \left| R_\kappa^{(1)}(\nabla^{\kappa}(v+\tau w)) \right|\\
 \lesssim&\, \begin{multlined}[t]\int_{\mathbb{R}^n} |[\nabla ^{k},v]w| | \tilde{v}+\tau \tilde{w}| \dx+\int_{%
\mathbb{R}^n} |v\tilde{w}| |\left( \tilde{v}+\tau \tilde{w}\right)|\dx \\
+\int_{\mathbb{R}^n}|[\nabla ^{\kappa},\nabla \psi]\nabla v||\tilde{v}+\tau
\tilde{w}|\dx + \int_{\mathbb{R}^n} |\nabla \psi||\nabla \tilde{v}|| \tilde{v}+\tau \tilde{w}|\dx\end{multlined}\\
 =: &\,\textup{I}_1+\textup{I}_2+\textup{I}_3+\textup{I}_4.
\end{aligned}
\end{equation}
We need to further bound each of the terms on the right. Starting from the last term, we have
\begin{equation}  \label{T_4}
\begin{aligned}
\mathrm{I}_4=  \int_{\mathbb{R}^n} |\nabla \psi||\nabla \tilde{v}|| \tilde{v}+\tau \tilde{w}|\dx  \leq&\, \int_{\mathbb{R}^n} |\tilde{v}||\nabla \psi||\nabla \tilde{v }|\dx+\int_{\mathbb{R}%
^n} \tau|\nabla \psi||\nabla \tilde{v}|| \tilde{w}| \dx \\
:=&\, \mathrm{I}_{4 \textup{a}}+ \mathrm{I}_{4 \textup{b}}.
\end{aligned}
\end{equation}
We can rely on the following bound:
\begin{equation}
 \mathrm{I}_{4 \textup{b}}  \lesssim \Vert \nabla \psi\Vert_{L^\infty}\Vert \nabla
\tilde{v}\Vert_{L^2}\Vert \tilde{w}\Vert_{L^2}
\end{equation}
to infer
\begin{equation} 
\begin{aligned}
&\int_0^t (1+\sigma)^{j-1/2} \mathrm{I}_{4b} \, \textup{d}\sigma\\
\lesssim&\, \int_0^t (1+\sigma)^{j-1/2}\Vert \nabla \psi\Vert_{L^\infty}\Vert \nabla
\tilde{v}\Vert_{L^2}\Vert \tilde{w}\Vert_{L^2} \, \textup{d}\sigma\\
\lesssim&\, \begin{multlined}[t]\sup_{0\leq \sigma\leq  t}\left( 1+\sigma \right) ^{%
	\frac{n}{2}}\left\Vert  \vecc{U}(\sigma)\right\Vert _{L^{\infty }} \int_0^t (1+\sigma)^{j-1/2-n/2}(\Vert \nabla \tilde{v}(\sigma)\Vert_{L^2}^2+ \Vert \tilde{w}(\sigma)\Vert_{L^2}^2)\, \textup{d}\sigma, \end{multlined}
\end{aligned}
\end{equation}
where we recall that $\vecc{U}=(v+\tau w, \nabla (\psi+\tau v), \nabla v)^T$. We note that the integral over time on the right can be estimated as follows:
\begin{equation}
\begin{aligned}
& \int_0^t (1+\sigma)^{j-1/2-n/2}(\Vert \nabla^{\kappa+1}v(\sigma)\Vert_{L^2}^2+ \Vert \nabla^{\kappa}w(\sigma)\Vert_{L^2}^2)\, \textup{d}\sigma \\
\lesssim&\,\Upsilon^{(j,\kappa)}(t)=\int_0^t (1+  \sigma)^{j-3/2}\mathbf{E}^{(\kappa)}(\sigma) \, \textup{d}\sigma+\int_0^t (1+\sigma)^{j-1/2}\mathbf{D}^{(\kappa)}(\sigma)\, \textup{d}\sigma,
\end{aligned}
\end{equation}
because $n/2+1/2\geq 3/2$. Therefore,
\begin{equation} \label{J_2_k_Main_New}
\begin{aligned}
\int_0^t (1+\sigma)^{j-1/2} \mathrm{I}_{4b} \, \textup{d}\sigma
\lesssim\, M_0[\vecc{U}](t)\Upsilon^{(j,\kappa)}(t).   
\end{aligned}
\end{equation}
\indent Next we  have 
\begin{equation}
\begin{aligned}
\mathrm{I}_{4 \textup{a}} \lesssim  \Vert \tilde{v}\Vert_{L^\frac{2n}{n-2}}\Vert \nabla \psi\Vert_{L^n}\Vert \nabla
\tilde{v}\Vert_{L^{2}}
\lesssim\, \Vert \nabla
\tilde{v}\Vert_{L^{2}}^2 \Vert \nabla \psi\Vert_{\dot{H}^{\frac{n-2}{2}}},
\end{aligned}
\end{equation}
where we have used the endpoint Sobolev embeddings 
\[
\begin{aligned}
\|\tilde{v}\|_{L^{\frac{2n}{n-2}}} \lesssim\,  \|\nabla \tilde{v}\|_{L^2}, \qquad
\Vert \nabla \psi\Vert_{L^n} \lesssim\,  \Vert \nabla \psi\Vert_{\dot{H}^{\frac{n-2}{2}}}, \qquad n \geq 3;
\end{aligned}
\]
cf. Lemma~\ref{Lemma:EndpointEmbedding}. Since $s_0\geq [\frac{n}{2}]+1 $, we have
  \begin{equation}\label{Psi_M_Estimate}
 \Vert \nabla \psi (\sigma)\Vert_{\dot{H}^{\frac{n-2}{2}}} \lesssim  \Vert \nabla \psi (\sigma)\Vert_{\dot{H}^{[\frac{n}{2}]}}\lesssim (1+\sigma)^{-n/4}\mathdutchcal{M}[v, w, \vecc{U}](t);
 \end{equation}
cf. \eqref{Def_Mcal}. Therefore,
 \begin{equation}\label{J_1_k_1_New}
 \begin{aligned}
\int_0^t (1+\sigma)^{j-1/2} \mathrm{I}_{4a}\, \textup{d}\sigma 
\lesssim& \, \begin{multlined}[t]\mathdutchcal{M}[v, w, \vecc{U}](t)\int_0^t (1+\sigma)^{j-1/2-n/4} \Vert \nabla^{\kappa+1}v(\sigma)\Vert_{L^2}^2 \, \textup{d}\sigma \end{multlined}\\
\lesssim&\, \mathdutchcal{M}[v, w, \vecc{U}](t)\Upsilon^{(j,\kappa)}(t).  
\end{aligned}  
\end{equation}
Consequently, we deduce from \eqref{J_2_k_Main_New} and \eqref{J_1_k_1_New} that 
\begin{equation}\label{T_4_Estimate_New}
\int_0^t (1+\sigma)^{j-1/2} I_4 \, \textup{d}\sigma\lesssim \mathdutchcal{M}[v, w, \vecc{U}](t)\Upsilon^{(j,\kappa)}(t) .
\end{equation}
\indent We wish to estimate term $\textup{I}_2$ within \eqref{I_1_Terms_New} next. By relying again on H\"older's inequality and the endpoint Sobolev embeddings, we find
\begin{equation}
\begin{aligned}
\label{T_2_Terms_New}
\mathrm{I}_2 =\int_{%
	\mathbb{R}^n} |v\tilde{w}| |\left( \tilde{v}+\tau \tilde{w}\right)|\dx 
\lesssim&\, \Vert v\Vert_{\dot{H}^{\frac{n-2}{2}}}\Vert \nabla \tilde{v}\Vert_{L^2}\Vert  \tilde{w}\Vert_{L^2}+\Vert v\Vert_{L^\infty}\Vert 
\tilde{w}\Vert_{L^{2}}^2.
\end{aligned}
\end{equation}  
Therefore, similarly to before we arrive at
\begin{equation}\label{T_2_Main_N_1_New}
\begin{aligned}
\int_0^t(1+\sigma)^{j-1/2}\mathrm{I}_2\, \textup{d}\sigma 
\lesssim&\,\begin{multlined}[t] \mathdutchcal{M}[v,w, \vecc{U}](t)\Upsilon^{(j,\kappa)}(t) \\ 
+\sup_{0\leq \sigma\leq t}\left( 1+\sigma \right) ^{%
	\frac{n}{2}}\left\Vert v(\sigma)\right\Vert _{L^{\infty }} \int_0^t (1+\sigma)^{j-1/2-n/2}   \int_{\mathbb{R}^{n}}|\tilde{w}(\sigma)|^{2}\dx \textup{d}\sigma
  \end{multlined}\\
\lesssim&\, (M_0[v](t)+
\mathdutchcal{M}[v,w, \vecc{U}](t))\Upsilon^{(j,\kappa)}(t).
\end{aligned}
\end{equation}
Since $s_0 \geq [n/2]+2$ by our assumption, we can rely on  Lemma \ref{M_estimate}, which yields $M_0[v](t)\lesssim \mathdutchcal{M}[v,w, \vecc{U}](t)$ and leads to
 \begin{equation}\label{T_2_Main_N_1_New}
 \begin{aligned}
 \int_0^t(1+\sigma)^{j-1/2}\mathrm{I}_2\, \textup{d}\sigma 
 \lesssim\,  \mathdutchcal{M}[v,w, \vecc{U}](t))\Upsilon^{(j,\kappa)}(t).
 \end{aligned}
 \end{equation}
\indent We continue with estimating the right-hand side terms in \eqref{I_1_Terms_New}. We have
\begin{equation}
\begin{aligned}
\label{T_1_Com_N_2_New}
\mathrm{I}_1 \lesssim&\,
\left\Vert \lbrack \nabla ^{\kappa},v]w\right\Vert _{L^{\frac{2n}{n+2}}}\left\Vert 
\tilde{v}+\tau \tilde{w} \right\Vert _{L^{\frac{2n}{n-2}}}
\lesssim\, \left\Vert \lbrack \nabla ^{\kappa},v]w\right\Vert _{L^{\frac{2n}{n+2}}}  \left\Vert \nabla\left(
\tilde{v}+\tau \tilde{w}\right) \right\Vert _{L^{2}}.
\end{aligned}
\end{equation}
The last term on the right can be further bounded by using the standard commutator estimate as follows:
\begin{equation}
\begin{aligned}
\label{Com_1_N_2_New}
\left\Vert \lbrack \nabla ^{\kappa},v]w\right\Vert _{L^{\frac{2n}{n+2}}}
=&\,  \Vert \nabla^\kappa(v w)-v \nabla^\kappa w\Vert_{L^{\frac{2n}{n+2}}} \\
\lesssim& \,\Vert w\Vert _{L^{n}}\Vert
\nabla ^{\kappa}v\Vert _{L^{2}}+\Vert \nabla
v\Vert _{L^{n}}\Vert \nabla ^{\kappa-1}w\Vert _{L^{2}}\\
\lesssim& \,  (\Vert
\nabla ^{\kappa}v\Vert _{L^{2}}+\Vert \nabla ^{\kappa-1}w\Vert _{L^{2}})(\Vert \nabla v\Vert_{\dot{H}^{\frac{n-2}{2}}} +\Vert w\Vert_{\dot{H}^{\frac{n-2}{2}}} );
\end{aligned}
\end{equation}
cf. Lemma~\ref{Guass_symbol_lemma}. Therefore, we obtain
\begin{equation}
\begin{aligned}
\label{T_1_Main_Estimate_N_2}
&\int_{0}^{t}(1+\sigma)^{j-1/2}\mathrm{I}_1\, \textup{d}\sigma\\
 \lesssim &\, \mathdutchcal{M}[v,w, \vecc{U}](t) \int_{0}^{t}(1+\sigma)^{j-1/2-n/4}(\sigma ) (\Vert \nabla ^{\kappa-1}w(\sigma )\Vert _{L^{2}}+\Vert
\nabla ^{\kappa}v(\sigma )\Vert _{L^{2}})\, \textup{d}\sigma\\
\lesssim &\,\mathdutchcal{M}[v,w, \vecc{U}](t) \Upsilon^{(j,\kappa)}(t).
\end{aligned}
\end{equation}
In the above estimate, we have used that, for $s_0\geq 1+ [n/2]$,
\begin{equation}
\begin{aligned}
 \Vert w(\sigma)\Vert_{\dot{H}^{\frac{n-2}{2}}}\lesssim \Vert w(\sigma)\Vert_{\dot{H}^{[\frac{n}{2}]}} \lesssim&\, \Vert w (\sigma)\Vert_{L^2}+\sum_{j=1}^{[\frac{n}{2}]} \Vert \nabla^j w\Vert_{L^2},\\
\lesssim&\, (1+\sigma)^{-\frac{n}{4}}\mathdutchcal{M}[v,w, \vecc{U}](t)
\end{aligned}
\end{equation}
and, similarly,
\begin{equation}
\Vert \nabla v(\sigma)\Vert_{\dot{H}^{\frac{n-2}{2}}}\lesssim  (1+\sigma)^{-\frac{n}{4}}\mathdutchcal{M}[v,w, \vecc{U}](t). 
\end{equation}
Finally, to estimate $\textup{I}_3$, we have as in \eqref{T_1_Com_N_2_New}, 
 \begin{equation}\label{T_3_N_2_5}
 \begin{aligned}
\mathrm{I}_3 \lesssim&\, \Vert \lbrack \nabla ^{\kappa},\nabla \psi]\nabla v\Vert _{L^{\frac{2n}{n+2}}}\left\Vert \left(
\tilde{v}+\tau \tilde{w}\right) \right\Vert _{L^{\frac{2n}{n-2}}}\\
\lesssim& \,\left\Vert \lbrack \nabla ^{\kappa},\nabla \psi]\nabla v\right\Vert _{L^{\frac{2n}{n+2}}} \left\Vert \nabla\left(
\tilde{v}+\tau \tilde{w}\right) \right\Vert _{L^{2}}.
\end{aligned}
\end{equation}
Applying again  the commutator estimate and the endpoint Sobolev embeddings  yields 
\begin{equation}
\begin{aligned}
\label{Estimate_N_2_T_3_New}
\left\Vert \lbrack \nabla ^{\kappa},\nabla \psi]\nabla v\right\Vert _{L^{\frac{2n}{n+2}}}\lesssim & \, (\Vert \nabla^2
\psi\Vert _{L^{n}}\Vert \nabla ^{\kappa}v\Vert _{L^{2}}+\Vert \nabla v\Vert _{L^{n}}\Vert
\nabla ^{\kappa+1}\psi\Vert _{L^{2}})\\
 \lesssim & \,  (\Vert \nabla^2
\psi\Vert _{\dot{H}^{[\frac{n}{2}]}}+\Vert \nabla v\Vert  _{\dot{H}^{[\frac{n}{2}]}})(\Vert \nabla ^{\kappa}v\Vert _{L^{2}}+\Vert
\nabla ^{\kappa+1}\psi\Vert _{L^{2}});
\end{aligned}
\end{equation}
cf.  Lemma~\ref{Guass_symbol_lemma}. At this point we can also employ the estimate
\begin{equation}
\begin{aligned}
\Vert \nabla^2 \psi\Vert _{\dot{H}^{[\frac{n}{2}]}}\lesssim &\, \| \nabla \vecc{U} \|_{\dot{H}^{[\frac{n}{2}]}} \lesssim \displaystyle \sum_{j=0}^{[\frac{n}{2}]+1} \|\nabla^j \vecc{U}\|_{L^2},
\end{aligned}
\end{equation}
from which it follows that
\begin{equation}\label{Ineq_Discussion}
\begin{aligned}
(1+\sigma)^{n/4+  1/2}
 \Vert \nabla^2 \psi\Vert _{\dot{H}^{[\frac{n}{2}]}}
\lesssim &\,\displaystyle \sum_{j=0}^{[\frac{n}{2}]+1}\sup_{0 \leq \sigma \leq t}(1+\sigma)^{n/4+ \frac{j}{2}} \| \nabla^{j} \vecc{U} \|_{L^2}\\
\lesssim&\, \mathdutchcal{M}[v,w, \vecc{U}](t)
\end{aligned}
\end{equation}
for every $\sigma \in [0, t]$ Then by using $ (1+t)^{n/4+1/2} \leq \displaystyle \sup_{0\leq \sigma\leq t}(1+\sigma)^{n/4+1/2}$,
 we obtain
\begin{equation}
\begin{aligned}
 \Vert \nabla^2 \psi\Vert _{\dot{H}^{[\frac{n}{2}]}}\lesssim &\, (1+t)^{-n/4-1/2} \mathdutchcal{M}[v,w, \vecc{U}](t),
\end{aligned}   
\end{equation}
provided that $s_0 \geq [n/2]+1$ in the definition of $\mathdutchcal{M}[v,w, \vecc{U}](t)$. Similarly, we have    
\begin{equation}
\Vert \nabla v\Vert  _{\dot{H}^{[\frac{n}{2}]}}\lesssim (1+t)^{-n/4-1/2}\mathdutchcal{M}[v,w, \vecc{U}](t),
\end{equation} 
The derived estimates further yield
 \begin{equation}   
 \begin{aligned}
 \label{T_3_Estimate_N_2}
&\int_0^t (1+\sigma)^{j-1/2}\mathrm{I}_3 \, \textup{d}\sigma \\
\lesssim &\,\mathdutchcal{M}[v,w, \vecc{U}](t)\int_0^t (1+\sigma)^{j-1-n/4} (\Vert \nabla ^{\kappa}v\Vert _{L^{2}}+\Vert
\nabla ^{\kappa+1}\psi\Vert _{L^{2}})\left\Vert \nabla^{\kappa+1}\left(
v+\tau w\right) \right\Vert _{L^{2}}\\
\lesssim&\,
\mathdutchcal{M}[v,w, \vecc{U}](t)\int_0^t (1+\sigma)^{j-1-n/4} (\Vert \nabla ^{\kappa}v\Vert _{L^{2}}^2+\Vert
\nabla ^{\kappa+1}\psi\Vert _{L^{2}}^2+\left\Vert \nabla^{\kappa+1}\left(
v+\tau w\right) \right\Vert _{L^{2}}^2). 
\end{aligned}
\end{equation}
We proceed to bound the three terms on the right-hand side above. Since $n \geq 3$, we have
\begin{equation}
\begin{aligned}
\int_0^t (1+\sigma)^{j-1-n/4} \Vert \nabla ^{\kappa}v\Vert _{L^{2}}^2 \, \textup{d}\sigma\lesssim&\, \int_0^t (1+\sigma)^{j-3/2}\Vert \nabla ^{\kappa}v\Vert _{L^{2}}^2 \, \textup{d}\sigma\\
\lesssim&\, \int_0^t (1+  \sigma)^{j-3/2}\mathbf{E}^{(\kappa)}(\sigma)\, \textup{d}\sigma
\lesssim\,\Upsilon^{(j,\kappa)}(t).
\end{aligned}
\end{equation}
The second term containing $\Vert
\nabla ^{\kappa+1}\psi\Vert _{L^{2}}^2$ can be estimated similarly. For the third term, we have 
\begin{equation}
\begin{aligned}
\int_0^t (1+\sigma)^{j-1/2-n/4-1/2} \left\Vert \nabla^{\kappa+1}\left(
	v+\tau w\right) \right\Vert _{L^{2}}^2\, \textup{d}\sigma
\lesssim&\, \int_0^t (1+  \sigma)^{j-1/2}\mathbf{D}^{(\kappa)}(\sigma)\, \textup{d}\sigma\\
\lesssim&\,\Upsilon^{(j,\kappa)}(t),
\end{aligned}
\end{equation}
which proves 
\[
\int_0^t (1+\sigma)^{j-1/2}I_3 \, \textup{d}\sigma \lesssim \,\mathdutchcal{M}[v,w, \vecc{U}](t) \Upsilon^{(j,\kappa)}(t).
\]
Finally, our estimates combined yield
\begin{equation}
\begin{aligned}  \label{R_1_Estimate_k_1}
\int_0^t (1+\sigma)^{j-1/2}| R_\kappa^{(1)}\nabla^\kappa(v + \tau w)(\sigma)|\, \textup{d}\sigma
\lesssim \,
\mathdutchcal{M}[v,w, \vecc{U}](t)\Upsilon^{(j,\kappa)}(t).
\end{aligned}
\end{equation}
The estimate of $\int_0^t (1+\sigma)^{j-1/2}| R_\kappa^{(1)}\nabla^\kappa(\tau w)(\sigma)|\, \textup{d}\sigma$ can be derived analogously. We omit the details here.
\end{proof}
\newtheorem*{prop2}{Step II}
 \begin{prop2}\label{lemma_k_R_2}
Under the assumptions of Theorem~\ref{Thm_R_k}, it holds
 \begin{equation}\label{R_2_kappa_New}
 \begin{aligned}
&\int_0^t (1+\sigma)^{j-1/2}\left(|R_\kappa^{(2)}(\nabla^\kappa(v + \tau w)(\sigma)|+|R_\kappa^{(2)}(\tau \nabla^\kappa v)(\sigma)|\right)\, \textup{d}\sigma\\  
\lesssim& \,\mathdutchcal{M}[v,w, \vecc{U}](t) \Upsilon^{(j,\kappa)}(t) . 
\end{aligned}
\end{equation}
Furthermore,
\begin{equation}  \label{R_2_kappa_New_}
\begin{aligned}
\int_0^t (1+\sigma)^{j-1/2}|R_{\kappa+1}^{(2)}(\nabla^{\kappa+1}(v + \tau w)(\sigma)|\, \textup{d}\sigma
\lesssim\,\mathdutchcal{M}[v,w, \vecc{U}](t) \Upsilon^{(j,\kappa)}(t).   
\end{aligned}
\end{equation}
 \end{prop2}   
 \begin{proof}
We present first the proof of estimate \eqref{R_2_kappa_New}. We have by the Cauchy--Schwarz inequality 
 \begin{equation}  \label{R_2_R_4}
 \begin{aligned}
&|R_\kappa^{(2)}(\nabla^\kappa(v+\tau w))|(\sigma)+|R_\kappa^{(2)}(\tau \nabla^\kappa v)| \\
\lesssim&\, \Vert \nabla F^{(\kappa)}\Vert _{L^{2}}(%
\Vert\nabla\left( \tilde{v}+\tau \tilde{w}\right)\Vert_{L^2}+\Vert\nabla \tilde{v}\Vert_{L^2}).
\end{aligned}
\end{equation}
Keeping in mind that
\begin{equation}
F^{(\kappa)}(\psi,v,\nabla \psi,\nabla v)=2k[\nabla^\kappa,v]w+2k v\tilde{w}+2\kappa[\nabla^\kappa,\nabla \psi] \cdot \nabla v+2k\nabla \psi \cdot \nabla\tilde{w},
\end{equation}
 applying the commutator estimate \eqref{First_inequaliy_Guass} yields
\begin{equation}  \label{R_1_k_Nabla} 
\begin{aligned}
\Vert \nabla F^{(\kappa)}\Vert _{L^{2}} \lesssim&\,\begin{multlined}[t]\Vert w\Vert
_{L^{\infty }}\Vert \nabla ^{\kappa+1}v\Vert _{L^{2}}+\Vert v\Vert _{L^{\infty
}}\Vert \nabla ^{\kappa+1}w\Vert _{L^{2}}\\
+ \Vert \nabla \psi\Vert _{L^{\infty }}\Vert \nabla ^{\kappa+2}v\Vert
_{L^{2}}+\Vert \nabla v\Vert _{L^{\infty }}\Vert \nabla ^{\kappa+2}\psi\Vert
_{L^{2}}.\end{multlined}
\end{aligned}
\end{equation}%
The last two terms can be further estimated as follows:
\begin{equation}
\begin{aligned}
&\Vert \nabla \psi\Vert _{L^{\infty }}\Vert \nabla ^{\kappa+2}v\Vert _{L^{2}}+\Vert
\nabla v\Vert _{L^{\infty }}\Vert \nabla ^{\kappa+2}\psi\Vert _{L^{2}} \\
\lesssim&\, \Vert \nabla \psi \Vert _{L^{\infty }}\Vert \Delta \nabla
^{\kappa}v\Vert _{L^{2}}+\Vert \nabla v\Vert _{L^{\infty }}\Vert \Delta \nabla
^{\kappa}\psi\Vert _{L^{2}} \\
\lesssim&\, \Vert \nabla \psi\Vert _{L^{\infty }}\Vert \Delta \nabla
^{\kappa}v\Vert _{L^{2}}+\Vert \nabla v\Vert _{L^{\infty }}\left( \Vert \Delta
\nabla ^{\kappa}\left( \psi+\tau v\right) \Vert _{L^{2}}+\Vert \Delta \nabla
^{\kappa}v\Vert _{L^{2}}\right) .
\end{aligned}
\end{equation}%
By inserting the above estimates into \eqref{R_1_k_Nabla} and using
\begin{equation}  
\begin{aligned}
\Vert w\Vert_{L^\infty}\lesssim\, \Vert v\Vert_{L^\infty}+\Vert v+\tau w\Vert_{L^\infty}
\lesssim\, (M_0[v](t)+M_0[\vecc{U}](t))(1+t)^{-n/2},
\end{aligned}  
\end{equation}
as well as 
\begin{equation}
\begin{aligned}
\Vert \nabla \psi\Vert _{L^{\infty }}+\Vert \nabla v\Vert _{L^{\infty }}\lesssim M_0[\vecc{U}](t) (1+t)^{-n/2}, 
\end{aligned}
\end{equation}
 we infer \begin{equation}
 \begin{aligned}
   \label{Nabla_R_1_k}
\Vert \nabla F^{(\kappa)}\Vert _{L^{2}}\lesssim &\, \begin{multlined}[t] (M_0[\vecc{U}](t)+M_0[v](t))(1+t)^{-n/2}\\
\times ( \Vert \nabla
\tilde{v}\Vert _{L^{2}}+\Vert \nabla \tilde{w}\Vert _{L^{2}}+\Vert \Delta \tilde{v}\Vert
_{L^{2}}+\Vert \Delta ( \tilde{\psi}+\tau \tilde{v}) \Vert _{L^{2}}). \end{multlined}
\end{aligned}
\end{equation}
Noting that $s_0 \geq [n/2]+2$, then by virtue of the above estimates and Lemma \ref{M_estimate},  estimate \eqref{R_2_kappa_New} also holds. We can proceed similarly to obtain 
\begin{equation}
\begin{aligned}
|R_{\kappa+1}^{(2)}(\nabla^{\kappa+1}(v + \tau w)(\sigma)|\lesssim\Vert \nabla F^{(\kappa+1)}\Vert _{L^{2}} \Vert\nabla^{\kappa+1}\left( v+\tau w\right)\Vert_{L^2}. 
\end{aligned}
\end{equation}
Analogously to \eqref{Nabla_R_1_k}, we have 
\begin{equation}
\begin{aligned}
   \label{Nabla_R_1_k+1}
\Vert \nabla F^{(\kappa+1)}\Vert _{L^{2}}\lesssim &\,  (M_0[\vecc{U}](t)+M_0[v](t))(1+t)^{-n/2}\\
&\times ( \Vert \nabla^{\kappa+2}
v\Vert _{L^{2}}+\Vert \nabla^{\kappa+2} w\Vert _{L^{2}}+\Vert \Delta \nabla^{\kappa+1}v\Vert
_{L^{2}}+\Vert \Delta \nabla^{\kappa+1}( \psi+\tau v) \Vert _{L^{2}}).
\end{aligned}
\end{equation}
Therefore, using the above estimate and keeping in mind how $\Upsilon^{j,\kappa}$ is defined, estimate \eqref{R_2_kappa_New_} follows; cf. \eqref{D_k_dissipation}.    
\end{proof}

\newtheorem*{prop3}{Step III}
  \begin{prop3}\label{Lemma_R_2_psi}
Under the assumptions of Theorem~\ref{Thm_R_k}, it holds
  \begin{equation}\label{Equation_Main_R_k}
\begin{aligned}
&\int_0^t (1+\sigma)^{j-1/2} |R_\kappa^{(2)}(\nabla^{\kappa}(\psi+\tau v)(\sigma))|\, \textup{d}\sigma\\
\lesssim&\, \mathdutchcal{M}[v,w, \vecc{U}](t) \left(\Upsilon^{(j,\kappa)}(t)+\int_0^t (1+\sigma)^{j-3/2}\mathbf{D}^{\kappa-1}(\sigma)\textup{d}\sigma\right).
\end{aligned}  
\end{equation}
  \end{prop3}
  \begin{proof}
We observe that
\begin{equation}
\begin{aligned}
|R_\kappa^{(2)}(\nabla^{\kappa}(\psi+\tau v))|= \, |(\nabla F^{(\kappa)}, \nabla (\nabla^{\kappa}(\psi+\tau v)))_{L^2}|
=&\, |(F^{(\kappa)}, \Delta \nabla^{\kappa}(\psi+\tau v))_{L^2}| \\
\leq&\, \Vert F^{(\kappa)}\Vert _{L^{2}} \Vert \Delta   ( \tilde{\psi}+\tau
\tilde{v}) \Vert_{L^2}. 
\end{aligned}
\end{equation}
Keeping in mind how $F^{(\kappa)}$ is defined in \eqref{F_k_Form}, we can estimate its $L^2$ norm as follows: 
\begin{equation} \label{R_1_k_0_1}
\begin{aligned}
\Vert F^{(\kappa)}\Vert _{L^{2}} \lesssim&\,  \begin{multlined}[t]\Vert \lbrack \nabla
^{k},v]w\Vert _{L^{2}}+\Vert v\Vert _{L^{\infty }}\Vert \tilde{w}\Vert
_{L^{2}} \\
+ \Vert \lbrack \nabla ^{\kappa},\nabla \psi]\nabla v\Vert _{L^{2}}+\Vert
\nabla \psi\Vert _{L^{\infty }}\Vert \nabla \tilde{v}\Vert _{L^{2}}. \end{multlined}
\end{aligned}
\end{equation}
By applying the commutator estimate, we deduce
\begin{equation}\label{w_k-1_Estimate}
\Vert \lbrack \nabla ^{\kappa},v]w\Vert _{L^{2}}\lesssim \Vert%
\nabla v\Vert _{L^{\infty }}\Vert \nabla ^{\kappa-1}w\Vert _{L^{2}}+\Vert w\Vert
_{L^{\infty }}\Vert \nabla ^{\kappa}v\Vert _{L^{2}};  
\end{equation}
cf. Lemma~\ref{Guass_symbol_lemma}. This further yields 
 \begin{equation}\label{First_Estimate_Com}
\begin{aligned}
 &\Vert\nabla v\Vert _{L^{\infty }}\Vert \nabla ^{\kappa-1}w\Vert _{L^{2}}\Vert \Delta   ( \tilde{\psi}+\tau
\tilde{v}) \Vert_{L^2}\\
\lesssim &\, M_0[\vecc{U}](t)(1+t)^{-n/2}
  \left(\Vert \nabla ^{\kappa-1}w\Vert _{L^{2}}^2+\Vert \Delta   \nabla^\kappa( \psi+\tau  
v) \Vert_{L^2}^2\right). 
\end{aligned}  
\end{equation}     
For the first term on the right, since $n\geq 3$, we have 
\begin{equation}
\begin{aligned}
&\int_0^t (1+\sigma)^{j-1/2}M_0[\vecc{U}](t)(1+t)^{-n/2}
  \Vert \nabla ^{\kappa-1}w\Vert _{L^{2}}^2\\
\lesssim &\, M_0[\vecc{U}](t)\int_0^t (1+\sigma)^{j-1-1/2}\Vert \nabla ^{\kappa-1}w\Vert _{L^{2}}^2\textup{d}\sigma\\
\lesssim &\, M_0[\vecc{U}](t)\int_0^t (1+\sigma)^{j-1-1/2}\mathbf{D}^{\kappa-1}(\sigma)\textup{d}\sigma. 
\end{aligned}
\end{equation}
Similarly, it holds
\begin{equation}
\begin{aligned}
\Vert \lbrack \nabla ^{\kappa},\nabla \psi]\nabla v\Vert _{L^{2}} \lesssim&\, (\Vert
\nabla v\Vert _{L^{\infty }}\Vert \nabla ^{\kappa+1}\psi\Vert _{L^{2}}+%
\Vert \nabla^2 \psi\Vert _{L^{\infty }}\Vert \nabla ^{\kappa}v\Vert
_{L^{2}})  \\
\lesssim& \, (\Vert \nabla v\Vert _{L^{\infty }}\left( \Vert \nabla ^{\kappa+1}\left(
\psi+\tau v\right) \Vert _{L^{2}}+\Vert \nabla ^{\kappa+1}v\Vert _{L^{2}}\right) +%
\Vert \nabla^2 \psi\Vert _{L^{\infty }}\Vert \nabla ^{k}v\Vert
_{L^{2}}).  
\end{aligned}
\end{equation}
By using the bounds
\[\Vert \nabla^2 \psi\Vert_{L^{\infty }} \lesssim M_1[\vecc{U}](t) (1+t)^{-n/2-1/2} \quad \text{and} \quad \Vert \nabla v\Vert _{L^{\infty }} \lesssim M_0[\vecc{U}](t) (1+t)^{-n/2}, \]
we then have
\begin{equation}\label{Second_Estimate_Com}
\begin{aligned}
&\Vert \lbrack \nabla ^{\kappa},\nabla \psi]\nabla v\Vert _{L^{2}}\Vert \Delta   ( \tilde{\psi}+\tau
\tilde{v}) \Vert_{L^2}\\
\lesssim &\,\begin{multlined}[t](M_0[\vecc{U}](t)+M_1[\vecc{U}](t)) (1+t)^{-n/2}\\
\times \left(\Vert \Delta   ( \nabla^\kappa(\psi+\tau
v) \Vert_{L^2}^2+\Vert \nabla ^{\kappa+1}\left(
\psi+\tau v\right) \Vert _{L^{2}}^2+\Vert \nabla ^{\kappa+1}v\Vert _{L^{2}}^2+\Vert \nabla ^{k}v\Vert
_{L^{2}}^2\right). \end{multlined}
\end{aligned}
\end{equation}
With the same approach, we infer 
\begin{equation}\label{Third_Estimate_Com}
\begin{aligned}
&\left(\Vert v\Vert _{L^{\infty }}\Vert \tilde{w}\Vert
_{L^{2}}+\Vert
\nabla \psi\Vert _{L^{\infty }}\Vert \nabla \tilde{v}\Vert _{L^{2}}\right)\Vert \Delta   ( \tilde{\psi}+\tau
\tilde{v}) \Vert_{L^2}\\
\lesssim&\,\begin{multlined}[t]\left(M_0[\vecc{U}](t)+M_0[v](t)\right)(1+t)^{-n/2}
 \left(\Vert \nabla^\kappa w\Vert
_{L^{2}}^2+\Vert \Delta   ( \nabla^\kappa(\psi+\tau
v) \Vert_{L^2}^2+\Vert \nabla ^{\kappa+1}v\Vert _{L^{2}}^2\right). \end{multlined}  
\end{aligned}
\end{equation}
By plugging the derived estimates into \eqref{R_1_k_0_1} and applying Lemma \ref{M_estimate},   
we obtain \eqref{Equation_Main_R_k}.   \\
\indent Steps I--III taken together complete the proof of Theorem~\ref{Thm_R_k}.
\end{proof}
\subsection{Proof of estimate \eqref{Main_Estimate_E_D_Main}}\label{Proof_Main_Estimate_1}
Our next aim is to prove the following estimate:
\begin{equation} 
\begin{aligned}
\|\vecc{\Psi}\|_{\mathbbm{E},t}^2+\|\vecc{\Psi}\|_{\mathbbm{D},t}^2 \lesssim& \,\begin{multlined}[t] ||| \vecc{\Psi}(0)|||_{H^s}^2+\Vert \vecc{U}_0\Vert^2_{L^1}+\mathdutchcal{M}[v, w, \vecc{U}](t)\|\vecc{\Psi}\|_{\mathbbm{D},t}^2,\end{multlined}
\end{aligned}
\end{equation}	
where the energy norms are defined in \eqref{Energy_Main} and \eqref{Energy_Main_Dissipative}. We begin by proving a helpful auxiliary inequality.
\begin{lemma}   \label{Lemma:Upsilon}
For an integer $s \geq 1$, it holds
\begin{equation}\label{Right_2}
	\sum_{\kappa=0}^{s-1}\Upsilon^{(0,\kappa)}(t)\lesssim \|\vecc{\Psi}\|_{\mathbbm{D},t}^2,
\end{equation}
 where 
\begin{equation} \tag{\ref{D_k_dissipation}}
\Upsilon^{(0,\kappa)}(t)= \int_0^t (1+  \sigma)^{-3/2}\mathbf{E}^{(\kappa)}(\sigma) \, \textup{d}\sigma+\int_0^t (1+\sigma)^{-1/2}\mathbf{D}^{(\kappa)}(\sigma)\, \textup{d}\sigma. 
\end{equation}	
\end{lemma}
\begin{proof}
The statement follows by noting that
	\begin{equation}
	\begin{aligned}
	\int_0^t (1+\sigma)^{-1/2}\sum_{\kappa=0}^{s-1}\mathbf{D}^{(\kappa)}(\sigma)\, \textup{d}\sigma =&\, \int_0^t (1+\sigma)^{-1/2}|\vecc{\Psi}(\sigma)|_{\mathbf{H}^{s}}^2.
	\end{aligned}
	\end{equation}
and that
	\begin{equation}
	\begin{aligned}
	\int_0^t (1+  \sigma)^{-3/2}\sum_{\kappa=0}^{s-1}\mathbf{E}^{(\kappa)}(\sigma) \, \textup{d}\sigma = &\,\int_0^t (1+  \sigma)^{-3/2}||| \vecc{\Psi}(\sigma)|||_{H^{s}}^2 \, \textup{d}\sigma;
	\end{aligned}
	\end{equation}
cf. \eqref{1}, \eqref{2}, and \eqref{identity_EHs}.	
\end{proof}
\noindent We are now ready to prove the second crucial bound for the weighted energy of the solution.
  \begin{theorem}\label{Thm_Main_1_repeat}  
  	Let $s_0 \geq [n/2]+2$. Then the solution $\vecc{\Psi}=(\psi, v, w, \eta)^T$ of \eqref{Main_System}, \eqref{Main_System_IC} with $b= \tau c^2$ satisfies
  	\begin{equation} 
\|\vecc{\Psi}\|_{\mathbbm{E},t}^2+\|\vecc{\Psi}\|_{\mathbbm{D},t}^2 \lesssim \, ||| \vecc{\Psi}(0)|||_{H^s}^2+\mathdutchcal{M}[v, w, \vecc{U}](t)\|\vecc{\Psi}\|_{\mathbbm{D},t}^2,
  	\end{equation}
  where \begin{equation} 
  \vecc{U} =(v+\tau w,\nabla(\psi+\tau v),\nabla v)^T.  
  \end{equation}	
  \end{theorem}
\begin{proof}
By writing out the energy terms, we see that our claim can be restated as
  	\begin{equation} \label{Main_est_a}
  	\begin{aligned}
&\begin{multlined}[t] \sum_{j=0}^{[\frac{s-1}{2}]}\sup_{0\leq \sigma\leq t} (1+\sigma)^{j-1/2}||| \nabla^i\vecc{\Psi}(\sigma)|||_{H^{s-2j}}^2\\+\sum_{j=0}^{[\frac{s-1}{2}]}\int_0^t (1+\sigma)^{j-1/2}| \nabla^i\vecc{\Psi}(\sigma)|_{\mathbf{H}^{s-2j}}^2\, \textup{d}\sigma\\
+\sum_{j=0}^{[\frac{s-1}{2}]}     \int_0^t (1+\sigma)^{j-3/2}|||\nabla^i \vecc{\Psi}(t)|||_{H^{s-2j}}^2(\sigma) \, \textup{d}\sigma \end{multlined}\\
 \lesssim &\, |||\vecc{\Psi}(0)|||_{H^s}^2+\mathdutchcal{M}[v, w, \vecc{U}](t)\|\vecc{\Psi}\|_{\mathbbm{D},t}^2.
\end{aligned}    
\end{equation}
It is, therefore, enough to prove the following three estimates: 
 \begin{equation}
\begin{aligned}  
\label{Main_Estimate_s_0}
&\begin{multlined}[t](1+t)^{-1/2}||| \vecc{\Psi}(t)|||_{H^s}^2+\int_0^t (1+\sigma)^{-1/2}| \vecc{\Psi}(\sigma)|_{\mathbf{H}^{s}}^2 \, \textup{d}\sigma\\
+\int_0^t (1+\sigma)^{-3/2}|||\vecc{\Psi}(t)|||_{H^s}^2(\sigma) d\sigma
 \end{multlined}\\
\lesssim &||| \vecc{\Psi}(0)|||_{H^s}^2
+\mathdutchcal{M}[v, w, \vecc{U}](t) \|\vecc{\Psi}\|_{\mathbbm{D},t}^2, \end{aligned} 
\end{equation}
as well as
\begin{equation}\label{First_Main_Ind}
\begin{aligned}
&\begin{multlined}[t] (1+t)^{j-1/2}||| \nabla^j\vecc{\Psi}(\sigma)|||_{H^{s-2j}}^2 
+ \int_0^t (1+\sigma)^{j-1/2}| \nabla^j\vecc{\Psi}(\sigma)|_{\mathbf{H}^{s-2j}}^2\, \textup{d}\sigma \end{multlined}\\   
\lesssim &\,\, ||| \vecc{\Psi}(0)|||_{H^s}^2     
 +\mathdutchcal{M}[v, w, \vecc{U}](t)\|\vecc{\Psi}\|_{\mathbbm{D},t}^2,  \qquad j \in \left\{1, \ldots, \left[\tfrac{s-1}{2}\right]\right\},
 \end{aligned}
\end{equation}  
and
\begin{equation}\label{Induction_Key-Estimate}
\begin{aligned}  
 &\int_0^t (1+\sigma)^{j-3/2}|||\nabla^j \vecc{\Psi}(t)|||_{H^{s-2j}}^2(\sigma) \, \textup{d}\sigma\\
 \lesssim &\,\, ||| \vecc{\Psi}(0)|||_{H^s}^2
 +\mathdutchcal{M}[v, w, \vecc{U}](t) \|\vecc{\Psi}\|_{\mathbbm{D},t}^2, \qquad j \in \left\{1, \ldots, \left[\tfrac{s-1}{2}\right]\right\}. 
\end{aligned} 
\end{equation}
We can rewrite \eqref{Induction_Key-Estimate} by shifting the index as
\begin{equation}\label{Induction_Key-Estimate_2}
\begin{aligned}  
&\int_0^t (1+\sigma)^{j-1/2}|||\nabla^{j+1} \vecc{\Psi}(t)|||_{H^{s-2j-2}}^2(\sigma) \, \textup{d}\sigma\\
\lesssim&\, ||| \vecc{\Psi}(0)|||_{H^s}^2  
+\mathdutchcal{M}[v, w, \vecc{U}](t) \|\vecc{\Psi}\|_{\mathbbm{D},t}^2, \qquad j \in \left\{0, \ldots, \left[\tfrac{s-1}{2}\right]-1\right\}. 
\end{aligned}
\end{equation}
We begin by proving estimate \eqref{Main_Estimate_s_0}.\\
\paragraph {\bf Proof of estimate \eqref{Main_Estimate_s_0}: }      
By Theorem~\ref{Thm_R_k}, we have the bound
\begin{equation} \label{R_kappa_D}
\begin{aligned}
\int_0^t (1+\sigma)^{-1/2}\mathbf{R}_\kappa(\sigma)\, \textup{d}\sigma 
\lesssim\, \mathdutchcal{M}[v, w, \vecc{U}](t)\left(\int_0^t (1+\sigma)^{-3/2}\mathbf{D}^{\kappa-1}(\sigma)\textup{d}\sigma+\Upsilon^{(0,\kappa)}(t)\right),
\end{aligned} 
\end{equation}
provided that $s_0\geq [n/2]+2$. Taking $\gamma=-1/2$ in estimate \eqref{Main_Estimate_k} yields
	\begin{equation}
	\begin{aligned}
	&(1+t)^{-1/2}\mathbf{E}^{(\kappa)}(t)+\frac12 \int_0^t (1+\sigma)^{-3/2}\mathbf{E}^{(\kappa)}(\sigma) \, \textup{d}\sigma+\int_0^t (1+\sigma)^{-1/2}\mathbf{D}^{(\kappa)}(\sigma)\, \textup{d}\sigma\\
	\lesssim& \, \begin{multlined}[t]\mathbf{E}^{(\kappa)}(0)	+\int_0^t (1+\sigma)^{-1/2}\mathbf{R}_\kappa(\sigma)\, \textup{d}\sigma.\end{multlined}
	\end{aligned}    
	\end{equation}  
Using the above bound on the $\mathbf{R}_\kappa$ term results in
	\begin{equation}\label{Order_zero_Estimate}
\begin{aligned}
&(1+t)^{-1/2}\mathbf{E}^{(\kappa)}(t)+\frac12 \int_0^t (1+\sigma)^{-3/2}\mathbf{E}^{(\kappa)}(\sigma) \, \textup{d}\sigma+\int_0^t (1+\sigma)^{-1/2}\mathbf{D}^{(\kappa)}(\sigma)\, \textup{d}\sigma\\
\lesssim& \, \begin{multlined}[t]\mathbf{E}^{(\kappa)}(0)+\mathdutchcal{M}[v, w, \vecc{U}](t)\Big(\int_0^t (1+\sigma)^{-3/2}\mathbf{D}^{\kappa-1}(\sigma)\textup{d}\sigma+\Upsilon^{(0,\kappa)}(t)\Big).\end{multlined}
\end{aligned} 
\end{equation}
We next wish to sum the above inequalities over  $1\leq \kappa<s-1$ and add the resulting bound to \eqref{Main_Estimate_0_1}, which we restate here for convenience
\begin{equation}   
(1+t)^{-1/2}\mathbf{E}^{(0)}(t)+\Upsilon^{(0,0)}(t)
\lesssim \mathbf{E}^{(0)}(0)
+\int_0^t (1+\sigma)^{-1/2}\mathbf{R}  _0(\sigma)\, \textup{d}\sigma.
\end{equation}
To this end, we first recall that
\begin{equation} 
\displaystyle \sum_{\kappa=0}^{s-1} \mathbf{E}^{(\kappa)}(t) = ||| \vecc{\Psi}(t)|||_{H^s}^2, \qquad \sum_{\kappa=0}^{s-1} \mathbf{D}^{(\kappa)}(t)= |\vecc{\Psi}(t)|_{\mathbf{H}^{s}}^2.
\end{equation}
Additionally, by virtue of Lemma~\ref{Lemma:Upsilon}, the following estimate holds:
\begin{equation}\label{Key_estimate}
\sum_{\kappa=1}^{s-1}\Big(\int_0^t (1+\sigma)^{-3/2}\mathbf{D}^{\kappa-1}(\sigma)\textup{d}\sigma+\Upsilon^{(0,\kappa)}(t)\Big)\lesssim \sum_{\kappa=0}^{s-1}\Upsilon^{(0,\kappa)}(t)\lesssim  \|\vecc{\Psi}\|_{\mathbbm{D},t}^2, 
\end{equation} 
Hence, by summing  inequalities \eqref{Order_zero_Estimate} over $1\leq \kappa<s-1$ and
adding the resulting estimate to \eqref{Main_Estimate_0_1},  we arrive at
  \begin{equation}  \label{Main_Estimate_s}
 \begin{aligned}
&(1+t)^{-1/2}||| \vecc{\Psi}(t)|||_{H^s}^2+\int_0^t (1+\sigma)^{-3/2}||| \vecc{\Psi}(t)|||_{H^s}^2(\sigma) \, \textup{d}\sigma
+\int_0^t (1+\sigma)^{-1/2}| \vecc{\Psi}(\sigma)|_{\mathbf{H}^{s}}^2\, \textup{d}\sigma \\
\lesssim&\, ||| \vecc{\Psi}(0)|||_{H^s}^2
 +\mathdutchcal{M}[v, w, \vecc{U}](t)\|\vecc{\Psi}\|_{\mathbbm{D},t}^2, \end{aligned} 
\end{equation}
 where we have additionally used the estimate of the $\mathbf{R}_0$ term given in \eqref{R_0_Main_Estimate}. This proves \eqref{Main_Estimate_s_0}. \\
 \indent We next prove estimates \eqref{First_Main_Ind} and \eqref{Induction_Key-Estimate_2} by induction on $j$.  \\[2mm]
\paragraph {\bf Basis step for  \eqref{First_Main_Ind}  and \eqref{Induction_Key-Estimate_2}:}
By Lemma~\ref{Norm_Estimate}, we know that
\begin{equation}
|||\nabla  \vecc{\Psi}(t)|||_{H^{s-2}}\lesssim | \vecc{\Psi}(t)|_{\mathbf{H}^{s}}.
\end{equation}
Thus, it is clear from  the proven bound \eqref{Main_Estimate_s_0}  that estimate \eqref{Induction_Key-Estimate_2} holds for $j=0$ and \eqref{First_Main_Ind} holds for $j=1$.  \\

\paragraph {\bf Inductive step for estimates \eqref{First_Main_Ind}  and \eqref{Induction_Key-Estimate_2}:} 
Let $1\leq j\leq [\frac{s-1}{2}]$. We assume that \eqref{First_Main_Ind} holds with $j-1$ in place of $j$:
\begin{equation}
\begin{aligned}
&\begin{multlined}[t] (1+t)^{j-3/2}||| \nabla^{j-1}\vecc{\Psi}(\sigma)|||_{H^{s-2(j-1)}}^2   
+ \int_0^t (1+\sigma)^{j-3/2}| \nabla^{j-1}\vecc{\Psi}(\sigma)|_{\mathbf{H}^{s-2(j-1)}}^2\, \textup{d}\sigma \end{multlined}\\   
\lesssim &\,\, ||| \vecc{\Psi}(0)|||_{H^s}^2     
+\mathdutchcal{M}[v, w, \vecc{U}](t)\|\vecc{\Psi}\|_{\mathbbm{D},t}^2,
\end{aligned}
\end{equation}
and that \eqref{Induction_Key-Estimate_2} holds with $j-1$ in place of $j$: 


\begin{equation}\label{Ind_Hyp}  
\begin{aligned}
\int_0^t (1+\sigma)^{j-3/2}||| \nabla^j\vecc{\Psi}(\sigma)|||_{H^{s-2j}}^2\, \textup{d}\sigma
\lesssim\,|||\vecc{\Psi}(0)|||_{H^s}^2  +\mathdutchcal{M}[v, w, \vecc{U}](t)\|\vecc{\Psi}\|_{\mathbbm{D},t}^2.
\end{aligned}
\end{equation}
Under this induction hypothesis we first prove \eqref{First_Main_Ind}. Taking $\gamma=j-1/2$ in estimate \eqref{Main_Estimate_k} leads to
	\begin{equation}
\begin{aligned}
&(1+t)^{j-1/2}\mathbf{E}^{(\kappa)}(t)+\int_0^t (1+\sigma)^{j-1/2}\mathbf{D}^{(\kappa)}(\sigma)\, \textup{d}\sigma\\
\lesssim& \, \begin{multlined}[t]\mathbf{E}^{(\kappa)}(0)+ (j-1/2) \int_0^t (1+\sigma)^{j-3/2}\mathbf{E}^{(\kappa)}(\sigma) \, \textup{d}\sigma
+\int_0^t (1+\sigma)^{j-1/2}\mathbf{R}_\kappa(\sigma)\, \textup{d}\sigma.\end{multlined}
\end{aligned} 
\end{equation}
Using Theorem~\ref{Thm_R_k} allows us to bound the last term on the right above, thus obtaining
	\begin{equation}\label{Induction_2}
\begin{aligned}
&(1+t)^{j-1/2}\mathbf{E}^{(\kappa)}(t)+\int_0^t (1+\sigma)^{j-1/2}\mathbf{D}^{(\kappa)}(\sigma)\, \textup{d}\sigma\\
\lesssim& \, \begin{multlined}[t]\mathbf{E}^{(\kappa)}(0)+ (j-1/2) \int_0^t (1+\sigma)^{j-3/2}\mathbf{E}^{(\kappa)}(\sigma) \, \textup{d}\sigma
\\+\mathdutchcal{M}[v, w, \vecc{U}](t)\left(\Upsilon^{(j,\kappa)}(t)+\int_0^t (1+\sigma)^{(j-1)-1/2}\mathbf{D}^{\kappa-1}(\sigma)\textup{d}\sigma\right).\end{multlined}  
\end{aligned} 
\end{equation}
We next wish to sum up the above inequalities over $\kappa$, where $j \leq \kappa \leq s-j-1$. We note that
\begin{equation} \label{identities}
\begin{aligned}
\sum_{\kappa=j}^{s-j-1} \mathbf{E}^{(\kappa)}(t)=&\, |||\nabla^j \vecc{\Psi}(t)|||_{H^{s-2j}}^2, \quad \sum_{\kappa=j}^{s-j-1} \mathbf{D}^{(\kappa)}(t)= |\vecc{\Psi}(t)|_{\mathbf{H}^{s}}^2;
\end{aligned}
\end{equation}
cf. \eqref{energy_mathbfE}, \eqref{energy_mathbfD}. After summation, we thus have
\begin{equation}
\begin{aligned}
&(1+t)^{j-1/2}|||\nabla^j \vecc{\Psi}(t)|||_{H^{s-2j}}^2+\int_0^t (1+\sigma)^{j-1/2}| \nabla^j\vecc{\Psi} (\sigma)|_{\mathbf{H}^{s-2j}}^2\, \textup{d}\sigma\\  
\lesssim& \, \begin{multlined}[t]|||\nabla^j \vecc{\Psi}(0)|||_{H^{s-2j}}^2 +\int_0^t (1+\sigma)^{j-3/2}|||\nabla^j \vecc{\Psi}(\sigma)|||_{H^{s-2j}}^2 \, \textup{d}\sigma\\
+\sum_{\kappa=j}^{s-j-1} \mathdutchcal{M}[v, w, \vecc{U}](t)\Big(\Upsilon^{(j,\kappa)}(t)+\int_0^t (1+\sigma)^{(j-1)-1/2}\mathbf{D}^{\kappa-1}(\sigma)\textup{d}\sigma\Big).   \end{multlined}
\end{aligned}
\end{equation}
Next we need to bound $\sum_{\kappa=j}^{s-j-1}(1+\sigma)^{j-1-1/2}\mathbf{D}^{\kappa}(\sigma)\textup{d}\sigma$ on the right. We have
\begin{equation}
\begin{aligned}
&\sum_{\kappa=j}^{s-j-1} \left(\int_0^t (1+\sigma)^{(j-1)-1/2}\mathbf{D}^{\kappa-1}(\sigma)\textup{d}\sigma\right)\\
=&\, \sum_{\kappa=j+1}^{s-j-1}(1+\sigma)^{(j-1)-1/2}\mathbf{D}^{\kappa-1}(\sigma)\textup{d}\sigma
+\int_0^t (1+\sigma)^{(j-1)-1/2}\mathbf{D}^{j-1}(\sigma)\textup{d}\sigma\\
\leq&\,\sum_{\kappa=j}^{s-j-1}(1+\sigma)^{j-1-1/2}\mathbf{D}^{\kappa}(\sigma)\textup{d}\sigma
+\int_0^t (1+\sigma)^{(j-1)-1/2}\mathbf{D}^{j-1}(\sigma)\textup{d}\sigma.
\end{aligned}
\end{equation}
By setting  $\kappa^\prime=\kappa-(j-1)$, we deduce 
\begin{equation}
\begin{aligned}
\sum_{\kappa=j-1}^{s-j}\Vert \nabla^\kappa \vecc{\Psi}\Vert_{L^2}=\sum_{\kappa^\prime=0}^{s-2j+1}\Vert \nabla^{\kappa'}\nabla^{j-1} \vecc{\Psi}\Vert_{L^2}=&\,\Vert \nabla^{j-1}  \vecc{\Psi} \Vert_{H^{s-2j+1}}\\
=&\,\Vert \nabla^{j-1}  \vecc{\Psi} \Vert_{H^{s-2(j-1)-1}}=|\nabla^{j-1} \vecc{\Psi}|_{\mathbf{H}^{s-2(j-1)}}.
\end{aligned}  
\end{equation}
We thus obtain 
\begin{equation}
\int_0^t (1+\sigma)^{(j-1)-1/2}\mathbf{D}^{\kappa-1}(\sigma)\textup{d}\sigma\lesssim \Upsilon^{j-1,\kappa}(t).
\end{equation}
Furthermore, 
\begin{equation}
\sum_{\kappa=j}^{s-j-1}(1+\sigma)^{j-1-1/2}\mathbf{D}^{\kappa}(\sigma)\textup{d}\sigma\lesssim \sum_{\kappa=j}^{s-j-1}(1+\sigma)^{j-1/2}\mathbf{D}^{\kappa}(\sigma)\textup{d}\sigma\lesssim \sum_{\kappa=j}^{s-j-1}\Upsilon^{j,\kappa}(t).
\end{equation}   
Altogether, we have
\begin{equation} \label{ineq__}
\begin{aligned}  
&\sum_{\kappa=j}^{s-j-1} \Big(\Upsilon^{(j,\kappa)}(t)+\int_0^t (1+\sigma)^{(j-1)-1/2}\mathbf{D}^{\kappa-1}(\sigma)\textup{d}\sigma\Big)\\
\lesssim &\sum_{\kappa=j}^{s-j-1} \Upsilon^{(j,\kappa)}(t)+\int_0^t (1+\sigma)^{(j-1)-1/2}\mathbf{D}^{j-1}(\sigma)\textup{d}\sigma. 
\end{aligned}
\end{equation}
The above estimates imply that
\begin{equation}\label{D_j_1_Estimate}
\begin{aligned}
\int_0^t (1+\sigma)^{(j-1)-1/2}\mathbf{D}^{j-1}(\sigma)\textup{d}\sigma
\lesssim&\, \sum_{\kappa=j-1}^{s-j} \int_0^t (1+\sigma)^{(j-1)-1/2}\mathbf{D}^{\kappa}(\sigma)\textup{d}\sigma\\
\lesssim &\, \int_0^t (1+  \sigma)^{(j-1)-1/2} | \nabla^{j-1}\vecc{\Psi}(\sigma)|_{\mathbf{H}^{s-2(j-1)}}^2 \\   
\lesssim &\,\, ||| \vecc{\Psi}(0)|||_{H^s}^2     
+\mathdutchcal{M}[v, w, \vecc{U}](t)\|\vecc{\Psi}\|_{\mathbbm{D},t}^2,  
\end{aligned}      
\end{equation}
where we have used the induction hypothesis in the last inequality. Furthermore,
\begin{equation}\label{Upsilon_Estimate}
\begin{aligned}
\sum_{\kappa=j}^{s-j-1}\Upsilon^{(j,\kappa)}(t)=&\int_0^t (1+\sigma)^{j-3/2}||| \nabla^j(\psi, v, w,\eta)(\sigma)|||_{H^{s-2j}}^2d\sigma\\
&+\int_0^t (1+\sigma)^{j-1/2}|||\nabla^j (\psi, v, w,\eta)(\sigma)|||_{\mathbf{H}^{s-2j}}^2d\sigma\\
\lesssim &\,\,\|\vecc{\Psi}\|_{\mathbbm{D},t}^2.  
\end{aligned}
\end{equation}
Hence, this estimate taken together with \eqref{D_j_1_Estimate} leads to 
\begin{equation}\label{Inde_main_term}
\begin{aligned}
&(1+t)^{j-1/2}|||\nabla^j \vecc{\Psi}(t)|||_{H^{s-2j}}^2+\int_0^t (1+\sigma)^{j-1/2}| \nabla^j\vecc{\Psi}(\sigma)|_{\mathbf{H}^{s-2j}}^2\, \textup{d}\sigma\\
\lesssim& \, \begin{multlined}[t]|||\nabla^j \vecc{\Psi}(0)|||_{H^{s-2j}}^2 +(j-1/2)\int_0^t (1+\sigma)^{j-3/2}|||\nabla^j \vecc{\Psi}(\sigma)|||_{H^{s-2j}}^2 \, \textup{d}\sigma\\
+\mathdutchcal{M}[v, w, \vecc{U}](t)\|\vecc{\Psi}\|_{\mathbbm{D},t}^2. \end{multlined}
\end{aligned}
\end{equation} 
Now using the induction hypothesis, we can estimate the second term on the right-hand side above as
\begin{equation}\label{Induction_Assump}
\int_0^t (1+\sigma)^{j-3/2}|||\nabla^j \vecc{\Psi}(\sigma)|||_{H^{s-2j}}^2 \, \textup{d}\sigma \lesssim  \mathdutchcal{M}[v, w, \vecc{U}](t)\|\vecc{\Psi}\|_{\mathbbm{D},t}^2
\end{equation}
and further obtain
\begin{equation}\label{Inde_main_term_2}
\begin{aligned}
&\begin{multlined}[t]  (1+t)^{j-1/2}|||\nabla^j \vecc{\Psi}(t)|||_{H^{s-2j}}^2
+\int_0^t (1+\sigma)^{j-1/2}| \nabla^j\vecc{\Psi}(\sigma)|_{\mathbf{H}^{s-2j}}^2\, \textup{d}\sigma \end{multlined}\\
\lesssim&\, \begin{multlined}[t] |||\nabla^j \vecc{\Psi}(0)|||_{H^{s-2j}}^2 +\mathdutchcal{M}[v, w, \vecc{U}](t)\|\vecc{\Psi}\|_{\mathbbm{D},t}^2.
\end{multlined}
\end{aligned}  
\end{equation}  
This proves estimate \eqref{First_Main_Ind}. \\
\indent It remains to prove estimate \eqref{Induction_Key-Estimate_2}. Using $|\mathscr{F}^{(0)}(t)|\lesssim \mathbf{E}_0^2$ obtained in Lemma~\ref{Lemma_Equivl_2}, we have after multiplying by  $(1+t)^{\gamma}$ and integrating with respect to $t$,   
 \begin{equation}   
\begin{aligned} 
\label{Main_Estimate_2}
\int_0^t (1+\sigma)^{\gamma}\mathbf{D}_0^2(\sigma)d\sigma
\leq \mathbf{E}_0^2(0)+ (1+t)^\gamma\mathbf{E}_0^2(t)
+C\int_0^t (1+\sigma)^{\gamma}\mathbf{R}_0(\sigma)d\sigma.
\end{aligned}
\end{equation}
Similarly to before, we can show that for all $\kappa\geq 1$,
 \begin{equation} 
 \begin{aligned}
 \label{Main_Estimate_k_2}
\int_0^t (1+\sigma)^{\gamma}\mathbf{D}_\kappa^2(\sigma)d\sigma\leq&\,\mathbf{E}_\kappa^2(0)+(1+t)^\gamma\mathbf{E}_\kappa^2(t)
 +C\int_0^t (1+\sigma)^{\gamma}\mathbf{R}_\kappa(\sigma)d\sigma;
 \end{aligned} 
\end{equation}
see \eqref{Main_Estimate_k}. We set $\gamma=j-1/2$ in \eqref{Main_Estimate_k_2}. By relying on Theorem~\ref{Thm_R_k} and  summing up over $\kappa$ for $j\leq \kappa\leq s-j-1$,  we obtain 
 \begin{equation}\label{Inde_main_term_3}
\begin{aligned}
&\int_0^t (1+\sigma)^{j-1/2}| \nabla^j \vecc{\Psi}(\sigma)|_{\mathbf{H}^{s-2j}}^2\, \textup{d}\sigma\\  
\lesssim& \begin{multlined}[t]|||\nabla^j \vecc{\Psi}(0)|||_{H^{s-2j}}^2 +(1+t)^{j-1/2}|||\nabla^j \vecc{\Psi}(t)|||_{H^{s-2j}}^2\\
+\mathdutchcal{M}[v, w, \vecc{U}](t)\sum_{\kappa=j}^{s-j-1}\Big(\Upsilon^{(j,\kappa)}(t)+\int_0^t (1+\sigma)^{(j-1)-1/2}\mathbf{D}^{\kappa-1}(\sigma)\textup{d}\sigma\Big). \end{multlined}
\end{aligned}
\end{equation}
Using \eqref{First_Main_Ind}, which has been proven to hold for all $1\leq j\leq [\frac{s-1}{2}]$, we have the estimate 
\begin{equation}\label{Estima_Right_1}
\begin{aligned}
(1+t)^{j-1/2}|||\nabla^j \vecc{\Psi}(t)|||_{H^{s-2j}}^2
\lesssim ||| \vecc{\Psi}(0)|||_{H^s}^2
 +\mathdutchcal{M}[v, w, \vecc{U}](t) \|\vecc{\Psi}\|_{\mathbbm{D},t}^2, 
 \end{aligned}
\end{equation}
where we have also employed 
\begin{equation}\label{Initial_value_Est}
|||\nabla^j \vecc{\Psi}(0)|||_{H^{s-2j}}\lesssim ||| \vecc{\Psi}(0)|||_{H^{s}}, \qquad \sum_{\kappa=j}^{s-j-1}\Upsilon^{(j,\kappa)}(t)\lesssim  \|\vecc{\Psi}\|_{\mathbbm{D},t}^2, 
\end{equation}
and  \eqref{D_j_1_Estimate}.  
Hence, using the fact that $|||\nabla \vecc{\Psi}|||_{H^{s-2}} \lesssim |\vecc{\Psi}|_{\mathbf{H}^{s}} $ and the bounds derived above, estimate \eqref{Inde_main_term_3} further yields
\begin{equation}\label{Inde_main_term_4}
\begin{aligned}
&\int_0^t (1+\sigma)^{j-1/2}||| \nabla^{j+1}(\psi, v, w,\eta)(\sigma)|||_{H^{s-2j-2}}^2\, \textup{d}\sigma\\  
\lesssim& ~|||\vecc{\Psi}(0)|||_{H^{s}}^2 +\mathdutchcal{M}[v, w, \vecc{U}](t) \|\vecc{\Psi}\|_{\mathbbm{D},t}^2,
\end{aligned}
\end{equation}
 which proves \eqref{Induction_Key-Estimate_2}. This last step completes the proof.
  \end{proof}
\subsection{Proof of estimate \eqref{M_weighted_estimate_Main}}  \label{Proof_Decay_Estimates}
To complete our analysis, it remains to prove estimate \eqref{M_weighted_estimate_Main}. We do this next.

 \begin{theorem}
 	\label{Thm_Mcal}  Let $b=\tau c^2$ and $n \geq 3$. Suppose that \[\vecc{\Psi}_{0}=\vecc{\Psi}(t=0)\in (H^{s}(\mathbb{R}^n)\cap L^{1}(\mathbb{R}^n))^3,\]
 	where $s\geq [n/2]+3$. Then the following estimate holds for all $t\in \lbrack 0,T]$:
 	\begin{equation} \label{M_weighted_estimate}
 	\begin{aligned}
 	\mathdutchcal{M}[v, w, \vecc{U}](t) \lesssim&\, \begin{multlined}[t]|||\vecc{\Psi}_0|||_{H^{s}}^2+ \Vert\vecc{U}_0\Vert_{L^1}+ \mathdutchcal{M}^2[v, w, \vecc{U}](t)\vspace{0.2cm}\\
 	+(M_0[v](t)+M[\vecc{U}]_0(t)+M_1[\vecc{U}](t)) \|\vecc{\Psi}\|_{\mathbbm{E},t}.\end{multlined}
 	\end{aligned}%
 	\end{equation}
 \end{theorem}
\noindent Before proceeding to the proof, we note that \[\Vert \nabla ^{j} \vecc{U}( t )
\Vert _{L^2} \leq \Vert \nabla ^{j}\vecc{\Psi}( t )
\Vert _{\mathrm{H}},\]    
where $\mathrm{H}=\dot{H}^{1}(\mathbb{R}^{n })\times \dot{H}^{1}(\mathbb{R}%
^{n})\times L^{2}(\mathbb{R}^{n})\times\mathcal{M}^1 $  is the space endowed with the norm   
\begin{equation} \label{norm}
\begin{aligned}
\left\Vert \vecc{\Psi}\right\Vert _{\mathrm{H}}^{2} = \Vert
\nabla v\Vert _{L^{2}(\mathbb{R}^{n})}^{2}+\Vert \nabla (\psi+\tau v)\Vert
_{L^{2}(\mathbb{R}^{n})}^{2} +\Vert v+\tau w\Vert _{L^{2}(\mathbb{R}%
^{n})}^{2}+\Vert \nabla \eta\Vert^{2}_{L^2, -g'}. 
\end{aligned}
\end{equation}
Keeping in mind how $\mathdutchcal{M}[v, w, \vecc{U}](t)$ is defined, we can estimate it in three steps corresponding to the right-hand side terms; cf. \eqref{Def_Mcal}. We begin with the $\Vert \nabla ^{j}\vecc{\Psi}( t )
\Vert _{\mathrm{H}}$ term.
\newtheorem*{prop4}{Step I}
\begin{prop4}
Under the assumptions of Theorem~\ref{Thm_Mcal}, the following estimate holds:  
\begin{equation}\label{Main_estimate_last}
\begin{aligned}   
 \Vert \nabla^j\vecc{\Psi}(t)\Vert_{\mathrm{H}}\lesssim &~
(\Vert\vecc{U}_0\Vert_{L^1}+\Vert\vecc{U}_0\Vert_{H^s})\left( 1+t\right) ^{-n/4-j/2}\\
&+\left(\mathdutchcal{M}^2[v, w, \vecc{U}](t)
+\mathdutchcal{M}[v, w, \vecc{U}](t) \|\vecc{\Psi}\|_{\mathbbm{E},t} \right) \left( 1+t\right)   
^{-n/4-j/2}
\end{aligned}
\end{equation}
for all $j \in \{0, \ldots, s_0\}$.
\end{prop4}
\begin{proof}
 By applying  the operator $\nabla^\kappa$ to the mild solution \eqref{IP_Psi0} of the problem and then taking the norm in $\mathrm{H}$,  we find that
\begin{equation}
\begin{aligned}
\Vert \nabla^j\vecc{\Psi}(t)\Vert_{\mathrm{H}}\leq&\, \Vert \nabla^je^{t%
\mathcal{A}}\vecc{\Psi}_{0}\Vert_{\mathrm{H}}+\int_0^t \Vert
\nabla^je^{(t-\sigma)\mathcal{A}}\mathbb{F}(\vecc{\Psi},\nabla \vecc{\Psi}%
)(\sigma)\Vert_{\mathrm{H}} \, \textup{d}\sigma := \mathrm{J}_1+\mathrm{J}_2. 
\end{aligned}
\end{equation}
At this point we can relay on the decay rates for the linear problem. In particular, we use \eqref{decay_loss} with $j+\ell=s.$ Since $-(s-j)/2 \leq -n/4-j/2$, we obtain 
 \begin{equation}\label{Linear_Estimate_Main}
\mathrm{J}_1=\Vert \nabla^je^{t%
\mathcal{A}}\vecc{\Psi}_{0}\Vert_{\mathrm{H}} \lesssim \left(\Vert\vecc{U}_{0}\Vert _{L^{1}}+\Vert\vecc{U}_{0}\Vert _{H^{s}}\right)(1+t)^{-{n}/{4-{j}/{2}}}.    
\end{equation}
To estimate $\mathrm{J}_2$, it is convenient to write it as a sum of two integrals over $[0,t/2]$ and $[t/2,t]$, and
treat each of them separately:
\begin{equation} 
\begin{aligned}
\mathrm{J}_2=&\, \int_0^{t/2} \Vert
	\nabla^je^{(t-\sigma )\mathcal{A}}\mathbb{F}(\vecc{\Psi},\nabla \vecc{\Psi}%
	)(\sigma)\Vert_{\mathrm{H}} \, \textup{d}\sigma+\int_{t/2}^{t} \Vert
	\nabla^je^{(t-\sigma )\mathcal{A}}\mathbb{F}(\vecc{\Psi},\nabla \vecc{\Psi}%
	)(\sigma)\Vert_{\mathrm{H}} \, \textup{d}\sigma.
\end{aligned}
\end{equation}
\paragraph{\bf Integral over $\boldsymbol{[0, t/2]}$:} Using the decay rate on $\|\vecc{\Psi}\|_{\mathrm{H}}$ for the linear problem now with $\ell=j+n$, we have
\begin{equation} \label{Estimate_Decay_1}        
\begin{aligned}
&\int_0^{t/2} \Vert
\nabla^je^{(t-\sigma)\mathcal{A}}\mathbb{F}(\vecc{\Psi},\nabla \vecc{\Psi}%
)(\sigma)\Vert_{\mathrm{H}} \, \textup{d}\sigma\\
 \lesssim &\,
\int_0^{t/2} (1+t-\sigma)^{-n/4-j/2}\Vert \mathbb{F}(\vecc{\Psi},\nabla \vecc{\Psi}%
)(\sigma)\Vert_{L^1} \, \textup{d}\sigma  \\
&+ \int_0^{t/2} (1+t-\sigma)^{-n/2-j/2}\Vert \nabla^{2j+n} \mathbb{F}(\vecc{\Psi},\nabla \vecc{\Psi}
)(\sigma)\Vert_{L^2}\, \textup{d}\sigma
=:\,\mathrm{J}_{21}+\mathrm{J}_{22}.  
\end{aligned}
\end{equation}
Using H\"older's inequality yields
\begin{equation}  
\begin{aligned}
\Vert \mathbb{F}(\vecc{\Psi},\nabla \vecc{\Psi}%
)\Vert_{L^1}\lesssim\, \|vw\|_{L^1}+\|\nabla \psi \cdot \nabla v\|_{L^1} 
\lesssim\, \Vert\vecc{U}\Vert_{L^2}^2+\Vert
v\Vert_{L^2}^2.
\end{aligned}
\end{equation}
Therefore, \[\Vert\vecc{U}(\sigma)\Vert_{L^2}+\Vert v(\sigma)\Vert_{L^2}\lesssim (1+\sigma)^{-n/4}\mathdutchcal{M}[v, w, \vecc{U}](\sigma)\lesssim (1+\sigma)^{-n/4}\mathdutchcal{M}[v, w, \vecc{U}](t).\] 
Hence, we further have
\begin{equation}    \label{J_1_Estimate_I}
\begin{aligned}
\mathrm{J}_{21} \lesssim    &\,\mathdutchcal{M}^2[v, w, \vecc{U}](t)\int_0^{t/2} (1+t-\sigma)^{-n/4-j/2}(1+\sigma)^{-n/2} \, \textup{d}\sigma\\
\lesssim& \, \mathdutchcal{M}^2[v, w, \vecc{U}](t)(1+t)^{-n/4-j/2} \int_0^{t/2} (1+\sigma)^{-n/2} \, \textup{d}\sigma \\
\lesssim&\, \mathdutchcal{M}^2[v, w, \vecc{U}](t)(1+t)^{-n/4-j/2},
\end{aligned}
\end{equation}
on account of our assumptions on $n$. Next we wish to estimate $\mathrm{J}_{22}$. We have for $\kappa=2j+n$, 
  \begin{equation}
\begin{aligned}
\Vert \nabla^\kappa \mathbb{F}(\vecc{\Psi},\nabla \vecc{\Psi}%
)\Vert_{L^2}\lesssim&\,  \Vert \nabla^\kappa (v
w)\Vert_{L^2}+\Vert \nabla^\kappa (\nabla \psi \nabla v)\Vert_{L^2} \\
\lesssim & \, \Vert \nabla^\kappa (v (v+\tau w))\Vert_{L^2}+ \Vert \nabla^\kappa (v^2
)\Vert_{L^2} +\Vert \nabla^\kappa (\nabla \psi \nabla v)\Vert_{L^2}.
\end{aligned}
\end{equation}
This further yields 
\begin{equation}\label{L_2_Estimate_F}
\begin{aligned}
&\Vert \nabla^{n+2j} \mathbb{F}(\vecc{\Psi},\nabla \vecc{\Psi}%
)(\sigma)\Vert_{L^2} \\
\lesssim&\, \begin{multlined}[t]\Vert v\Vert_{L^\infty}\left(\Vert \nabla^{n+2j} (v+\tau
w)\Vert_{L^2}+\Vert \nabla^{n+2j} v\Vert_{L^2}\right) \vspace{0.2cm}\\
+\Vert v+\tau w\Vert_{L^\infty}\Vert \nabla^{n+2j} v\Vert_{L^2}+\Vert\nabla v
\Vert_{L^\infty}\Vert \nabla^{n+2j}\nabla v\Vert_{L^2} \vspace{0.2cm}\\
+\Vert\nabla v \Vert_{L^\infty}\Vert \nabla^{n+2j}\nabla (\psi+\tau v)\Vert_{L^2}
+\Vert\nabla (\psi+\tau v) \Vert_{L^\infty}\Vert \nabla^{2j+n}\nabla v \Vert_{L^2} \end{multlined}
\vspace{0.2cm}\\
\lesssim&\,\begin{multlined}[t] (1+\sigma)^{-n/2}(M_0[\vecc{U}](\sigma)+M_0[v](\sigma))\Vert\nabla^{n+2j}\vecc{U}\Vert_{L^2} \vspace{0.2cm}\\   
{+ (1+\sigma)^{-n/2}(M_0[\vecc{U}](\sigma)+M_0[v](\sigma))\Vert \nabla^{2j+n}  v \Vert_{L^2}};\end{multlined}
\end{aligned}
\end{equation}
see inequality \eqref{First_inequaliy_Guass}. Recalling how the norm $\|\cdot\|_{\mathbbm{E},t}$ is defined in \eqref{Energy_Main}, we obtain 
\begin{equation}
\begin{aligned}
\Vert\nabla^{n+2j}\vecc{U}(\sigma)\Vert_{L^2} \lesssim& \, \Vert \nabla^n\vecc{U}(\sigma)\Vert_{H^{s-2n}} \hspace*{1.5cm} \text{if } s\geq 2n+2j  \\[1mm]
\lesssim&\, |||\nabla^n\vecc{\Psi}(\sigma) |||_{H^{s-2n+1}}\\
\lesssim&\, |||\nabla^{n-1}\vecc{\Psi} (\sigma)|||_{H^{s-2(n-1)}}\\
 \lesssim&\, (1+\sigma)^{-n/2+3/4}  \|\vecc{\Psi}\|_{\mathbbm{E},\sigma}\\
 \lesssim&\, (1+\sigma)^{-n/2+3/4}  \|\vecc{\Psi}\|_{\mathbbm{E},t}. 
\end{aligned}   
\end{equation}
Similarly, we have 
\begin{equation}
\begin{aligned}
\Vert \nabla^{n+2j} v(\sigma)\Vert_{L^2}\lesssim&\, \Vert \nabla^{n+2j}w(\sigma)\Vert_{L^2} +\Vert \nabla^{n+2j}(v+\tau w)(\sigma)\Vert_{L^2}\\
\lesssim &\, \Vert \nabla^{n+2j}w(\sigma)\Vert_{L^2}+\Vert\nabla^{n+2j}\vecc{U}(\sigma)\Vert_{L^2} \\
\lesssim &\, |||\nabla^n\vecc{\Psi}(\sigma) |||_{H^{s-2n+1}}\\
\lesssim &\,\Vert \nabla^{n-1}\vecc{\Psi}(\sigma)\Vert_{H^{s-2(n-1)}}.
\end{aligned}
\end{equation}
Since $n \geq 3$, we further have
\begin{equation}\label{J_2_Estimate_I}
\begin{aligned}
&\int_0^{t/2} (1+t-\sigma)^{-n/2-j/2}\Vert \nabla^{2j+n} \mathbb{F}(\vecc{\Psi},\nabla \vecc{\Psi}%
)(\sigma)\Vert_{L^2}\, \textup{d}\sigma \\  
\lesssim&\,(M_0[\vecc{U}](t)+M_0[v](t)+M_1[\vecc{U}](t))\|\vecc{\Psi}\|_{\mathbbm{E},t}\int_0^{t/2} (1+t-\sigma)^{-n/2-j/2}(1+\sigma)^{-n/2}\, \textup{d}\sigma\\
  \lesssim&\, (M_0[\vecc{U}](t)+M_0[v](t)+M_1[\vecc{U}](t))\|\vecc{\Psi}\|_{\mathbbm{E},t}(1+t)^{-n/4-j/2},
\end{aligned}
\end{equation}
where we have used
 \begin{equation}
\begin{aligned}
&\int_0^{t/2} (1+t-\sigma)^{-n/2-j/2}(1+\sigma)^{-n/2}d\sigma\\
\leq&\, (1+t/2)^{-n/4-j/2}\int_0^{t/2} (1+t-\sigma)^{-n/4}(1+\sigma)^{-n/2}d\sigma \lesssim  (1+t/2)^{-n/4-j/2}.
\end{aligned}  
\end{equation}
see \cite[Lemma 7.4]{RR2015} for similar arguments. Thanks to the bounds derived above, we find that
 \begin{equation}\label{I_1_1_Estimate_Main}
 \begin{aligned}
&\int_0^{t/2} \Vert
\nabla^je^{(t-r )\mathcal{A}}\mathbb{F}(\vecc{\Psi},\nabla \vecc{\Psi}%
)(\sigma)\Vert_{\mathrm{H}} \, \textup{d}\sigma\\
\lesssim&\, \big(\mathdutchcal{M}^2[v, w, \vecc{U}]+(M_0(t)+M_0[v](t)+M_1[\vecc{U}](t))\|\vecc{\Psi}\|_{\mathbbm{E},t}\big)(1+t)^{-n/4-j/2}. 
\end{aligned}
\end{equation}
~\vspace*{1mm}
\paragraph{\bf Integral over $\boldsymbol{[t/2, t]}$:}  Next, we estimate the integral over $[t/2, t]$ in \eqref{I2}. By applying the linear decay rate \eqref{v_L_2_Estimate} with $j=1$ and $\ell=n$ and with $\nabla^{j-1}\mathbb{F}(\vecc{\Psi},\nabla \vecc{\Psi}%
)$ instead of $\vecc{U}_0$,  we obtain 
\begin{equation}  
\begin{aligned}
\int_{t/2}^t \Vert
\nabla^je^{(t-r )\mathcal{A}}\mathbb{F}(\vecc{\Psi},\nabla \vecc{\Psi}%
)(\sigma)\Vert_{\mathrm{H}} \, \textup{d}\sigma  =&\,\int_{t/2}^t\left\Vert
\nabla e^{(t-r )\mathcal{A}}\nabla^{j-1}\mathbb{F}(\vecc{\Psi},\nabla \vecc{\Psi}%
)(\sigma)\right\Vert_{\mathrm{H}} \, \textup{d}\sigma\\
\lesssim&\,
\int_{t/2}^t (1+t-\sigma)^{-n/  4-1/2}\Vert\nabla^{j-1}\mathbb{F}(\vecc{\Psi},\nabla \vecc{\Psi}%
)(\sigma) \Vert_{L^1}\, \textup{d}\sigma\\
&+\int_{t/2}^t (1+t-\sigma)^{-n/2}\Vert\nabla^{j+n}\mathbb{F}(\vecc{\Psi},\nabla \vecc{\Psi}%
)(\sigma) \Vert_{L^2}\, \textup{d}\sigma\\
=:&\, \mathrm{J}_{23}+\mathrm{J}_{24}.
\end{aligned}
\end{equation}
We proceed to estimate the two terms on the right. We have 
  \begin{equation}
  \begin{aligned}
  \Vert\nabla^{j-1}\mathbb{F}(\vecc{\Psi},\nabla \vecc{\Psi}%
)(\sigma) \Vert_{L^1}\lesssim&\, \Vert \nabla^{j-1} (vw)\Vert_{L^1}+\Vert \nabla^{j-1} (\nabla \psi \cdot \nabla v )\Vert_{L^1}\\
\lesssim&\,\Vert v\Vert_{L^2}\Vert\nabla^{j-1}w \Vert_{L^2}+ \Vert w\Vert_{L^2}\Vert\nabla^{j-1}v \Vert_{L^2}\\
&+\Vert \nabla \psi\Vert_{L^2}\Vert\nabla^{j}v \Vert_{L^2}+\Vert \nabla v\Vert_{L^2}\Vert\nabla^{j}\psi \Vert_{L^2},
  \end{aligned}
\end{equation}  
where we have relied on the inequality \eqref{First_inequaliy_Guass}. Using 
\begin{equation}\label{Ineq_w_U_v}
\Vert\nabla^{j-1}w \Vert_{L^2}\lesssim \Vert \nabla^{j-1}\vecc{U}\Vert_{L^2}+ \Vert \nabla^{j-1} v\Vert_{L^2},
\end{equation}
and recalling the definition of $\mathdutchcal{M}[v, w, \vecc{U}]$ in \eqref{Def_Mcal}, we obtain 
 \begin{equation}
  \begin{aligned}
  \Vert\nabla^{j-1}\mathbb{F}(\vecc{\Psi},\nabla \vecc{\Psi}%
)(\sigma) \Vert_{L^1}\lesssim&\, \begin{multlined}[t](1+t)^{-n/4}(1+t)^{-n/4-(j-1)/2}\mathdutchcal{M}^2[v, w, \vecc{U}](t)\\    
+(1+t)^{-n/4-1/2}(1+t)^{-n/4-(j-1)/2}\mathdutchcal{M}^2[v, w, \vecc{U}](t)\\  
+(1+t)^{-n/4}( 1+t
) ^{-n/4-j/2}\mathdutchcal{M}^2[v, w, \vecc{U}](t) \end{multlined}\\  
\lesssim&\, (1+t)^{-n/2-(j-1)/2}\mathdutchcal{M}^2[v, w, \vecc{U}](t)  
\end{aligned}
\end{equation}  
for $j\leq s_0$. On account of our assumption on $n$, we then have
\begin{equation}  \label{J_3_Estimate_Main_1}
\begin{aligned}  
\mathrm{J}_{23}\lesssim & \,\mathdutchcal{M}^2[v, w, \vecc{U}](t)\int_{t/2}^t (1+t-\sigma)^{-n/  4-1/2}(1+\sigma)^{-n/2-(j-1)/2}\, \textup{d}\sigma\\
\lesssim&\,\mathdutchcal{M}^2[v, w, \vecc{U}](t) (1+t)^{-n/2-(j-1)/2}\int_{t/2}^t (1+t-\sigma)^{-n/  4-1/2}\, \textup{d}\sigma\\
\lesssim&\,\mathdutchcal{M}^2[v, w, \vecc{U}](t) (1+t)^{-n/2-(j-1)/2}\\  
\lesssim&\,\mathdutchcal{M}^2[v, w, \vecc{U}](t) (1+t)^{-n/4-j/2}. 
\end{aligned}    
\end{equation}  
\indent Now, to estimate $\mathrm{J}_{24}$, we first use \eqref{L_2_Estimate_F} for $\kappa=j+n$. In this manner, we obtain    
\begin{equation}\label{L_2_Estimate_F_2}
\begin{aligned}
\Vert \nabla^\kappa \mathbb{F}(\vecc{\Psi},\nabla \vecc{\Psi}%
)(\sigma)\Vert_{L^2} 
\lesssim&\, (1+t)^{-n/2}(M_0[\vecc{U}](t)+M_0[v](t))\Vert\nabla^{n+j}\vecc{U}\Vert_{L^2}\\
& + (1+t)^{-n/2-1/2}M_1[\vecc{U}](t)\Vert\nabla^{n+j}\vecc{U}\Vert_{L^2}\\
\lesssim&\, (M_0[\vecc{U}](t)+M_0[v](t)+M_1[\vecc{U}](t))(1+t)^{-n/2}\Vert\nabla^{n+j}\vecc{U}\Vert_{L^2}. 
\end{aligned}
\end{equation}
For $j\leq [s/2]-n-1 $, we have
\begin{equation}   
\begin{aligned}  
\Vert\nabla^{n+j}\vecc{U}\Vert_{L^2}\leq\Vert\nabla^{n+j}\vecc{U}\Vert_{H^{s-2j-2n}}\lesssim &\,\|\vecc{\Psi}\|_{\mathbbm{E},t}(1+t)^{-n/2-j/2+1/2}\\
\lesssim&\,\|\vecc{\Psi}\|_{\mathbbm{E},t}(1+t)^{-n/4-j/2}  
 \end{aligned}
\end{equation}  
because $n \geq 3$. Hence,
 \begin{equation}\label{J_4_Estimate_Main_1}
\begin{aligned}      
\mathrm{J}_{24}\lesssim &\,(M_0[\vecc{U}](t)+M_0[v](t)+M_1[\vecc{U}](t))\|\vecc{\Psi}\|_{\mathbbm{E},t}\int_{t/2}^t (1+t-\sigma)^{-n/2}(1+t)^{-3n/4-j/2}\, \textup{d} \sigma\\
\lesssim&\,(M_0[\vecc{U}](t)+M_0[v](t)+M_1[\vecc{U}](t))\|\vecc{\Psi}\|_{\mathbbm{E},t}(1+t)^{-n/4-j/2}. 
\end{aligned}
\end{equation}   
By adding estimate \eqref{J_3_Estimate_Main_1} to the above inequality, we obtain 
\begin{equation}
\begin{aligned}
&\int_{t/2}^t \Vert
\nabla^je^{(t-r )\mathcal{A}}\mathbb{F}(\vecc{\Psi},\nabla \vecc{\Psi}%
)(\sigma)\Vert_{\mathrm{H}} \, \textup{d}\sigma \\
\lesssim&\, \left(\mathdutchcal{M}^2[v, w, \vecc{U}](t)+(M_0[\vecc{U}](t)+M_0[v](t)+M_1[\vecc{U}](t))\|\vecc{\Psi}\|_{\mathbbm{E},t}\right)(1+t)^{-n/4-j/2}\\
\lesssim&\, \left(\mathdutchcal{M}^2[v, w, \vecc{U}](t)+\mathdutchcal{M}^2[v, w, \vecc{U}](t)\|\vecc{\Psi}\|_{\mathbbm{E},t}\right)(1+t)^{-n/4-j/2}.
\end{aligned}
\end{equation}
Collecting the derived estimates completes Step I.
\end{proof}

\newtheorem*{prop6}{Step II} 
\begin{prop6}
Under the assumptions of Theorem~\ref{Thm_Mcal}, the following inequality holds for all $j\in\{0,\dots, s_0-1\}$:
\begin{equation}  \label{Main_estimate_last_3}
\begin{aligned}   
\left\Vert \nabla^{j}w( t )  
\right\Vert _{L^2}\lesssim &\,
(\Vert \nabla^j \psi_2 
\Vert _{L^2} +\Vert\vecc{U}_0\Vert_{L^1}+  \Vert\vecc{U}_0\Vert_{H^s})\left( 1+t\right) ^{-n/4-1/2-j/2}\\
&+\left(\mathdutchcal{M}^2[v, w, \vecc{U}]   
+\mathdutchcal{M}[v, w, \vecc{U}](t) \|\vecc{\Psi}\|_{\mathbbm{E},t} \right)  \left( 1+t\right)
^{-n/4-1/2-j/2}.
\end{aligned}
\end{equation}
\end{prop6}
\begin{proof}
We first prove the estimate when $j=0$. We multiply the third equation in \eqref{Main_System} by $w$ and integrate with respect to space to arrive at
\begin{equation}\label{w_Eenrgy_Nonl}
\begin{aligned}
\frac{{\tau}}{2}\frac{\textup{d}}{\textup{d}t}\Vert w\Vert_{L^2} ^{2}+\Vert w\Vert_{L^2} ^{2}\lesssim&\, \Vert \Delta \psi\Vert_2^2+ \Vert\Delta \eta\Vert^{2}_{L^2, g}+\int_{\R^n}|( vw+\nabla \psi \cdot \nabla v)w | \dx\\
\lesssim&\, \Vert \Delta (\psi+\tau v)\Vert_{L^{2}}^{2}+\Vert \Delta v\Vert_{L^{2}}^{2}+\Vert\Delta \eta\Vert^{2}_{L^2, g}\\
&+\epsilon \Vert w\Vert_{L^2}^2+ \frac{1}{4\epsilon}\Vert  vw+\nabla \psi \cdot \nabla v\Vert_2^2, 
\end{aligned}  
\end{equation}  
with $\epsilon>0$. Let $C>0$ be the hidden constant in the above estimate. Then
\begin{equation}
\begin{aligned}
\frac{{\tau}}{2}\frac{\textup{d}}{\textup{d}t}\Vert w\Vert_{L^2} ^{2}+(1-\varepsilon C )\Vert w\Vert_{L^2} ^{2}\lesssim&\,\Vert \nabla\vecc{\Psi}\Vert_{\mathrm{H}}^2+\Vert  vw+\nabla \psi \cdot \nabla v\Vert_{L^2}^2\\
\lesssim &\, \Vert \nabla\vecc{\Psi}\Vert_{\mathrm{H}}^2+\Vert v\Vert_{L^\infty}^2\Vert w\Vert_{L^2}^2+\Vert \nabla \psi\Vert_{L^\infty}^2\Vert \nabla v\Vert_{L^2}^2.
\end{aligned}   
\end{equation}
We fix $\epsilon>0$ such that $1-C\epsilon\geq 1/2$ and multiply the above inequality by $e^{\frac{1}{{\tau}} t}$. In this manner, we arrive at
\begin{equation}
\begin{aligned}
\frac{\textup{d}}{\textup{d}t} \left(e^{\frac{1}{{\tau}} t}\Vert w\Vert_{L^2} ^{2}\right)\lesssim {\frac{1}{\tau}}e^{\frac{1}{{\tau}} t}\left(\Vert \nabla\vecc{\Psi}\Vert_{\mathrm{H}}^2+\Vert v\Vert_{L^\infty}^2\Vert w\Vert_{L^2}^2+\Vert \nabla \psi\Vert_{L^\infty}^2\Vert \nabla v\Vert_{L^2}^2\right).     
\end{aligned}
\end{equation}
By then integrating with respect to time, we obtain 
\begin{equation}\label{Energy_First_Order}
\begin{aligned}
 \Vert w\Vert_{L^2} ^{2}\lesssim&\, \begin{multlined}[t] \ e^{-\frac{1}{{\tau}}t}\Vert \psi_2 \Vert_{L^2} ^{2}+{\frac{1}{\tau}}\int_0^t e^{-{\frac{1}{\tau}}(t-\sigma)}\left(\Vert \nabla\vecc{\Psi}(\sigma)\Vert_{\mathrm{H}}^2+\Vert v(\sigma)\Vert_{L^\infty}^2\Vert w(\sigma)\Vert_{L^2}^2\right.\\
 \left.+\Vert \nabla \psi(\sigma)\Vert_{L^\infty}^2\Vert \nabla v(\sigma)\Vert_{L^2}^2\right)\, \textup{d}\sigma. \end{multlined}
\end{aligned}
\end{equation}
Applying the decay rate \eqref{Main_estimate_last} from Step I yields
\begin{equation}
\begin{aligned}   
\Vert \nabla\vecc{\Psi}\Vert_{\mathrm{H}}\lesssim&\, (\Vert\vecc{U}_0\Vert_{L^1}+\Vert\vecc{U}_0\Vert_{H^s})\left( 1+t\right) ^{-n/4-1/2}\\
&+\left(\mathdutchcal{M}^2[v, w, \vecc{U}](t)
+\mathdutchcal{M}[v, w, \vecc{U}](t) \|\vecc{\Psi}\|_{\mathbbm{E},t} \right) \left( 1+t\right)   
^{-n/4-1/2}. 
\end{aligned}
\end{equation}  
Additionally using an elementary integral inequality 
 \begin{equation}
\int_0^t e^{-\gamma(t-\sigma)}(1+\sigma)^{-\beta}\, \textup{d} \sigma\lesssim (1+t)^{-\beta},   
\end{equation}
cf. \eqref{Expon_Poly_Ineq}, leads to
\begin{equation}\label{Phi_w_Estimate}
\begin{aligned}   
\int_0^t e^{-{\frac{1}{\tau}}(t-\sigma)}\Vert \nabla\vecc{\Psi}(\sigma)\Vert_{\mathrm{H}}^2\, \textup{d}\sigma
\lesssim\,\begin{multlined}[t] (\Vert\vecc{U}_0\Vert_{L^1}+\Vert\vecc{U}_0\Vert_{H^s})^2\left( 1+t\right) ^{-n/2-1}\\
+\left(\mathdutchcal{M}^2[v, w, \vecc{U}](t)
+\mathdutchcal{M}[v, w, \vecc{U}](t) \|\vecc{\Psi}\|_{\mathbbm{E},t} \right)^2 \left( 1+t\right)   
^{-n/2-1}. \end{multlined}
\end{aligned}   
\end{equation}
On the other hand, by using the estimates 
\begin{alignat}{2}
\|v\|_{L^\infty}\lesssim&\,  (1+t)^{-n/2}M_0[v](t), \qquad
&&\Vert w\Vert_{L^2}\lesssim\, (1+t)^{-n/4-1/2}\mathdutchcal{M}[v, w, \vecc{U}](t),\\
\Vert \nabla v\Vert_{L^2}\lesssim&\,(1+t)^{-n/4-1/2}\mathdutchcal{M}[v, w, \vecc{U}](t), \qquad
&&\Vert \nabla   \psi\Vert_{L^\infty}\lesssim\, (1+t)^{-n/2}M_0[\vecc{U}](t),
\end{alignat}
together with again \eqref{Expon_Poly_Ineq}, we obtain 
\begin{equation}\label{Second_w_Estimate}
\begin{aligned}
\int_0^t &e^{-{\frac{1}{\tau}}(t-\sigma)} \Vert v(\sigma)\Vert_{L^\infty}^2\Vert w(\sigma)\Vert_{L^2}^2\, \textup{d}\sigma \\
\lesssim&\, (M_0[v](t))^2\big(\mathdutchcal{M}[v, w, \vecc{U}](t)\big)^2\int_0^te^{-{\frac{1}{\tau}}(t-\sigma)}(1+\sigma)^{-\frac{3n}{2}-1}\, \textup{d}\sigma \\
\lesssim&\,(M_0[v](t))^2\big(\mathdutchcal{M}[v, w, \vecc{U}](t)\big)^2 (1+t)^{-\frac{3n}{2}-1}\\
\lesssim&\,(M_0[v](t))^2\big(\mathdutchcal{M}[v, w, \vecc{U}](t)\big)^2 (1+t)^{-\frac{n}{2}-1}. 
\end{aligned}
\end{equation}  
Similarly, we have 
 \begin{equation}\label{Third_w_Estimate}
\begin{aligned}
\int_0^t &e^{-{\frac{1}{\tau}}(t-\sigma)} \Vert \nabla \psi(\sigma)\Vert_{L^\infty}^2\Vert \nabla v(\sigma)\Vert_{L^2}^2 \, \textup{d}\sigma\\
\lesssim&\, (M_0[\vecc{U}](t))^2\big(\mathdutchcal{M}[v, w, \vecc{U}](t)\big)^2\int_0^te^{-{\frac{1}{\tau}}(t-\sigma)}(1+\sigma)^{-(n+1)}\, \textup{d}\sigma \\
\lesssim&\,(M_0[\vecc{U}](t))^2\big(\mathdutchcal{M}[v, w, \vecc{U}](t)\big)^2(1+t)^{-\frac{n}{2}-1}. 
\end{aligned}
\end{equation}
By collecting the derived bounds and making use of Lemma \ref{M_estimate}, we obtain estimate \eqref{Main_estimate_last_3} when $j=0$.\\
\indent We next wish to derive the corresponding bound for $j \geq 1$. To this end,  we apply the operator $\nabla^j,\, j\geq 1$ to the third equation in \eqref{Main_System}, and multiply the resulting equation  by $\nabla^j w$. Similarly to \eqref{Energy_First_Order}, we obtain 
\begin{equation}\label{Energy_higher_Order_w}
\begin{aligned}
 \Vert \nabla^jw\Vert_{L^2} ^{2}\lesssim&\, \begin{multlined}[t] \ e^{-{\frac{1}{\tau}}t}\Vert \nabla^j \psi_2\Vert_{L^2} ^{2}+\int_0^t e^{-{\frac{1}{\tau}}(t-\sigma)}\Big(\Vert \nabla^{j+1}\vecc{\Psi}(\sigma)\Vert_{\mathrm{H}}^2+\Vert \nabla^j( vw+\nabla \psi \cdot \nabla v)\Vert_{L^2}^2\Big)\, \textup{d}\sigma. \end{multlined}
\end{aligned}
\end{equation}
We can further estimate the right-hand side by noting that
\begin{equation}\label{Prod_1}
\begin{aligned}
\Vert \nabla^j( vw)\Vert_{L^2}\lesssim&\, \Vert w\Vert_{L^\infty} \Vert \nabla^jv\Vert_{L^2}+\Vert v\Vert_{L^\infty} \Vert \nabla^jw\Vert_{L^2}\\
\lesssim&\, (\Vert v+\tau w\Vert_{L^\infty}+\Vert v\Vert_{L^\infty} )\Vert \nabla^jv\Vert_{L^2}+\Vert v\Vert_{L^\infty} \Vert \nabla^jw\Vert_{L^2}\\
\lesssim&\,(M_0[\boldsymbol{\vecc{U}}](t)+M_0[v](t)) (1+t)^{-n/2}\mathdutchcal{M}[v, w, \vecc{U}](t)(1+t)^{-n/4-j/2}\\
&+M_0[v](t) (1+t)^{-n/2}\mathdutchcal{M}[v, w, \vecc{U}](t)(1+t)^{-n/4-j/2-1/2}\\
\lesssim&\,\Big(M_0[\boldsymbol{\vecc{U}}](t)+M_0[v](t)\Big)\mathdutchcal{M}[v, w, \vecc{U}](1+t)^{-n/4-j/2-1/2},
\end{aligned}
\end{equation}
where we have utilized the well-known product estimate \eqref{First_inequaliy_Guass}. Similarly, we have 
\begin{equation}\label{Prod_2}
\begin{aligned}
\Vert \nabla^j( \nabla \psi \cdot \nabla v)\Vert_{L^2}\lesssim&\, \Vert \nabla \psi\Vert_{L^\infty}\Vert \nabla^{j+1} v\Vert_{L^2} +\Vert \nabla v\Vert_{L^\infty}\Vert \nabla^{j+1} \psi\Vert_{L^2}\\
\lesssim&\, M_0[\boldsymbol{\vecc{U}}](t)(1+t)^{-n/2}\mathdutchcal{M}[v, w, \vecc{U}](1+t)^{-n/4-j/2}\\
\lesssim&\,M_0[\boldsymbol{\vecc{U}}](t)(1+t)^{-n/4-j/2-1/2}. 
\end{aligned}
\end{equation}
Furthermore, we have
\begin{equation}
\begin{aligned}
\Vert \nabla^{j+1}\vecc{\Psi}(\sigma)\Vert_{\mathrm{H}}\lesssim&\,\begin{multlined}[t] (\Vert\vecc{U}_0\Vert_{L^1}+\Vert\vecc{U}_0\Vert_{H^s})^2\left( 1+t\right) ^{-n/4-1/2-j/2}\\
+\left(\mathdutchcal{M}^2[v, w, \vecc{U}](t)  
+\mathdutchcal{M}[v, w, \vecc{U}](t) \|\vecc{\Psi}\|_{\mathbbm{E},t} \right)^2 \left( 1+t\right)   
^{-n/4-1/2-j/2}.  \end{multlined}
\end{aligned}      
\end{equation}
By combining the derived bounds, the estimate follows analogously to the case $j=0$. We omit the details. 
\end{proof}
\newtheorem*{prop5}{Step III}
\begin{prop5}
Under the assumptions of Theorem~\ref{Thm_Mcal}, the following bound holds:
	\begin{equation}\label{Main_estimate_last_2}
	\begin{aligned} 
	\left\Vert \nabla^j v( t )
	\right\Vert _{L^2}\lesssim &~
	(\Vert \nabla^j \psi_2  
\Vert _{L^2}+\Vert\vecc{U}_0\Vert_{L^1}+\Vert\vecc{U}_0\Vert_{H^{s}})\left( 1+t\right) ^{-n/4-j/2
	 }\\
	&+\left(\mathdutchcal{M}^2[v, w, \vecc{U}](t)
	+\mathdutchcal{M}[v, w, \vecc{U}](t)\|\vecc{\Psi}\|_{\mathbbm{E},t} \right) \left( 1+t\right)
	^{-n/4-j/2}  
	\end{aligned}
	\end{equation}
	for $j \in \{0, \ldots, s_0-1\}$.
\end{prop5}
\begin{proof}
By utilizing the estimate
\begin{equation}  
\begin{aligned}
\Vert \nabla^jv\Vert_{L^2}\lesssim\, \Vert \nabla^jw\Vert_{L^2}+\Vert\nabla^j( v+\tau w)\Vert_{L^2}
\lesssim\,\Vert \nabla^j w\Vert_{L^2}+\left\Vert  \nabla^j\vecc{\Psi} ( t )
	\right\Vert _{\mathrm{H}},
\end{aligned}
\end{equation}
the claim immediately follows from the bounds derived in Steps I and II. \\
\indent Steps I--III complete the proof of Theorem~\ref{Thm_Main_1_repeat}.
\end{proof} 
 We have thus derived both estimates \eqref{Main_Estimate_E_D_Main} and \eqref{M_weighted_estimate_Main}, which were missing to complete the proof of our main result stated in Theorem~\ref{Main_Theorem}.
\section*{Conclusion and outlook}
\indent In this work, we have investigated the asymptotic behavior of solutions to the Jordan--Moore--Gibson--Thompson equation in inviscid media with type I memory. Due to the critical condition $b= \tau c^2$ satisfied by the medium parameters, the linearized equation's decay estimates are of a regularity-loss type. Such a loss of derivatives prevents the use of classical energy methods in the analysis of the corresponding nonlinear problem. Our approach instead relied on devising appropriate time-weighted norms, which helped introduce artificial damping to the problem. In turn, this damping allowed us to control the loss of derivatives and the problem's nonlinearity. \\
\indent As is the case in the analysis of many nonlinear PDEs, the unpleasant restriction on the initial data size is required. It remains an interesting open question to show  finite-time blow-up for large initial data. It has been numerically observed that for the Kuznetsov equation, formally obtained in the limiting case when $\tau \rightarrow 0^+$ and $g=0$ in \eqref{Main_Equation}, if the sound diffusivity is negligible, gradient blow-up occurs after a certain time; see, for example,~\cite{walsh2007finite_} for the revealing numerical experiments. The theoretical justification of this observation remains an open problem as well.  \\
\indent There are also many other directions of possible future research. Our assumption on the memory kernel $g$ was that of exponential decay. It would be interesting to consider how a polynomially decaying kernel influences the behavior of solutions. Singular limits for vanishing thermal relaxation are expected to lead to nonlocal second-order acoustic models and so their rigorous justification is of interest as well; see, for example,~\cite{bongarti2018singular, chen2020asymptotic} for such studies carried out for MGT equations and~\cite{KaltenbacherNikolic} for the nonlinear JMGT equation without memory.  \\
\indent The presence of the diffusion term $b\Delta \psi_t$ with $b>0$ was crucial in the analysis of equation \eqref{Main_Equation}. It is known that for $b=0$ and in the absence of the memory term, the problem is ill-posed; see~\cite{Fattorini_1983}.  Having in mind nonlinear sound propagation through biological tissues, another important question is to consider fractional diffusion $b(-\Delta)^\alpha$ for $0\leq \alpha\leq 1$. We leave these open questions for future work.
\appendix

\begin{appendices}   
 
\section{ Auxiliary inequalities} \label{AppendixA}
\noindent We gather here technical inequalities that have been frequently used in the preceding sections, as well as the known embedding results and the general Gagliardo--Nirenberg interpolation inequality. 
 \begin{lemma} (See~\cite{Se68})
 	Let $a>0$ and $b>0$. Then, 
 	\begin{equation} \label{Ineq_Segel} \label{First_integral_inequality}
 	\int_{0}^{t}\left( 1+t-s\right) ^{-a}\left( 1+s\right) ^{-b}ds\leq C\left(
 	1+t\right) ^{-\min   \left( a,b\right) }\quad\text{if  }\max (a,b)>1,
 	\end{equation}
 where constant $C>0$ does not depend on $t$.	
 \end{lemma}
\begin{lemma}
For any $t\geq 0$ and $\alpha$, $\sigma$, $\gamma>0$, it holds that   
\begin{equation}\label{Eq_Expo_1}
e^{-\gamma t^{\alpha}}\lesssim \gamma^{-\frac{\sigma}{\alpha}}(1+t)^{-\sigma}.
\end{equation}

\end{lemma}
\begin{lemma}  
For any  $\gamma>0$, for any $t\geq 0$ and for any $\beta>0$, the following inequality holds: 
 \begin{equation}\label{Expon_Poly_Ineq}
\int_0^t e^{-\gamma(t-\sigma)}(1+\sigma)^{-\beta}\, \textup{d} \sigma\lesssim (1+t)^{-\beta} .   
\end{equation}
\end{lemma}
\begin{proof}
We have by using   \eqref{Ineq_Segel} together with \eqref{Eq_Expo_1} with $\sigma>\max(1,\beta)$ that
\begin{equation}
\begin{aligned}
\int_0^t e^{-\gamma(t-\sigma)}(1+\sigma)^{-\beta}\, \textup{d} \sigma\lesssim& \int_{0}^{t}\left( 1+t-s\right) ^{-\sigma}\left( 1+s\right) ^{-\beta}ds\\
\lesssim&\, (1+t)^{-\beta}, 
\end{aligned}
\end{equation}
which yields the desired result. 
\end{proof}
\begin{lemma}[See Lemma 4.1 in~\cite{HKa06}]
\label{Guass_symbol_lemma} Let $1\leq p,\,q,\,r\leq \infty $ and $%
1/p=1/q+1/r $. Then, we have%
\begin{equation}  
\Vert \nabla^k( uv) \Vert _{L^p}\leq C( \Vert u\Vert _{L^q}\Vert \nabla^k
v\Vert _{L^r}+\Vert v\Vert _{L^q}\Vert \nabla^ku\Vert _{L^r}) ,\quad k\geq 0,
\label{First_inequaliy_Guass}
\end{equation}
and the commutator estimate
\begin{eqnarray}
\Vert [ \nabla^{k},f] g\Vert _{L^p}&=&\Vert \nabla^k(fg)-f\nabla^k g\Vert_{L^p}\notag\\
& \leq& C( \Vert \nabla f\Vert _{L^q}\Vert
\nabla^{k-1}g\Vert _{L^r}+\Vert g\Vert _{L^q}\Vert \nabla^{k}f\Vert _{L^r})
,\quad k\geq 1,  \label{Second_inequality_Gauss}
\end{eqnarray}
 for some constant $C>0$.
\end{lemma}
\begin{lemma}[The Gagliardo--Nirenberg interpolation inequality; See~\cite{Ner59}] \label{interpolation_lemma}  Let $1\leq
p,\,q\,,r\leq \infty $, and let $m$ be a positive integer. Then for any
integer $j$ with $0\leq j< m$, we have%
\begin{equation}
\left\Vert \nabla ^{j}u\right\Vert _{L^{p}}\leq C\left\Vert
\nabla^{m}u\right\Vert _{L^{r}}^{\alpha}\left\Vert u\right\Vert
_{L^{q}}^{1-\alpha}  \label{Interpolation_inequality}
\end{equation}%
where
\begin{equation}
\frac{1}{p}=\frac{j}{n}+\alpha\left( \frac{1}{r}-\frac{m}{n}\right) + \frac{%
1-\alpha}{q}
\end{equation}%
for $\alpha$ satisfying $j/m\leq \alpha \leq 1$ and $C$ is a positive
constant depending only on $n,\, m,\,j,\,q,\,r$ and $\alpha$.
 There
are the following exceptional cases:
\begin{enumerate}
\item If $j=0,\,rm<n$ and $q=\infty $, then we made the additional
assumption that either $u(x)\rightarrow 0$ as $|x|\rightarrow \infty $ or $%
u\in L^{q^{\prime }}$ for some $0<q^{\prime }<\infty .$

\item If $1<r<\infty $ and $m-j-n/r$ is a nonnegative integer, then (\ref{Interpolation_inequality}) holds
only for $j/m\leq \alpha< 1$.
\end{enumerate}   
\end{lemma}

\begin{lemma}[Endpoint Sobolev embedding; see~\cite{bahouri2011fourier}] \label{Lemma:EndpointEmbedding}
Assume that $2<p<\infty$. Then, there exists a constant $C=C(n,p)$ such that if $f\in \dot{H}^s(\R^n)$ with $s=n(1/2-1/p)$, then $f\in L^p(\R^n)$ and 
\begin{eqnarray}\label{Emedding_Sobolev}
\Vert f\Vert_{L^p}\leq C\Vert f\Vert_{\dot{H}^s}. 
\end{eqnarray}    
In particular, we have
 \begin{equation} 
\|\psi\|_{L^{\frac{2n}{n-2}}} \lesssim \|\nabla \psi\|_{L^2},
\end{equation}
and
\begin{equation} 
\|\psi\|_{L^{n}} \lesssim \|\psi\|_{\dot{H}^{\frac{n-2}{2}}}.
\end{equation}
\end{lemma}
\begin{lemma}[See Lemma 3.5 in~\cite{PellSaid_2019_1}] \label{Lemma:Ineq}
   	Let $n\geq 1$ and $t\geq 0 $. Then the following estimate holds:
   	\begin{equation}\label{eq of DR}
   	\int_{0}^{1}r^{n-1}e^{-r^{2}t}\textup{d}r \leq C(n)(1+t)^{-{n}/{2}}.
\end{equation}    
\end{lemma}
\noindent We state here one more useful inequality that will be crucial in our energy arguments.
\begin{lemma}[See Lemma 3.7 in~\cite{strauss1968decay}] 
\label{Lemma_Stauss} Let $M=M(t)$ be a non-negative continuous function
satisfying the inequality
\begin{equation}
M(t)\leq C_1+C_2 M(t)^{\kappa},
\end{equation}
in some interval containing $0$, where $C_1$ and $C_2$ are positive
constants and $\kappa>1$. If $M(0)\leq C_1$ and
\begin{equation}   \label{Condition_C_1_C_2}
C_1C_2^{1/(\kappa-1)}<(1-1/\kappa)\kappa^{-1/(\kappa-1)},
\end{equation}
then in the same interval
\begin{equation}
M(t)<\frac{C_1}{1-1/\kappa}.
\end{equation}
\end{lemma}
\section{Proof of Proposition~\ref{Decay_w_New}}\label{Appendix_B}
\noindent We present here the proof of Proposition~~\ref{Decay_w_New} on the decay in $w$ for the linear equation.
\renewcommand*{\proofname}{Proof of Proposition~\ref{Decay_w_New}}
\begin{proof}
	The proof is similar in most parts to the one in  \cite[Proposition 7.1]{nikolic2020mathematical} and follows by employing energy arguments in the Fourier space. We note that the following estimate holds:
	\begin{equation}\label{w_Energy_Fourier}
	\begin{aligned}
	\frac{1}{2}\frac{\textup{d}}{\dt}\tau \left\vert \hat{w}\right\vert
	^{2}+ \frac{1}{2}|\hat{w}|^2 \lesssim \, |\xi|^2 \hat{E}_1(\xi,t);
	\end{aligned}
	\end{equation}
	see~\cite[Proposition 7.1]{nikolic2020mathematical}. Above, we have introduced
	\begin{align} \label{energy}
	\hat{E}_1(\xi,t)=& \dfrac{1}{2}\left[\vphantom{\int_{0}^{\infty}} c^2_g|\xi|^{2}|\hat{\psi}+\tau\hat{v}|^{2}+\tau(b-\tau c^2_g)|\xi|^{2}|\hat{v}|^{2}+|\hat{v}+\tau\hat{w}|^{2}+\tau\Vert\hat{\eta}\Vert_{L^2, -g'}^{2}\right.  \nonumber\\
	&\left. +|\xi|^{2} \int_{0}^{\infty}g(r)|\hat{\eta}(s)|^{2}\ds+2\tau|\xi|^{2}{\Re}\left(\int_{0}^{\infty}g(r)\left\langle \hat{\eta}(s),\bar{\hat{v}}\right\rangle \ds\right)\right],
	\end{align}
	where we denote the variable dual to $x$ by $\xi$ in the Fourier transform. Together with taking into account the estimate
	\begin{equation}\label{Eexp}
	\hat{E}_1(\xi,t)\lesssim \hat{E}_1(\xi,0)\exp{\left(-\lambda \tfrac{|\xi|^2}{(1+|\xi|^2)^2}t \right)}, \qquad t \geq 0
	\end{equation} 
	we obtain 
	\begin{equation}\label{w_Fourier_Estimate}
	\left\vert \hat{w}\right\vert
	^{2}\leq |\hat{w}_0|^2 \exp{(-\tfrac{1}{\tau} t)}+C|\xi|^2 \hat{E}_1(\xi,0)\exp{(-\lambda \tfrac{|\xi|^2}{(1+|\xi|^2)^2})t)},
	\end{equation}  
	provided that the thermal relaxation is small enough; see~\cite[\S 7.1]{nikolic2020jordan} for a similar approach. We define 
	\begin{equation}
	\begin{aligned}
	|\hat{\vecc{\Psi}}(\xi,t)|^2_{\mathrm{H}}=&\,|\xi|^{2}|\alpha\hat{\psi}+\tau\hat{v}|^{2}+|\alpha\hat{v}+\tau\hat{w}|^{2}+|\xi|^{2}|\hat{v}|^{2}+ \Vert\hat{\eta}\Vert_{\mathcal{M}}^{2}\\
	=&\, |\hat{\vecc{U}}(\xi,t)|^2+\Vert\hat{\eta}\Vert_{\mathcal{M}}^{2},
	\end{aligned}
	\end{equation}  
	where $\|\hat{\eta}\|_{\mathcal{M}}=|\xi|^{2} \int_{0}^{\infty}g(r)|\hat{\eta}(s)|^{2}\ds$. Hence, it holds that   
	\begin{equation} \label{ineq_U_E_1}
	\hat{E}_1(\xi,t) \lesssim |\hat{\vecc{\Psi}}(\xi,t)|^2, \quad t > 0, \quad \text{and}\quad \hat{E}_1(\xi,0) \lesssim |\hat{\vecc{U}}(\xi,0)|^2
	\end{equation}
	In fact we also have (see \cite{Bounadja_Said_2019}),   
	\begin{equation}  
	\hat{E}_1(\xi,t)\gtrsim
	|\hat{\vecc{\Psi}}(\xi,t)|^2_{\mathrm{H}}.
	\end{equation}
	So, combining the above estimates, we have 
	\begin{equation}
	\begin{aligned}
	|\hat{\vecc{U}}(\xi,t)|^2\lesssim |\hat{\vecc{\Psi}}(\xi,t)|_{\mathrm{H}}^2\lesssim&\, \hat{E}_1(\xi,t)\\
	\lesssim&\, \hat{E}_1(\xi,0)\exp{\left(-\lambda \tfrac{|\xi|^2}{(1+|\xi|^2)^2}t\right)}\\
	\lesssim&\, |\hat{\vecc{U}}(\xi,0)|^2 \exp{\left(-\lambda \tfrac{|\xi|^2}{(1+|\xi|^2)^2}t\right)};
	\end{aligned}
	\end{equation}
	see \eqref{Eexp}. By applying  Plancherel's theorem, the  above estimate yields  \eqref{decay_loss}; see also~\cite{Bounadja_Said_2019} for more details. Applying Plancherel's theorem together with  \eqref{ineq_U_E_1} at $t=0$ yields
	\begin{equation}\label{Plancherel}
	\begin{aligned}
	\Vert\nabla^{j}w(t)\Vert_{L^{2}}^{2}=& \, \int_{\R^{n}}|\xi|^{2j}|\hat{w}(\xi,t)|^{2}\, \textup{d}\xi\\
	\lesssim& \,\Vert \nabla^j \psi_2 \Vert_{L^2}^2\exp{\left(-\tfrac{1}{\tau} t\right)}+ \int_{\R^{n}}|\xi|^{2(j+1)} \exp{\left(-\lambda \tfrac{|\xi|^2}{(1+|\xi|^2)^2}t\right)}|\hat{\mathbf{U}}(\xi,0)|^{2}\, \textup{d}\xi     
	\end{aligned}	
	\end{equation}
	for any integer $j\geq 0$.  The second term on the right-hand side can be split into
	\begin{equation} \label{split Intg} 
	\begin{aligned}
	&\int_{\R^{n}}|\xi|^{2(j+1)} \exp{\left(-\lambda \tfrac{|\xi|^2}{(1+|\xi|^2)^2}t\right)}|\hat{\mathbf{U}}(\xi,0)|^{2}\textup{d} \xi\\
	=& \, \begin{multlined}[t]\int_{|\xi|\leq 1}|\xi|^{2(j+1)} \exp{\left(-\lambda \tfrac{|\xi|^2}{(1+|\xi|^2)^2}t\right)}|\hat{\mathbf{U}}(\xi,0)|^{2}\textup{d} \xi \\  
	+ \int_{|\xi|\geq 1}|\xi|^{2(j+1)} \exp{\left(-\lambda \tfrac{|\xi|^2}{(1+|\xi|^2)^2}t\right)}|\hat{\mathbf{U}}(\xi,0)|^{2}\textup{d}\xi.\end{multlined}
	\end{aligned}
	\end{equation}   
	We note that
	\begin{equation} \label{rho*}
	\frac{|\xi|^2}{1+|\xi|^2} \gtrsim 
	\begin{cases}  
	|\xi|^{2}, & \text{if }\quad |\xi|\leq 1, \vspace{0.2cm}\\ 
	|\xi|^{-2}, & \text{if} \quad|\xi|\geq 1.
	\end{cases}  
	\end{equation}  	
	Concerning the first integral on the right in \eqref{split Intg}, by exploiting the inequality \[\int_{0}^{1}r^{n-1}e^{-r^{2}t}\textup{d}r \leq C(n)(1+t)^{-{n}/{2}},\]
	given in Lemma~\ref{Lemma:Ineq} together with \eqref{rho*}, we find that
	\begin{equation}\label{I1}
	\begin{aligned}
	\int_{|\xi|\leq 1}|\xi|^{2(j+1)} \exp{\left(-\lambda \tfrac{|\xi|^2}{1+|\xi|^2}t\right)}|\hat{\mathbf{U}}(\xi,0)|^{2}\textup{d} \xi\leq& \, \Vert\hat{\mathbf{U}}_{0}\Vert_{L^{\infty}}^{2}\int_{|\xi|\leq 1}|\xi|^{2(j+1)} \exp{\left(-\tfrac{\lambda}{2} \tfrac{|\xi|^2}{1+|\xi|^2}t\right)}\, \textup{d}\xi\\
	\lesssim& \, (1+t)^{-\frac{n}{2}-1-j}\Vert \mathbf{U}_{0}\Vert_{L^{1}}^{2}.
	\end{aligned}
	\end{equation}
	On the other hand, in the high-frequency region where $|\xi|\geq 1$, we have
	by using the estimate 
	\begin{equation}
	\sup_{|\xi|\geq 1}\left\lbrace |\xi|^{-2\ell}e^{-c|\xi|^{-2}t}\right\rbrace \lesssim (1+t)^{-\ell}
	\end{equation}  
	that	
	\begin{equation}\label{I2}
	\begin{aligned}
	&\int_{|\xi|\geq 1}|\xi|^{2(j+1)} \exp{\left(-\lambda \tfrac{|\xi|^2}{(1+|\xi|^2)^2}t\right)}|\hat{\mathbf{U}}(\xi,0)|^{2}\textup{d}\xi\\
	\leq&\, \sup_{|\xi|\geq 1}\left\lbrace |\xi|^{-2\ell}e^{-c|\xi|^{-2}t}\right\rbrace\int_{|\xi|\geq 1}|\xi|^{2(j+\ell+1)} |\hat{\mathbf{U}}(\xi,0)|^{2}d\xi \\
	\leq& \, (1+t)^{-\ell}\Vert\nabla^{j+\ell+1}\mathbf{U}_{0}\Vert_{L^{2}}^{2},
	\end{aligned}	
	\end{equation} 
	which completes the proof.	
\end{proof}
\section{Proof of Proposition~\ref{Lemma_decay_v_infty}} \label{AppendixC}
\renewcommand*{\proofname}{Proof of Proposition~\ref{Lemma_decay_v_infty}}
\begin{proof}    
	The decay estimate for $\Vert \nabla^j v\Vert_{L^2}$ follows by combining the derived bounds \eqref{decay_loss} and \eqref{Decay_estimate_W}. For the second estimate, we use the Gagliardo--Nirenberg interpolation inequality \eqref{L_infty_Interp} for an integer $q\geq n/2$. We combine it with particular cases of estimate \eqref{v_L_2_Estimate}:
	\begin{equation}
	\begin{aligned}
	\Vert v(t)\Vert_{L^{2}}\lesssim&\, (\Vert \psi_2 \Vert_{L^2}+\Vert \mathbf{U}_{0}\Vert_{L^{1}})(1+t)^{-\frac{n}{4}}+(1+t)^{-\ell/2} \Vert\nabla^{\ell}\vecc{U}_{0}\Vert _{H^{1}} \\
	\lesssim&\, (\Vert \psi_2 \Vert_{L^2}+\Vert \mathbf{U}_{0}\Vert_{L^{1}}+\Vert\nabla^{\ell}\vecc{U}_{0}\Vert _{H^{1}})(1+t)^{-\frac{n}{4}}
	\end{aligned}  
	\end{equation}
	as well as
	\begin{equation}
	\begin{aligned}
	\Vert\nabla^{q}v(t)\Vert_{L^{2}}\lesssim&\, (\Vert \nabla^q \psi_2 \Vert_{L^2}+\Vert \mathbf{U}_{0}\Vert_{L^{1}})(1+t)^{-\frac{n}{4}-\frac{q}{2}}+(1+t)^{-\ell/2} \Vert\nabla^{q+\ell}\vecc{U}_{0}\Vert _{H^{1}}
	\\  
	\lesssim&\, (\Vert \nabla^q \psi_2\Vert_{L^2}+\Vert \mathbf{U}_{0}\Vert_{L^{1}}+\Vert\nabla^{q+\ell}\vecc{U}_{0}\Vert _{H^{1}})(1+t)^{-\frac{n}{4}-\frac{q}{2}},
	\end{aligned}
	\end{equation}
	which hold provided that $\ell/2 \geq n/4 +q/2$. We thus obtain
	\begin{equation}
	\begin{aligned}   
	\Vert v(t) \Vert_{L^\infty}\lesssim&\, \begin{multlined}[t] (\Vert \nabla^q \psi_2 \Vert_{L^2}+\Vert \mathbf{U}_{0}\Vert_{L^{1}}+\Vert\nabla^{q+\ell}\vecc{U}_{0}\Vert _{H^{1}})^{\frac{n}{2q}} \vspace{0.2cm}\\
	\times (\Vert \psi_2 \Vert_{L^2}+\Vert \mathbf{U}_{0}\Vert_{L^{1}}+\Vert\nabla^{\ell}\vecc{U}_{0}\Vert _{H^{1}})^{(1-\frac{n}{2q})}(1+t)^{-n/2} \end{multlined} \vspace{0.2cm}\\
	\lesssim&\, \begin{multlined}[t] (\Vert \psi_2 \Vert_{L^2}+\Vert \nabla^q \psi_2 \Vert_{L^2}+\Vert \mathbf{U}_{0}\Vert_{L^{1}}+\Vert\nabla^{\ell}\vecc{U}_{0}\Vert _{H^{1}}+\Vert\nabla^{q+\ell}\vecc{U}_{0}\Vert _{H^{1}(\R^{n})})(1+t)^{-n/2} \end{multlined} \vspace{0.2cm}\\
	\lesssim&\, \begin{multlined}[t] (\Vert \psi_2 \Vert_{H^q}+\Vert \mathbf{U}_{0}\Vert_{L^{1}}+\Vert \vecc{U}_{0}\Vert _{H^{q+\ell+1}})(1+t)^{-n/2}. \end{multlined} 
	\end{aligned}
	\end{equation}
	Finally, we take $q=[n/2]+1$ and $\ell=n+1$ in the above estimate to arrive at \eqref{v_L_infty_Estimate}.
\end{proof}
Note that in order to estimate $\Vert \nabla \vecc{U}\Vert_{L^\infty}$, we can use the interpolation inequality: 
\begin{equation}\label{L_infty_Interp_2}
\left\Vert \nabla \vecc{U}\right\Vert _{L^{\infty }}\lesssim \left\Vert \nabla ^{q}%
\vecc{U}\right\Vert _{L^{2}}^{\frac{2+n}{2q}}\left\Vert \vecc{U}
\right\Vert _{L^{2}}^{1-\frac{2+n}{2q}},  
\end{equation}  
for $q> \frac{n}{2}+1 $. We can repeat the previous arguments with $q=[n/2]+2$ and $\ell=n+2$, assuming $\vecc{U}_0\in \left(L^1(\R^n) \cap H^{n+[n/2]+4}(\R^n) \right)^3$.\\
\end{appendices}

\bibliography{references}{}
\bibliographystyle{siam} 
\end{document}